\documentclass{amsart}

\usepackage{latexsym,amsfonts,amssymb,exscale,enumerate}
\usepackage{amsmath,amsthm,amscd}
\usepackage{hyperref}

\usepackage{pstricks}

\psset{linewidth=0.3pt,dimen=middle}
\psset{xunit=.70cm,yunit=0.70cm}
\psset{arrowsize=1pt 5,arrowlength=0.6,arrowinset=0.7}

\input xy
\usepackage[all,dvips]{xy}
\xyoption{line}
\xyoption{arrow}
\xyoption{color}
\SelectTips{cm}{}

\usepackage{graphicx}
\usepackage{color}
\newcommand{\ep}{\underline{\epsilon}}
\newcommand{\onen}{{\mathbf 1}_{n}}
\newcommand{\one}{{\mathbf 1}_{-n}}
\newcommand{\onenn}[1]{{\mathbf 1}_{#1}}

\newcommand{\onem}{{\mathbf 1}_{m}}
\newcommand\rE{{\sf{E}}}
\newcommand{\F}{\cal{F}}
\newcommand{\rib}{{r}}
\newcommand\rF{{\sf{F}}}

\newcommand{\maps}{\colon}
\newcommand{\co}{{\rm co}}

\newcommand{\Int}{{(\rib\E\onen)'}}

\newcommand{\Kom}{{\rm Kom}}
\newcommand{\Com}{{\rm Com}}
\newcommand{\Gr}{\cat{Flag}_{N}}
\newcommand{\Grn}[1]{\cat{Flag}_{#1}}
\newtheorem{theorem}{Theorem}

\newtheorem{conjecture}[theorem]{Conjecture}

\newtheorem{lemma}[theorem]{Lemma}
\newcommand{\U}{\dot{{\bf U}}}
\newcommand{\Ucat}{\cal{U}}

\newcommand{\UcatD}{\dot{\cal{U}}}

\newcommand{\refequal}[1]{\xy {\ar@{=}^{#1}
(-1,0)*{};(1,0)*{}};
\endxy}
\newcommand{\cat}[1]{\ensuremath{\mbox{\bfseries {\upshape {#1}}}}}

\newcommand{\Hom}{{\rm Hom}}

\renewcommand{\to}{\rightarrow}

\def\shuffle{\,\raise 1pt\hbox{$\scriptscriptstyle\cup{\mskip
               -4mu}\cup$}\,}

\theoremstyle{definition}
\newtheorem{thm}{Theorem}[section]

\newtheorem{lem}[thm]{Lemma}

\newtheorem{prop}[thm]{Proposition}

\numberwithin{equation}{section}

\def\emph#1{{\sl #1\/}}

\let\hat=\widehat
\let\tilde=\widetilde

\let\phi=\varphi
\let\theta=\vartheta
\let\epsilon=\varepsilon

\usepackage{bbm}
\def\C{{\mathbbm C}}
\def\N{{\mathbbm N}}

\def\Z{{\mathbbm Z}}

\def\cal#1{\mathcal{#1}}%
\def\1{\mathbbm{1}}%
\def\Dot{\mathrm{Dot}}
\def\nn{\notag}

\def\la{\langle}
\def\ra{\rangle}

\newcommand{\Ucupr}{\;\;
    \vcenter{\xy (-2,3)*{}; (2,3)*{} **\crv{(-2,-1) & (2,-1)}?(1)*\dir{>};
            (2,-3)*{};(-2,3)*{}; \endxy} \;\; }

\newcommand{\Ucross}{\;\;
    \vcenter{\xy {\ar (2.5,-2.5)*{};(-2.5,2.5)*{}}; {\ar (-2.5,-2.5)*{};(2.5,2.5)*{}};
    (4,0)*{};(-4,0)*{};\endxy} \;\; }

\newcommand{\lowrru}[1]{\xybox{%
  (-8,0)*{};
  (8,0)*{};
  (-6,-18)*{};(6,-9)*{} **\crv{(-6,-13) & (6,-15)} ?(1)*\dir{>};
  (6,-9)*{};(6,0)*{}  **\dir{-} ?(.3)*\dir{ }+(2,0)*{\scs {\bf j}};
}}

\newcommand{\lowllu}[1]{\xybox{%
  (-8,0)*{};
  (8,0)*{};
  (6,-18)*{};(-6,-9)*{} **\crv{(6,-13) & (-6,-15)} ?(1)*\dir{>};
  (-6,-9)*{};(-6,0)*{}  **\dir{-} ?(.3)*\dir{ }+(-2,0)*{\scs {\bf j}};
}}

\newcommand{\bbdl}[1]{\xybox{%
  (2,0);(0,-8) **\crv{(2,-2)&(0,-6)}; ?(.5)*\dir{>}
}}
\newcommand{\bbdlu}[1]{\xybox{%
  (2,0);(0,-8) **\crv{(2,-2)&(0,-6)}; ?(.5)*\dir{<}
}}
\newcommand{\bbdr}[1]{\xybox{%
  (-2,0);(0,-8) **\crv{(-2,-2)&(0,-6)}; ?(.5)*\dir{>}
}}
\newcommand{\bbdru}[1]{\xybox{%
  (-2,0);(0,-8) **\crv{(-2,-2)&(0,-6)}; ?(.5)*\dir{<}
}}

\newgray{whitegray}{.9}

\DeclareGraphicsExtensions{.pdf,.eps}
\usepackage{psfrag} 
\usepackage{bbm}

\def\cal#1{\mathcal{#1}}

\def \Z {\mathbbm{Z}}

\def \N {\mathbbm{N}}

\def \E {\mathcal{E}}
\def \F {\mathcal{F}}
\def \U {\mathcal{U}}

\def \C {\mathcal{C}}

\def \Span{\operatorname{Span}}

\newcommand{\xto}[1]{{\overset{#1}{\longrightarrow}}}

\newcommand{\bigb}[1]{
\begin{pspicture}(0,0)
 \rput(0,0){\psframebox[framearc=.5,fillstyle=solid]{\small $#1$}}
\end{pspicture}}

\newcommand\nc{\newcommand}
\nc\rnc{\renewcommand}

\nc\block[2]{\begin{#1}#2\end{#1}}
\nc\comm[1]{\ {\tt[* #1 *]}\ }
\nc\note[1]{{\small{{\tt [note:}\ {#1} {]}}}}

\nc\Kar{\operatorname{Kar}}
\nc\sgn{\operatorname{sgn}}
\nc\Center{\operatorname{Center}}
\nc\ad{{\operatorname{ad}}}
\nc\op{{\operatorname{op}}}
\nc\End{\operatorname{End}}
\nc\Spec{\operatorname{Spec}}
\nc\even{{\operatorname{even}}}
\nc\odd{{\operatorname{odd}}}
\nc\Gal{\operatorname{Gal}}
\nc\Aut{\operatorname{Aut}}
\nc\half{{\frac12}}
\nc\halfof[1]{{\frac{#1}2}}
\nc\bb[2]{\biggl[\begin{matrix}{#1}\\{#2}\end{matrix}\biggr]}
\nc\BB[2]{\{#1\}_{#2}}
\nc\onto\twoheadrightarrow
\nc\projto{\underset{\text{proj}}{\longrightarrow}}
\nc\no[1]{}
\nc\ok{\comm{ok?}}
\nc\ho{{\hat\otimes }}
\nc\zzzvert {\;|\;}
\nc\zzzcolon {\colon\thinspace}
\nc\zzzsharp {\sharp}
\nc\plim{\varprojlim}
\nc\np{\newpage}
\nc\modone {{\mathbf 1}}
\nc\g{{\mathfrak g}}
\nc\modk {{\mathbf k}}
\nc\modL {{\mathcal L}}
\nc\modN {{\mathbb N}}
\nc\modQ {{\mathbb Q}}
\nc\modR {{\mathcal R}}
\nc\modZ {{\mathbb Z}}
\nc\ul{\underline}

\nc\simeqto{\overset{\simeq}{\longrightarrow }}
\nc\trl{\triangleleft}
\nc\trr{\triangleright}

\nc\fig[1]{Figure~\ref{#1}}
\nc\FI[2]{\begin{figure}
    \begin{center}\input{#1.pstex_t}\end{center}
    \caption{#2}
    \label{#1}
  \end{figure}}
\nc\FIGURE[2]{\begin{figure}
    \boxed{\tt fig: #1}
    \caption{#2}
    \label{fig:#1}
    \end{figure}}

\nc\xysquare[8]{\xymatrix{
    #1 \ar[r]#5 \ar[d]#6 & #2 \ar[d]#7 \\
    #3 \ar[r]#8          & #4
  }
  }

\nc\modC {{\mathcal C}}
\nc\modD {{\mathcal D}}
\nc\modE {{\mathcal E}}
\nc\modV {{\mathcal V}}
\nc\Vect{\mathbf{Vect}}
\nc\Spaces{\mathbf{Spaces}}
\nc\Mod{\mathbf{Mod}}
\nc\Mor{\operatorname{Mor}}

\nc\K{\mathcal {K}}

\nc\CC{\mathbf{C}}

\nc\congto{\xto{\cong}}

\newcommand{\scs}{\scriptstyle}
\newcommand{\lbub}{\xybox{
  (-3,0)*{};(3,0)*{} **\crv{(-3,4) & (3,4)} ?(.0)*\dir{>};
  (3,0)*{};(-3,0)*{} **\crv{(3,-4) & (-3,-4)} ?(.1)*{\ast};
  (-5,-5)*{}; (5,5)*{};
}}
\newcommand{\rbub}{\xybox{
  (-3,0)*{};(3,0)*{} **\crv{(-3,4) & (3,4)} ?(.0)*\dir{<};
  (3,0)*{};(-3,0)*{} **\crv{(3,-4) & (-3,-4)} ?(.1)*{\ast};
  (-5,-5)*{}; (5,5)*{};
}}

\usepackage{diagcat-0.6}
\newstrandstyle{black}{type=string, width=1pt,
style=solid, arrow=xy}

\newcouponstyle{solid}{%
      linecolor=black,%
      linewidth=0.5pt,%
      linestyle=solid,%
      fillcolor=white,%
      fillstyle=solid,%
      linearc=0.1} 
\newstrandstyle{d}{type=string,width=1pt,arrow=xy} 
\newstrandstyle{black}{type=string,color=black,style=solid,width=.5pt,arrow=xy}
\nc\nempty{\neq\emptyset}
\newcommand\rr{{\sf{r}}}

\nc\diag{\operatorname{diag}}

\input xy

\usepackage{color}
\begin{document}
\title{A categorification of the ribbon element in quantum $sl(2)$}

\date{\today}

\author{Anna Beliakova}
\address{Universit\"at Z\"urich, Winterthurerstr. 190
CH-8057 Z\"urich, Switzerland}
\email{anna@math.uzh.ch}

\author{Kazuo Habiro}
\address{Research Institute for Mathematical Sciences, Kyoto
  University, Kyoto 606-8502, Japan}
\email{habiro@kurims.kyoto-u.ac.jp}
\begin{abstract}
We define a bicomplex whose Euler characteristic is the 
idempotented version of the ribbon element of
quantum sl(2).
We show that  properties of this bicomplex descend to the
centrality, invertibility and symmetries of the ribbon
element after decategorification.

%(We show that this bicomplex belongs to
%the Drinfeld center of the additive monoidal category
%$\Com(\UcatD)$, where $\UcatD$ is the 2-category
%defined by A. Lauda. ?)
\end{abstract}

\maketitle
\section{Introduction}
The program of  categorification of quantum groups was initiated
by  Frenkel and carried out by Khovanov and Lauda \cite{KL1}, \cite{KL2},
\cite{KL3} and 
Rouquier \cite{Rou2}.
In \cite{Lau1}, \cite{KLMS} the 2-category $\UcatD$ was constructed
that categorifies the integral idempotented version
of the quantum enveloping algebra of $\mathfrak {sl}_2$.

The objects of $\UcatD$  are natural numbers, interpreted as the
integral weight lattice of $\mathfrak {sl}_2$.
 The 1-morphisms of $\UcatD$ are generated by 
 $\E^{(a)}\onen\la t\ra $ and $\F^{(b)}\onen\la t\ra$,
for all $a,b\in \N$, $n,t\in \Z$, which are lifts
 of the Lusztig divided powers. 
 The 2-morphisms are  $\Z$-linear 
combinations of planar diagrams modulo local relations. The split
Grothendieck group $K_0(\UcatD)$ satisfies
$$K_0(\E^{(a)}\onen\la t\ra)= q^t E^{(a)}1_n
 \quad {\text {and}}\quad 
K_0(\F^{(a)}\onen\la t\ra)=q^t F^{(a)}1_n  $$
and  coincides with the integral
idempotented version of the quantum enveloping algebra 
 ${\bf U}_q (\mathfrak {sl}_2)$ constructed in \cite{BLM}.

On the other hand, there exists
a universal knot invariant \cite{Law} (see also \cite{Ha})
 which takes values 
in the center of quantum $\mathfrak {sl}_2$ and dominates all colored
Jones polynomials. This paper can be considered as a first 
step towards a categorification of the universal link invariant.

The most challenging problem of this program
 is to find a categorical equivalent of
the $R$-matrix, associated by the universal invariant
 to a crossing in a link diagram.
Here we resolve a  simpler problem, we categorify
the ribbon element, which is the universal invariant 
of a self-crossing. The bicomplex we construct in this paper
is conjecturally  an element of the Drinfeld center of $\Com(\UcatD)$ and 
is related to 
a Serre functor on category $\mathcal O$ 
  associated with the longest braid \cite{MS} by Conjecture 3.
Here we denote by $\Kom (\UcatD)$  the category 
of bicomplexes over the 2-category $\UcatD$ and
by  $\Com(\UcatD)$ its homotopy
version.

\nc\rmr{\mathrm{r}}
\nc\sltwo{\mathfrak{sl}_2}

The ribbon element $\rmr$ is an element of the $h$-adic version
$\mathbf{U}_h(\mathfrak{sl}_2)$ of the quantized enveloping algebra $\mathbf{U}_q(\mathfrak{sl}_2)$.
It is given by
\begin{gather*}
{\rm r}= q^{-\frac{H^2}{2}-H} \sum^\infty_{k=0} (-1)^k
q^{-kH-k}(q^{-2};q^{-2})_k F^{(k)} E^{(k)}
\end{gather*}
where $(q^a;q^b)_k=(1-q^a)(1-q^{a+b})\dots (1-q^{a+(k-1)b})$.
Thus, the idempotented version of $\rmr$ is
\begin{align}
  \label{e1}
{\rm r}1_n &= q^{-\frac{n^2}{2}-n} \sum^\infty_{k=0} (-1)^k
q^{-kn-k}(q^{-2};q^{-2})_k F^{(k)} E^{(k)}1_n\\
&= q^{-\frac{n^2}{2}+n} \sum^\infty_{k=0} (-1)^k
q^{kn-k}(q^{-2};q^{-2})_k E^{(k)} F^{(k)}1_n,
\end{align}
where $n\in\mathbb{Z}$.

Recall that the ribbon element is central and invertible in  ${\bf U}_q (\mathfrak {sl}_2)$ with the inverse
\begin{align*}
{\rm r}^{-1}1_n &= q^{\frac{n^2}{2}+n} \sum^\infty_{k=0} (-1)^k
q^{kn+k}(q^{2};q^{2})_k F^{(k)} E^{(k)}1_n
\end{align*}

The definition of the bicomplex $\rib\onen\in \Kom(\UcatD)$
categorifying the ribbon element is outlined as follows.  First, we
construct a bicomplex $C_{\bullet,\bullet}$ in $\Kom(\UcatD)$
\begin{equation}\label{bicomplex}
 \xymatrix{C_{00}\ar[d]^{d^V_{00}}
	     &	& &	  \\
      C_{10}\ar[r]^{d^H_{10}}\ar[d]^{d^V_{10}}&C_{11}\ar[d]^{d^V_{11}}&&\\
		 C_{20}\ar[r]^{d^H_{20}}\ar[d]^{d^V_{20}}&
C_{21}\ar[r]^{d^H_{21}}\ar[d]^{d^V_{21}}&
C_{22}\ar[d]^{d^V_{22}}&\\
C_{30}\ar[r]^{d^H_{30}}\ar[d]^{d^V_{30}} &
C_{31}\ar[r]^{d^H_{31}}\ar[d]^{d^V_{31}}
 & C_{32}\ar[r]^{d^H_{32}}\ar[d]^{d^V_{32}}&
C_{33}\ar[d]^{d^V_{33}}\\
\dots &\dots &\dots & \dots}
       \end{equation}
where the $1$-morphism $C_{k,l}\colon n\to n$ in $\UcatD$ (i.e. an
object $C_{k,l}$ in the additive category $\UcatD(n,n)$) is defined by
\begin{gather*}
  C_{k,l}:= \F^{(k)} \E^{(k)}\onen\la -kn-k\ra\otimes \Lambda^l W_k,
\end{gather*}
where $W_k=\Span_\Z\{w_1,\dots, w_k\}$
with $\deg(w_j)=-2j$.  
Note that $C_{00}=\F^{(0)}\E^{(0)}\onen\la 0\ra\otimes \Lambda^0
W_0\cong \onen$.  
The differentials 
$d^V_{k,l}$, $d^H_{k,l}$ are defined 
in Section \ref{rib-def}.  Figure 1 there shows the beginning of this
bicomplex.
The Euler characteristic of the bicomplex $C_{\bullet,\bullet}$ is
equal to the sum part of \eqref{e1}
\begin{gather*}
   \sum^\infty_{k=0} (-1)^k q^{-kn-k}(q^{-2};q^{-2})_k F^{(k)} E^{(k)}1_n.
\end{gather*}

Finally, we shift our bicomplex by $\la -\frac{n^2}{2}-n\ra$ in $q$-degree and by
$[n/2,n/2]$ in homological bi-degree to obtain
\begin{gather*}
  r \onen := C_{\bullet,\bullet}\la -\frac{n^2}{2}-n\ra [n/2,n/2].
\end{gather*}
The $q$-degree shift $\la -\frac{n^2}{2}-n\ra$ corresponds to the factor
$q^{-\frac{n^2}{2}-n}$ in \eqref{e1}.
The homological degree shift $[n/2,n/2]$ is added to make $r\onen$
commute with the $1$-morphisms in $\UcatD$ up to homotopy equivalence
(see Theorem \ref{main}). We  use the convention that the 
Euler characteristic is not affected by a global homological shift.

The bicomplex  $r^{-1}\onen$ is given by inverting
all arrows in \eqref{bicomplex}, rotating 
 diagrams representing differentials  by $180$ degree,
and
 replacing $C_{k,l}$
 with\footnote{The index $L$ stays here for the left adjoint.}

 $$C^L_{k,l}:= \F^{(k)} \E^{(k)}\onen\la kn+k\ra\otimes \Lambda^l {\bar W}_k,\quad
{\bar W}_k=\Span_\Z\{\bar w_1,\dots, \bar w_k\} \,  $$
where ${\rm deg}(\bar w_j)=2j$. 
Finally, we have
\begin{gather*}
  r^{-1} \onen := C^L_{\bullet,\bullet}\la \frac{n^2}{2}+n\ra [-n/2,-n/2].
\end{gather*}
It is an easy check that the Euler characteristics   of
 $\rib\onen$ and $r^{-1}\onen$ in
 $\Kom(\UcatD)$ coincide with ${\rm r}1_n$ and ${\rm r}^{-1}1_n$,
respectively.

Let us list some properties of these  
bicomplexes, which follow directly from the definition.
\begin{itemize}
\item
For each $k$, the horizontal complex 
$C_{k,\bullet}=\oplus_{l\in\mathbb{Z}} C_{k,l}$
 forms a $k$-dimensional
cube and hence is bounded.

\item
The total complex $\operatorname{Tot}(C_{\bullet,\bullet})$ is a
well-defined complex in $\UcatD$, i.e., for each integer $p$,
$\bigoplus_{k+l=p}C_{k,l}\in\operatorname{Ob}(\UcatD)$ is a finite direct sum.

\item
The horizontal differentials are given by certain central elements
in $\End (\F^{(k)}\E^{(k)}\onen)$ 
defined in Section \ref{center} and generated by dots.
\end{itemize}

Recall that $\UcatD$ is the Karoubi envelope 
of the 2-category $\Ucat$ defined 
in \cite{Lau1}.
The generators for 1-morphisms
in $\U$ are $\E\onen \la t\ra$ and $\F\onen \la t\ra$ for $n,t\in \Z$.
The involutive 2-functors $\omega$, $\sigma$ and $\psi$,
 generating
the symmetry group $\mathcal G=(\Z/2\Z)^3$ of 
 $\U$  were also introduced  in \cite{Lau1}.
We recall these definitions and extend them to $\UcatD$ in Section \ref{sym}.
Let us denote by $\mathcal G_1:=\{1, \sigma\omega\}$
 the subgroup of $\mathcal G$.

Throughout this paper
 we adopt the notation $r r^{-1} \onen$ for the tensor 
product of two complexes, i.e. we omit  tensor sign, since
 on  1-morphisms of $\UcatD$  tensor product is just a concatenation.

Our main result is the following:
\begin{theorem}\label{main}
 {\rm (Centrality)}
For any complex $X\in\Com(\U)$, there are isomorphisms
 \[\label{main:3}
\kappa_X: X \rib\onen \to\rib X\onen \quad \quad
\eta_X: X r^{-1}\onen \to r^{-1} X\onen
\]
 of bicomplexes  in $\Com(\UcatD)$.

\vskip1mm
{\rm (Invertibility)} 
The bicomplexes  $r
 r^{-1}\onen$ and $r^{-1} r\onen$ in $^l\Kom(\UcatD)$
are homotopy equivalent to $\onen$. 
\vskip1mm

 {\rm (Symmetry)} 
 The bicomplexes $\rib\onen$ and $r^{-1}\onen$
in $\Kom (\UcatD)$ are invariant under the action of $\mathcal G_1\subset \mathcal G$. Moreover, $r^{-1}\onen=\sigma\omega \psi (r\onen)$ 
is isomorphic to $\psi(r\onen)$.
\end{theorem}

Note that since the bicomplexes $r\onen$ and $r^{-1}\onen$ are bounded
from above and below respectively, and $\UcatD$ does not admit infinite direct 
sums, their tensor product $r r^{-1}\onen$ does not belong to $\Kom(\UcatD)$.
Instead we are using $^l\Kom(\UcatD)$ which is the inverse limit of
the categories of bounded bicomplexes
$\Kom^b(\UcatD_N)$, where $\UcatD_N$ is the  Schur quotient of $\UcatD$
defined by setting ${\bf 1}_{N+2}=0$ (see Section 12  for more details).

Properties listed in
Theorem \ref{main} are natural lifts
of the  properties  of the ribbon element to  higher
categorical level. We would also  expect
the following to hold.

%In Section \ref{old} we define another  bicomplex which is 
%isomorphic to $r\onen$. This new complex is
%used to prove the invariance of
% the ribbon bicomplex under
% some symmetry of $\UcatD$ (Theorem \ref{sym-inv-thm}).

\begin{conjecture}
{\rm (Naturality)} {The chain maps $\kappa_X$ 
and $\eta_X$ are natural, i.e. they commute with
all 2-morphisms of $\UcatD$.}
\vskip1mm

{\rm (Decomposability)} {
For $n\geq 0$ the bicomplex $\rib\onen$ is indecomposable in $\Kom(\UcatD)$,
and the bicomplex $\omega(r\one)$ is isomorphic to a direct sum
of $\rib\onen$ and a contractible complex.
For $n\leq 0$ the bicomplex $\omega(r\one)$ 
is indecomposable in $\Kom(\UcatD)$,
and the bicomplex $\rib\onen$ is isomorphic to a direct sum
of $\omega(\rib\one)$ and a contractible complex.
For $n=0$ the bicomplexes $\rib\onen$ and $\omega(r\one)$ are isomorphic.}
\end{conjecture}

Observe that we could use
$\sigma(r\one)$ instead of $\omega(r\one)$, since by Theorem \ref{main}
they are isomorphic.

Let us comment on this conjecture. 
We expect the isomorphisms $\kappa_X$ and $\eta_X$ to be natural,
 and hence, to be defined  for any $X\in \Com(\UcatD)$. This would imply
that the bicomplexes $\rib\onen$ and $r^{-1}\onen$  belong  to the
Drinfeld center of $\Com(\UcatD)$ viewed as an additive monoidal category.
Here we regard 1-morphisms in $\Com(\UcatD)$ 
as objects of the monoidal category.
 The monoidal structure is  given by composition of 1-morphisms
and horizontal composition of 2-morphisms.
The collection of chain maps $\kappa_X$ define then an 
invertible natural transformation
$\kappa: - \rib \Longrightarrow \rib -$
between endofunctors of $\Com(\UcatD)$ given by tensoring on the left and on 
the right with the complex $\rib\onen$ for an appropriate $n$.

%We also , which together with
%the main result of \cite{BL} provides a strong evidence for the 
%invertibility of the ribbon complex.

The first column of our bicomplex $C_{00}\to C_{10} \to C_{20}\to\dots $
is an example of  so-called Rickard complex introduced by
Chuang and Rouquier in \cite{CR} and intensively
studied  by Cautis and Kamnitzer \cite{CK}, \cite{C}. The
Rickard-Rouquier complex ${\bf 1}_{-n}{\sf T}\onen$
categorifies 
the action of the Weyl group on the finite-dimensional representations
and satisfies the braid relation. It can be defined as follows:
\begin{align}
{\sf T}\onen: &\dots \to \F^{(n+s)}\E^{(s)}\la s\ra \onen \to
\F^{(n+s-1)}\E^{(s-1)}\la s-1\ra \onen \to \dots
\to \F^{(n)} \onen  & \quad{\text {for}}\quad n\geq 0\\
{\sf T}\onen: &\dots \to \E^{(-n+s)}\F^{(s)}\la s\ra \onen \to
\E^{(-n+s-1)}\F^{(s-1)}\la s-1\ra \onen \to \dots
\to \E^{(-n)} \onen  & \quad{\text {for}}\quad n\leq 0
\end{align}
where the differential are non-zero maps.

\begin{conjecture}[Cautis]
The total complex of the ribbon bicomplex ${\operatorname{Tot}}(r^{-1}\onen)$
is homotopy equivalent
to 
$${\sf T}^2\onen\la \frac{n^2}{2}+n\ra [-n/2,-n/2]\, .$$ 
\end{conjecture}

Note that the decategorified version of this conjecture holds.
For $n\geq 0$,  the Euler characteristic 
\begin{align*}
T^2 1_n= T1_{-n} T1_n&=
\sum_{l,s\geq 0} (-q)^{l+s} E^{(n+s)} F^{(s)} F^{(n+l)} E^{(l)} 1_n\\
&=\sum^\infty_{k=0} (-1)^k
q^{kn+k}(q^{2};q^{2})_k F^{(k)} E^{(k)}1_n
\end{align*}
coincides with ${\rm r}^{-1}1_{n}$ after multiplying with $q^{n^2/2+n}$.
The case $n\leq 0$ can be obtained similarly, after replacing
$n$ by $-n$ and exchanging $E$'s and $F$'s. 
%In \cite{C}, Cautis proved 
%centrality and invertibility of the complex ${\sf T}^2\onen$.

\subsection{Strategy of the proof of Theorem \ref{main}}
It is enough to check
 {\it centrality}  on the generators.
This is because
 any ``chain group'' of $X$ is a composition of $\E$'s and $\F$'s 
and the maps $\eta_\E$ and $\eta_\F$ are adjoint to $\kappa_\F$ and $\kappa_\E$,
respectively.

%Let us denote by $X(a,b)$ the bicomplex $X$ shifted homologically by
%$a$ in vertical and by $b$ in horizontal direction. 

\begin{lemma}\label{mainlemma}
 There are maps
$$\kappa_\E: \E r\onen  \to r \E\onen $$
%\quad\quad {\text  {and}}\quad\quad
$$\bar \kappa_\F:  r \F {\bf 1}_{n+2}  \to  \F r {\bf 1}_{n+2}
$$
which are  homotopy equivalences in $\Kom(\UcatD)$.
\end{lemma}

Here  $\rib\E\onen$ is the bicomplex obtained by composing 
 $\rib\onenn{n+2}$ to the left of  $\E\onen$. 
The differentials are those
of $\rib\onenn{n+2}$ extended by identity on $\E\onen$.
The other bicomplexes are defined analogously.
% The Euler characteristic of $\rib\E\onen$ is
%\begin{equation}
%{\rm r}E1_n= q^{-\frac{n^2}{2}-3n-4} \sum^\infty_{k=0} (-1)^k
%q^{-kn-2k}(q^{-4};q^{-2})_k F^{(k)} E^{(k+1)}1_n \, .
%\label{Euler}
%\end{equation}

To construct $\kappa_\E$ we will proceed as follows.
We will define an intermediate bicomplex $\Int$
as an indecomposable 
summand of $\rib\E\onen$, whose
 ``chain groups'' are
$$C'_{k,l}=\F^{(k)}\E^{(k+1)}  \onen \la -kn-2k\ra\otimes \Lambda^l W'_k\, $$
with $W'_k:=\Span_\Z\{w_2, \dots, w_{k+1}\}$,
 the total $q$-degree shift $\la -\frac{n^2}{2}-3n-4\ra$ and the 
homological shift $[n/2+1,n/2+1]$. 
 Then we  show that
  $\E r\onen$  and $r\E\onen$ retract to $\Int$ (Theorems \ref{Er}, \ref{re}).
 Composing the
corresponding homotopy
equivalences we will get $\kappa_\E$.
The construction of the chain maps between   $\E r\onen$, $r\E\onen$
and $\Int$, and the proofs of Theorems \ref{Er} and \ref{re}
 is  the most involved  technical part of the paper.

To define
$\kappa_\F$ we use the invariance of $r\onen$ under $\sigma\omega$.
Indeed, we have
$$\sigma \omega (\kappa_\E): \sigma \omega (
\E r\onen) \to \sigma\omega (r\E\onen)\, .$$
However, $ \sigma \omega (
\E r\onen)= \sigma \omega ( r) \F {\bf 1}_{n+2} \simeq r\F {\bf 1}_{n+2}$
and similarly,
$\sigma\omega (r\E\onen)$ is isomorphic to $\F r{\bf 1}_{n+2}$.

The proof of {\it invertibility} is based on the next theorem
computing the action of the ribbon complex on the
category of complexes over 
 $\Gr$ defined in \cite{Lau1} and the main result of \cite{BL}.
Let us define the endofunctors $r^L_N$ and $r^R_N$ 
 of $\Com(\Gr)$ by tensoring with
$\Gamma_N(r\onen)$ on the left and right, respectively,
i.e. $r^L_N(X)= \Gamma_N(r\onen)X$. Analogously, the endofunctors
$(r^{-1})^L_N$ and $(r^{-1})^R_N$ are defined by tensoring with
$\Gamma_N(r^{-1}\onen)$.

\begin{theorem}\label{flag}
For any natural number $N$,
$r^L_N$ and $r^R_N$
are the identity endofunctors of $\Com(\Gr)$
up to degree shift. Their inverses are $(r^{-1})^L_N$ and $(r^{-1})^R_N$,
respectively.
  \end{theorem}

The main result of \cite{BL} says that $\UcatD$
is the inverse limit of {\bf Flag} 2-categories. 
Hence, we conclude that tensoring with
 $r^{\pm 1}r^{\mp 1}\onen$ is the inverse limit of the identity
functor in $\Gr$, which is the identity endofunctor of $^l\Com(\UcatD)$.

To prove {\it symmetry}, we construct a bicomplex
$\tilde r\onen$ which is 1) isomorphic to $r\onen$ in $\Kom(\UcatD)$,
and 2) invariant under $\sigma\omega$. This is done in Section \ref{old}.
Then, given the isomorphism $H: r\onen \to \tilde r\onen$,
the composition $\sigma \omega (H^{-1}) 
\circ H: r\onen \to \sigma\omega(r\onen)$ is the required isomorphism.

The paper is organized as follows. After some preliminaries, 
 we define the central elements $c_\lambda$ indexed by partitions,
$r\onen$, $r^{-1}\onen$ and the
intermediate bicomplex $\Int$.
The next four Sections are devoted 
to the definition of the chain maps and homotopies and 
to the proofs of Theorems \ref{Er}, \ref{re} and Theorem \ref{main} 
(Centrality). 
After that
we recall the definitions of the symmetry 2-functors and
define the images of $r\onen$ under those symmetries.
Section \ref{old} is devoted to the construction of $\tilde r\onen$.
In the last section we prove Theorem \ref{flag} and Theorem \ref{main}
(Invertibility).

In Appendix we collect  identities  needed for the proofs.
 
\subsection*{Acknowledgments}
The first author would like to thank Aaron Lauda for  helpful
discussions and Krzysztof Putyra for sharing
 his LaTeX package for drawing diagrams.

\section{General facts}

\subsection{Definitions and conventions}
We refer to \cite{Lau1} and \cite{KLMS} for
the definitions of the 2-categories $\U$ and its Karoubi envelope
$\UcatD$. The 2-morphisms in these 2-categories are given by
diagrams modulo some local relations.
The right most region in all our diagrams is labeled by $n$. 
For any
1-morphism $x\in \Hom_{\UcatD}(n,m)$, we denote by
$\Dot(x) \subset \End_{\UcatD}(x)$ the subspace of its 2-endomorphisms 
generated by dots. A thick line labeled with a 
positive integer $k$   denotes
the identity 2-morphism of $\E^{(k)}:=(\E^k, {\mathcal i}_k)\la 
-\frac{k(k-1)}{2}\ra$
if it is oriented upwards; and
the identity 2-morphism of
$\F^{(k)}:=(\F^k, {\mathcal i}'_k)\la \frac{k(k-1)}{2}\ra$ in $\UcatD$ otherwise, where
${\mathcal i}_k$ is the idempotent defined in \cite{KLMS} and ${\mathcal i}'_k$ is its image
under  $180$-degree rotation.
The case $k=1$ will be represented by a thin line without any label
for a better visibility.

In this paper, for any 2-category $\C$,
we denote by  $\Kom(\C)$  the 2-category of  bicomplexes
over the 2-category $\C$. The objects of $\Kom(\C)$ coincide with
  objects of $\C$, 1-morphisms
are bicomplexes of 1-morphisms in $\C$, and 2-morphisms 
are chain maps, constructed from 2-morphisms in $\C$.
Let $\Com(\C)$ be the 2-category with the same objects and 1-morphisms
as $\Kom(\C)$ but whose 2-morphisms are chain maps up to homotopy.
We will denote by $\Kom^b(\C)$ and $\Com^b(\C)$ corresponding bounded versions.

%The 2-category $U^*$ is obtained from $\UcatD$ by adding
%``translations'', which are 2-isomorphisms between $f$ and $f\la s\ra$,
%for any 1-morphism $f$ in $\UcatD$ and any $s\in \Z $.
%For any 1-morphism $f:n\to m$ in $\U^*$,
%let us denote by $\Dot(f) \subset \End_{\U^*}(f)$ the subspace of
%its 2-endomorphisms generated by dots.

The degree of a 2-morphism in $\U$ is defined as degree of the target 
minus degree of the source plus degree of the diagram.

\subsection{Strong deformation retraction}
Let us recall the definitions.

A chain complex $(C', d')$ is a  strong deformation retract of
a chain complex $(C,d)$ if there exist
\begin{itemize}
\item
a chain map $f: C\to C'$, i.e. $d'f=fd$;
\item
a chain map $g: C' \to C$, i.e. $d g = g d'$;
\item
a homotopy
$h: C_\bullet \to C_{\bullet -1}$ satisfying
$ hd+dh = 1 -gf$
and
\[\begin{array}{cc}
fg=1 & h^2=0\\
fh=0 &  hg=0
\end{array}\]
\end{itemize}

{\it Remark.} From the four equalities including $g$ it is enough
to show that $fg=1$ and $hd+dh=1-gf$. The other two equalities($dg=gd'$, $hg=0$)
follow from them.

\vskip2mm

Let $(C, d^V, d^H)$ and $(C', d'^V, d'^H)$ be two bicomplexes.
We say that the second bicomplex is 
a strong deformation retract of the first one if there exist
\begin{itemize}
\item
a chain map $f: C\to C'$ with $d'^Hf=fd^H$, $d'^Vf=fd^V$;
\item
a chain map $g: C' \to C$ with $d^H g = g d'^H$, $d^V g = g d'^V$;
\item
a homotopy
$h=h^H+h^V: C_\bullet \to C_{\bullet -1}$ with
$ h^Hd^H+ h^Vd^V+d^Hh^H+ d^Vh^V = 1 -gf$, $h^H d^V+d^Vh^H=0$ and $h^Vd^H+d^H h^V=0$
satisfying
\[\begin{array}{cccc}
fg=1 & \quad h^H h^H=0 & \quad h^Hh^V+h^Vh^H=0 &\quad h^Vh^V=0\\
fh^H=0 & \quad fh^V=0 & \quad h^Hg=0 & \quad h^Vg=0
\end{array}\]
\end{itemize}

{\it Remark.}  The  equalities $d^Hg=gd'^H$, $d^Vg=g d'^V$, 
$h^Hg=0$ and $h^Vg=0$
follow from the others.

\subsection{Symmetric functions}
Let us denote by $S_k$ the symmetric group and $A_k=\Z[x_1, \dots,x_k]^{S_k}$ 
the ring of symmetric polynomials. 
Let $A$ be the ring of symmetric functions, defined as the inverse
limit of the system $(A_k)_{k\in \N}$.

For a partition $\lambda=(\lambda_1,\lambda_2,\dots,\lambda_a)$ 
with $\lambda_1\geq \lambda_2\geq\dots\geq\lambda_a\geq 0$ let
$|\lambda| := \sum_{i=1}^a\lambda_i$. 
We denote by
 $P(a)$  the set of all partitions $\lambda$ with at most $a$ parts (i.e.\ with $\lambda_{a+1}=0$).
 Moreover,
the set of all partitions 
%with at most $a$
(i.e.\ the set $P(\infty)$) we denote simply by $P$.

The dual (conjugate) partition of  $\lambda$ is the partition
$\lambda^t=(\lambda^t_1,\lambda^t_2,\ldots)$ with
$\lambda^t_j=\sharp\{i|\lambda_i\ge j\}$ which is given by reflecting
the Young diagram of $\lambda$ along the diagonal.

The Schur polynomials $\{s_\lambda\;|\; \lambda \in P(k)\}$ form a basis
of $A_k$, as well as Schur functions
$\{s_\lambda\; |\; \lambda \in P\}$ is a base of $A$. 
The multiplication
in this basis is given by the following formula
$$s_\mu s_\nu =
\sum_{\lambda \in P} N^\lambda _{\mu \nu }s_\lambda ,$$
where $N^\lambda_{\mu\nu}$ are the Littlewood-Richardson coefficients.
The elementary symmetric functions
$\{s_{1^d}=e_d\;|\; d\in \N\}$ or the complete symmetric functions
$\{s_d=h_d\;|\; d\in \N\}$ generate $A$ multiplicatively.

The ring 
$A$ has a natural Hopf algebra structure  with comultiplication 
\begin{gather*}
  \Delta \zzzcolon A \longrightarrow A \otimes A 
\end{gather*}
 given by
\begin{gather*}
  \Delta (s_\lambda )=\sum_{\mu,\nu \in P} N^\lambda _{\mu \nu }s_\mu \otimes s_\nu ,
\end{gather*}
counit
\begin{gather*}
  \epsilon \zzzcolon A \longrightarrow \modZ ,\quad s_\lambda \mapsto \delta _{\lambda ,0},
\end{gather*}
and antipode
\begin{gather*}
  \gamma \zzzcolon A \longrightarrow A ,\quad s_\lambda \mapsto (-1)^{|\lambda |}s_{\lambda ^t}.
\end{gather*}

\section{The center of the 2-category $\UcatD$}\label{center}
After recalling the general definition of  a center for any linear category,
we construct central elements in $\UcatD(n,m)$.

\subsection{Center of a category}
For a linear category $C$, the {\em center} $Z(C)$ of $C$ is the ring
of endo-natural transformations on the identity functor $1_C\zzzcolon C\longrightarrow C$.
Thus, an element $\sigma $ of $Z(C)$ is a collection of endomorphisms
$\sigma _x\zzzcolon x\longrightarrow x$ for objects $x$ in $C$ such that we have
\begin{gather*}
  f\sigma _x=\sigma _yf
\end{gather*}
for any morphisms $f\zzzcolon x\longrightarrow y$ in $C$.  
Multiplication of two elements
$\sigma $ and $\tau $ in $Z(C)$  is
defined by
\begin{gather*}
  (\tau \sigma )_x:=\tau _x\sigma _x.
\end{gather*}
It is easily seen that $Z(C)$ is commutative.
We call $\sigma $ a {\em central element} of $C$.

Let $\modC $ be a linear $2$-category. For each pair $(x,y)$ of objects in
$\modC $, one can consider the center  $Z(\modC (x,y))$ of the category
$\modC (x,y)$ whose objects are
the  $1$-morphisms between $x$ and $y$ and morphisms are $2$-morphisms.
The center  $Z(\modC (x,y))$ is a commutative ring.

\subsection{Central elements in $\UcatD$}
\label{sub:center}

%For any 1-morphism $f:n\to m$ in $\UcatD$,
%let us denote by $\Dot(f) \subset \End_{\UcatD}(f)$ the subspace of
%its 2-endomorphisms generated by dots.

We have natural ring homomorphisms
\begin{gather*}
  d_a\zzzcolon A _a \longrightarrow  \Dot(\E^{(a)}\onen),\\
   d'_a\zzzcolon A _a \longrightarrow  \Dot(\F^{(a)}\onen).
\end{gather*}
which lift to
\begin{gather*}
  d_a\zzzcolon A  \longrightarrow  \Dot(\E^{(a)}\onen),\\
  d'_a\zzzcolon A  \longrightarrow  \Dot(\F^{(a)}\onen).
\end{gather*}

%\subsection{The maps $c\zzzcolon A \longrightarrow \Dot(f)$}
In this subsection, we define for each $1$-morphism 
$f\zzzcolon n\longrightarrow m$ in $\UcatD$ a
ring homomorphism
\begin{gather*}
  c=c_f\zzzcolon A \longrightarrow \Dot(f).
\end{gather*}

For $f=\E^{(a)}\onen$, we set
\begin{gather*}
  c_{\E^{(a)}\onen}(x)=d_a(x).
\end{gather*}
For $f=\F^{(a)}\onen$, we set
\begin{gather*}
  c_{\F^{(a)}\onen}(x)=d'_a(\gamma(x)).
\end{gather*}
For $f=f_1f_2\dots f_p$, where each $f_j$ is $\E^{(a_j)}{\mathbf 1}_{n_j}$ or
$\F^{(a_j)}{\mathbf 1}_{n_j}$, $a_j\ge 0$, $n_j\in \modZ $, we set
\begin{gather*}
  \begin{split}
    c_{f_1f_2\dots f_p}(x)&=\text{horizontal composition}\left((c_{f_1}\otimes \dots \otimes c_{f_p})\Delta ^{[p]}(x)\right)\\
    &=\sum c_{f_1}(x_{(1)})\circ\dots \circ c_{f_p}(x_{(p)})
  \end{split}
\end{gather*}
where $\Delta ^{[p]}\zzzcolon A \longrightarrow A ^{\otimes p}$ is the $p$-output comultiplication with
\begin{gather*}
  \Delta ^{[p]}(x)=\sum x_{(1)}\otimes \dots \otimes x_{(p)}.
\end{gather*}
Finally, for a direct sum $f=f_1\oplus\dots \oplus f_p$, we set
\begin{gather*}
  c_{f_1\oplus\dots \oplus f_p}(x)=\text{diag}(c_{f_1}(x),\ldots ,c_{f_p}(x)).
\end{gather*}
%[[notation correct?]]

We also adopt the notation $(c_\lambda)_f$ for $c_f(s_\lambda)$ and draw:
\begin{equation*}
(c_\lambda)_f =\xy
(-9,6);(-9,-6); **[black][|(2)]\dir{-} ?(1)*[|(1)]\dir{>};
(-5,-6);(-5,6); **[black][|(2)]\dir{-} ?(1)*[|(1)]\dir{>};
(5,6);(5,-6); **[black][|(2)]\dir{-} ?(1)*[|(1)]\dir{>};
(-11,-6)*{\scriptstyle a_1};
(-3,6)*{\scriptstyle a_2};
(7,-6)*{\scriptstyle a_p};
(0,-4)*{\dots}; (0,4)*{\dots};
(-2,0)*{\bigb{\quad\quad c_\lambda \quad\quad}};
(-12,0)*{};(8,0)*{};
\endxy
\end{equation*}
where $f=f_1f_2\dots f_p$, each $f_j$ is $\E^{(a_j)}{\mathbf 1}_{n_j}$ or
$\F^{(a_j)}{\mathbf 1}_{n_j}$. 
\vskip2mm

Note that, for $f=\E^{a}\onen$ and $f=\F^a\onen$  we have
\begin{equation*}
\xy
{\ar (-5,-6)*{}; (-5,6)*{}};
{\ar (5,-6)*{}; (5,6)*{}};
(0,-4)*{\dots}; (0,4)*{\dots};
(0,0)*{\bigb{\quad\,\,c_\lambda \quad\,\,}};
(-8,0)*{};(8,0)*{};
\endxy \;\;
:=\quad
\xy
{\ar (-5,-6)*{}; (-5,6)*{}};
{\ar (5,-6)*{}; (5,6)*{}};
(0,-4)*{\dots}; (0,4)*{\dots};
(0,0)*{\bigb{\;\Delta^{[a]}(s_\lambda)\; }};
(-8,0)*{};(8,0)*{};
\endxy,
\quad\quad\text{and}\quad\quad
\xy
{\ar (-5,6)*{}; (-5,-6)*{}};
{\ar (5,6)*{}; (5,-6)*{}};
(0,-4)*{\dots}; (0,4)*{\dots};
(0,0)*{\bigb{\quad c_\lambda \quad\,}};
(-8,0)*{};(8,0)*{};
\endxy\;\;
:=(-1)^{|\lambda|} \quad\xy
{\ar (-5,6)*{}; (-5,-6)*{}};
{\ar (5,6)*{}; (5,-6)*{}};
(0,-4)*{\dots}; (0,4)*{\dots};
(0,0)*{\bigb{\;\; \Delta^{[a]}(s_{\lambda^t}) \;\;}};
(-8,0)*{};(8,0)*{};
\endxy \;\; .
\end{equation*}
\vskip2mm
For $f=\E^{(a)}\onen$ and $f=\F^{(a)}\onen$  we have
\begin{equation*}
\xy
 (0,7);(0,-7); **[black][|(2)]\dir{-} ?(1)*[|(1)]\dir{>};
 (2,-7)*{\scriptstyle a};
 (4,0)*{};
 (-4,0)*{};
(0,0)*{\bigb{c_\lambda}};
\endxy=\xy
 (0,7);(0,-7); **[black][|(2)]\dir{-} ?(1)*[|(1)]\dir{>};
 (2,-7)*{\scriptstyle a};
 (4,0)*{};
 (-4,0)*{};
(0,0)*{\bigb{\lambda}};
\endxy,\quad\xy
 (0,-7);(0,7); **[black][|(2)]\dir{-} ?(1)*[|(1)]\dir{>};
 (2,7)*{\scriptstyle a};
 (4,0)*{};
 (-4,0)*{};
(0,0)*{\bigb{c_\lambda}};
\endxy=(-1)^{|\lambda|}\xy
 (0,-7);(0,7); **[black][|(2)]\dir{-} ?(1)*[|(1)]\dir{>};
 (2,7)*{\scriptstyle a};
 (4,0)*{};
 (-4,0)*{};
(0,0)*{\bigb{\lambda^t}};
\endxy.
\end{equation*}

The following  proposition is the direct consequence of the definitions.

\begin{prop}
  \label{r2}
  For $1$-morphisms $f\zzzcolon n\longrightarrow l$ and $g\zzzcolon l\longrightarrow m$ in $\UcatD$, we have
  \begin{gather*}
    (c_\lambda )_{gf}=\sum_{\mu ,\nu \in P}N^\lambda _{\mu ,\nu }(c_\mu )_g\circ(c_\nu )_f
  \end{gather*}
and 
\begin{gather*}
    (c_\lambda c_\mu)_{gf} =\sum_{\nu \in P}N^\nu _{\lambda ,\mu }(c_\nu)_{gf}\,  .
  \end{gather*}
\end{prop}

%\subsection{Kernel of $c_f$}[[ ]]

%\subsection{(Vertical) center of $\U^*(n,m)$}

%We consider the commutative ring $Z(\U^*(n,m))$.

In particular, for
 $\lambda \in P_d$, and $f=\E^{(i_1)}\F^{(j_1)}\dots \E^{(i_p)}\F^{(j_p)}\onen$
we have
\begin{gather*}
  (c_\lambda )_{\E^{(i_1)}\F^{(j_1)}\dots \E^{(i_p)}\F^{(j_p)}\onen}\\
  =\sum_{\lambda ^{(1)},\mu ^{(1)},\ldots ,\lambda ^{(p)},\mu ^{(p)}\in P}
  (-1)^{l_1+\dots +l_p}
  N^\lambda _{\lambda ^{(1)}\mu ^{(1)}\ldots \lambda ^{(p)}\mu ^{(p)}}
  \lambda ^{(1)}\circ(\mu ^{(1)})^t\circ\dots \circ\lambda ^{(p)}\circ(\mu ^{(p)})^t
\end{gather*}
where 
$N^\lambda _{\lambda ^{(1)}\mu ^{(1)}\ldots \lambda ^{(p)}\mu ^{(p)}}\in \modZ _{\ge 0}$ 
are
 the
Littlewood--Richardson coefficients.  
\vskip2mm

{\bf Examples.}
$$
1) \;\; (c_d)_{\F^{(b)}\E^{(a)}\onen}
=\sum_{k,l\ge 0 \atop k+l=d} d'_b(\gamma(h_l))\circ d_a(h_k)=
\sum_{k,l\ge 0 \atop k+l=d} (-1)^l
\xy
 (-8,6);(-8,-6); **[black][|(2)]\dir{-} ?(0)*[|(1)]\dir{<};
 (0,6);(0,-6); **[black][|(2)]\dir{-} ?(1)*[|(1)]\dir{>};
% (6,4)*{n};
 (1.5,-6)*{\scriptstyle {a}};
 (-6.5,6)*{\scriptstyle {b}};
 (4,0)*{};
 (-12,0)*{};
(0,0)*{\bigb{h_k}};
(-8,0)*{\bigb{e_l}};
\endxy ,
$$
%Here
%\begin{gather*}
%  h_d=(d),\quad \epsilon _d=(1^d)=(1,\ldots ,1).
%\end{gather*}

\begin{gather*}
2) \;\; (c_d)_{\E^{(i_1)}\F^{(j_1)}\dots \E^{(i_p)}\F^{(j_p)}\onen}
  =\sum_{k_1,l_1,\ldots ,k_p,l_p\ge 0 \atop k_1+l_1+\dots +k_p+l_p=d}
  d_{i_1}(h_{k_1})\circ d'_{j_1}(\gamma(h_{l_1}))\circ\dots \circ d_{i_p}(h_{k_p})\circ d'_{j_p}(\gamma(h_{l_p}))=\\
  =\sum_{k_1,l_1,\ldots ,k_p,l_p\ge 0 \atop k_1+l_1+\dots +k_p+l_p=d}
  (-1)^{l_1+\dots +l_p}
\xy
 (0,6);(0,-6); **[black][|(2)]\dir{-} ?(1)*[|(1)]\dir{>};
 (2,-6)*{\scriptstyle {i_1}};
 (4,0)*{};
 (-4,0)*{};
(0,0)*{\bigb{h_{k_1}}};
\endxy
\xy
 (0,-6);(0,6); **[black][|(2)]\dir{-} ?(1)*[|(1)]\dir{>};
 (2,6)*{\scriptstyle {j_1}};
 (4,0)*{};
 (-4,0)*{};
(0,0)*{\bigb{e_{l_1}}};
\endxy
\dots
\xy
 (0,6);(0,-6); **[black][|(2)]\dir{-} ?(1)*[|(1)]\dir{>};
 (2,-6)*{\scriptstyle {i_p}};
 (4,0)*{};
 (-4,0)*{};
(0,0)*{\bigb{h_{k_p}}};
\endxy
\xy
 (0,-6);(0,6); **[black][|(2)]\dir{-} ?(1)*[|(1)]\dir{>};
 (2,6)*{\scriptstyle {j_p}};
 (4,0)*{};
 (-4,0)*{};
(0,0)*{\bigb{e_{l_p}}};
\endxy .
\end{gather*}

\begin{prop}
%  \label{r1}
  For $m,n\in \modZ $, $\lambda \in P$ we have $c_\lambda \in Z(\UcatD(n,m))$.
\end{prop}

\begin{proof}
We need to prove that for every $2$-morphism $\alpha: f\rightarrow g$ in 
$\UcatD(n,m)$ we have
$$
\alpha(c_\lambda)_f=(c_\lambda)_g\alpha \, .
$$
%where $\cdot$ denotes vertical composition of $2$-morphisms.
%By Proposition \ref{r13} it is enough to prove the proposition for generators.

Splitting the thick lines and moving the dots as follows 
\begin{equation*}
\xy
 (0,1);(-9,8)*{} **\crv{(-9,2)}?(1)*\dir{>};
 (0,1);(-6,8)*{} **\crv{(-6,2)}?(1)*\dir{>};
 (0,1);(6,8)*{} **\crv{(6,2)}?(1)*\dir{>};
 (0,1);(9,8)*{} **\crv{(9,2)}?(1)*\dir{>};
 (0,1);(0,-8) **[black][|(2)]\dir{-};
 (0,5)*{\dots};
 (0,-3)*{\bigb{c_\lambda}};
 (7,-4)*{n};
 (2,-8)*{\scriptstyle a};
 (12,0)*{};
 (-12,0)*{};
\endxy=\xy
 (0,-1);(-9,8)*{} **\crv{(-9,0)}?(1)*\dir{>};
 (0,-1);(-6,8)*{} **\crv{(-6,0)}?(1)*\dir{>};
 (0,-1);(6,8)*{} **\crv{(6,0)}?(1)*\dir{>};
 (0,-1);(9,8)*{} **\crv{(9,0)}?(1)*\dir{>};
 (0,-1);(0,-8) **[black][|(2)]\dir{-};
 (0,7)*{\dots};
 (0,3)*{\bigb{\quad\quad\,c_\lambda \quad\quad \,}};
 (7,-4)*{n};
 (2,-8)*{\scriptstyle a};
 (12,0)*{};
 (-12,0)*{};
\endxy,\quad
\xy
 (0,-1);(-9,-8)*{} **\crv{(-9,-2)}?(.8)*\dir{<};
 (0,-1);(-6,-8)*{} **\crv{(-6,-2)}?(.8)*\dir{<};
 (0,-1);(6,-8)*{} **\crv{(6,-2)}?(.8)*\dir{<};
 (0,-1);(9,-8)*{} **\crv{(9,-2)}?(.8)*\dir{<};
 (0,-1);(0,8) **[black][|(2)]\dir{-};
 (0,-5)*{\dots};
 (0,3)*{\bigb{c_\lambda}};
 (7,4)*{n};
 (2,8)*{\scriptstyle a};
 (12,0)*{};
 (-12,0)*{};
\endxy=\xy
 (0,1);(-9,-8)*{} **\crv{(-9,0)}?(.9)*\dir{<};
 (0,1);(-6,-8)*{} **\crv{(-6,0)}?(.9)*\dir{<};
 (0,1);(6,-8)*{} **\crv{(6,0)}?(.9)*\dir{<};
 (0,1);(9,-8)*{} **\crv{(9,0)}?(.9)*\dir{<};
 (0,1);(0,8) **[black][|(2)]\dir{-};
 (0,-7)*{\dots};
 (0,-3)*{\bigb{\quad\quad\, c_\lambda \quad\quad\,}};
 (7,4)*{n};
 (2,8)*{\scriptstyle a};
 (12,0)*{};
 (-12,0)*{};
\endxy,
\end{equation*}
we see that it is enough
 to prove the proposition for 2-morphisms of $\Ucat$, generated by
 dots, crossings and turns. For dots the proposition is clear.

Since elementary functions   generate  $A$, 
it is enough to consider  $c_\lambda$ for
 $\lambda=1^d$, $d\in \N$ in what follows.
Note that
$(c_{1^d})_{\E^2\onen}\neq 0$ only for $d=1$ or $d=2$ and
\begin{equation*}
\xy
(-3,-6)*{}; (-3,6)*{}**\dir{-} ?(1)*\dir{>};
(3,-6)*{}; (3,6)*{}**\dir{-} ?(1)*\dir{>};
{\ar (3,-6)*{}; (3,6)*{}};
(0,0)*{\bigb{\,\,\, c_{(1)} \,\,\,}};
(-7,0)*{};(7,0)*{};
\endxy=\xy
(-3,-6)*{}; (-3,6)*{}**\dir{-} ?(1)*\dir{>}?(.5)*{\bullet};
(3,-6)*{}; (3,6)*{}**\dir{-} ?(1)*\dir{>};
(-6,0)*{};(6,0)*{};
\endxy+\xy
(-3,-6)*{}; (-3,6)*{}**\dir{-} ?(1)*\dir{>};
(3,-6)*{}; (3,6)*{}**\dir{-} ?(1)*\dir{>}?(.5)*{\bullet};
(-6,0)*{};(6,0)*{};
\endxy,\quad\xy
(-3,-6)*{}; (-3,6)*{}**\dir{-} ?(1)*\dir{>};
(3,-6)*{}; (3,6)*{}**\dir{-} ?(1)*\dir{>};
(0,0)*{\bigb{c_{(1,1)}}};
(-7,0)*{};(7,0)*{};
\endxy=\xy
(-3,-6)*{}; (-3,6)*{}**\dir{-} ?(1)*\dir{>}?(.5)*{\bullet};
(3,-6)*{}; (3,6)*{}**\dir{-} ?(1)*\dir{>}?(.5)*{\bullet};
(-6,0)*{};(6,0)*{};
\endxy.
\end{equation*}
Using the NilHecke relations
 we can easily check that both upper $2$-morphisms commute with the crossing. The same is true for downward oriented  arrows.

Similarly, we have
\begin{equation*}
\xy
(-3,-6)*{}; (-3,6)*{}**\dir{-} ?(0)*\dir{<};
(3,6)*{}; (3,-6)*{}**\dir{-} ?(0)*\dir{<};
%{\ar (3,-6)*{}; (3,6)*{}};
(0,0)*{\bigb{\,\,c_{ 1^d} \,\,\,}};
(-7,0)*{};(7,0)*{};
\endxy=\xy
(-3,-6)*{}; (-3,6)*{}**\dir{-} ?(0)*\dir{<}?(.5)*{\bullet};
(-1,1)*{\scriptstyle d};
(3,6)*{}; (3,-6)*{}**\dir{-} ?(0)*\dir{<};
(-6,0)*{};(6,0)*{};
\endxy-\xy
(-3,-6)*{}; (-3,6)*{}**\dir{-} ?(0)*\dir{<}?(.5)*{\bullet};
(0.5,1)*{\scriptstyle d-1};
(5,6)*{}; (5,-6)*{}**\dir{-} ?(0)*\dir{<}?(.5)*{\bullet};
(-6,0)*{};(8,0)*{};
\endxy.
\end{equation*}
We see that by multiplying with the turn from below, we get $0$,
which coincides with $(c_{1^d})_{1_n}=0$. 
The other turns can be proved similarly.
\end{proof}

\section{Definitions of $r\onen$,
$r^{-1}\onen$ and $\Int$}\label{rib-def}

\subsection{Ribbon bicomplex}
As it was already mentioned in Introduction,
the bicomplex $\rib\onen$ categorifying the ribbon element has
 ``chain groups'' 
 $$C_{k,l}:= \F^{(k)} \E^{(k)}\onen\la -kn-k\ra\otimes \Lambda^l W_k,\quad
W_k=\Span_\Z\{w_1,\dots, w_k\}, \quad \deg(w_j)=-2j\, $$
with the total $q$-degree shift $\la -\frac{n^2}{2}-n\ra$
and homological shift $[n/2,n/2]$. Here note that for an additive category
$\mathcal{C}$ and a free abelian group $G$ of
finite rank  one can construct a functor $-\otimes
G\colon\mathcal{C}\to\mathcal{C}$.

The horizontal differential
$ d^H_{k,l}: C_{k,l}\to C_{k,l+1}$ sends $x \mapsto c\wedge x$ where
\[c:= \sum^k_{j=1} c_j \otimes w_j :=
\sum^k_{j=1}  \left( \sum^j_{i=0} (-1)^i\quad
\xy
 (-8,6);(-8,-6); **[black][|(2)]\dir{-} ?(0)*[black][|(1)]\dir{<};
 (0,6);(0,-6); **[black][|(2)]\dir{-} ?(1)*[black][|(1)]\dir{>};
% (6,4)*{n};
 (2,5)*{\scriptstyle {k}};
 (-6,5)*{\scriptstyle {k}};
% (8,0)*{};
 (-8,0)*{};
(0,0)*{\bigb{h_{j-i}}};
(-8,0)*{\bigb{e_{i}}};
\endxy \;\right) \otimes w_j \in \Dot( \F^{(k)}\E^{(k)}\onen) \otimes W_k
\]
%We will  also use the notation
% $(c_d)_{\F^{(k)}\E^{(k)}\onen}=\sum^d_{j=0}(-1)^j e_j\otimes h_{d-j}$.
%Note that
%for any 1-morphism $x$ of $\UcatD$ and any partition $\lambda$,
%the elements $(c_\lambda)_x$ were introduced in \cite{BHZ} and shown to be
are our central elements.

To define the vertical differential we proceed as follows.
Consider the linear map
\begin{align*}
\alpha_k: W_k &\to 
\Dot(\E^{(k)}\onen)
 \otimes W_{k+1}\\
\alpha_k(w_i)&= {\bf 1}\otimes w_i -(-1)^{k+1-i}\;\quad
\xy
% (-9,6);(-9,-6); **[black][|(2)]\dir{-} ?(0)*[black][|(1)]\dir{<};
 (0,6);(0,-6); **[black][|(2)]\dir{-} ?(1)*[black][|(1)]\dir{>};
(2,5)*{\scriptstyle {k}};
 (-4,5)*{\scriptstyle {}};
(-0,0)*{\bigb{e_{k+1-i}}};
\endxy  \quad\otimes w_{k+1}
\end{align*}
where ${\bf 1}$ denotes the identity 2-morphism of $\E^{(k)}1_n$.
This map induces an algebra homomorphism
\begin{equation*}
\alpha_k: \Lambda^\bullet W_k \to 
\Dot(\E^{(k)}\onen) \otimes \Lambda^\bullet W_{k+1}
\quad {\text {and \; let}}\quad
\alpha_{k,l}:\Lambda^l W_k \to 
\Dot(\E^{(k)}\onen)
 \otimes \Lambda^l W_{k+1}
\end{equation*}
be its degree $l$ part.
Then the vertical differential is
\[ d^V_{k,l}:= (-1)^l 
\xy
 (-2,8);(-2,-8); **[black][|(2)]\dir{-} ?(.55)*[black][|(1)]\dir{<}?;
 (6,-8);(6,8); **[black][|(2)]\dir{-} ?(.85)*[black][|(1)]\dir{<}?;
 (6,6)*{};(-2,6)*{} **[black][|(1)]\crv{(6,2) & (-2,2)} ?(.25)*[black][|(1)]\dir{<}?;
 (-6,9)*{\scriptstyle k+1};
 (10,9)*{\scriptstyle k+1};
 (-5,-9)*{\scriptstyle k};
 (8,-9)*{\scriptstyle k};
 (6,-2)*{\bigb{\alpha_{k,l}}};
% (12,-3)*{n};
\endxy\; :  C_{k,l}\to C_{k+1,l}\, \quad \text{where}
\]
\begin{align} \label{alpha} 
\alpha_{k,l}(w_{i_1}\wedge w_{i_2}\wedge \dots\wedge w_{i_l})&=
\alpha_k(w_{i_1})\wedge \alpha_k(w_{i_2})\wedge \dots\wedge \alpha_k(w_{i_l})\\\nn=
{\bf 1} \otimes w_{i_1}\wedge \dots \wedge w_{i_l}&-
\sum^l_{j=1} (-1)^{k+1-i_j+l-j}\;
\xy
% (-10,6);(-10,-6); **[black][|(2)]\dir{-} ?(0)*[black][|(1)]\dir{<};
 (0,6);(0,-6); **[black][|(2)]\dir{-} ?(1)*[black][|(1)]\dir{>};
(2,5)*{\scriptstyle {k}};
 (-7,5)*{\scriptstyle {}};
(-0,0)*{\bigb{e_{k+1-i_j}}};
\endxy  \quad\;
\otimes (w_{i_i}\wedge \dots \wedge\hat w_{i_j}\wedge\dots \wedge w_{i_l})
\wedge w_{k+1}
\end{align}

In what follows to simplify the notation
we will often omit the tensor sign between the graphical part
and $\Lambda^l W_k$. Also, we will sometimes use
a bullet labeled with a symmetric function instead of a box.

\begin{prop}
$\rib\onen$ as defined above is  a bicomplex, i.e.
\[ d^H\circ  d^H=0 \quad\quad d^V\circ d^H + d^H\circ d^V=0
\quad\quad d^V \circ d^V=0\]
\end{prop}

Before giving the proof,
let us show how the beginning of this bicomplex looks like.

\[ 
\xy
(-50,50)*+{\scs \onen};
(-50,46);(-50,35); **\dir{-} ?(0)*\dir{<};
(-55,40)*+{\Ucupr};
(-50,30)*+{{\scs\F\E\onen \la -n-1\ra}};
(-30,35)*+{\scs c_1};
(-4,30)*+{{\scs \F\E\onen \la  -n-3\ra w_1}};
(-35,30);(-25,30)**\dir{-}?(1)*\dir{>};
(-50,25);(-50,13); **\dir{-} ?(0)*\dir{<};
(-58,22);(-58,16); **[black][|(.8)]\dir{-} ?(0)*[black][|(1)]\dir{<}?;
 (-58,22);(-58,24); **[black][|(2)]\dir{-};
 (-54,22);(-54,16); **[black][|(.8)]\dir{-}; 
 (-54,22);(-54,24); **[black][|(2)]\dir{-} ?(0)*[black][|(1)]\dir{<}?;
 (-58,22)*{};(-54,22)*{} **[black][|(.8)]\crv{(-58,18) & (-54,18)} ?(.65)*[black][|(1)]\dir{>}?;
 (-60,24)*{\scs 2};
 (-56,24)*{\scs 2};
%%%%%
(-5,25);(-5,13); **\dir{-} ?(0)*\dir{<};
(-1,22);(-1,16); **[black][|(.8)]\dir{-} ?(0)*[black][|(1)]\dir{<}?;
 (-1,22);(-1,24); **[black][|(2)]\dir{-};
 (3,22);(3,16); **[black][|(.8)]\dir{-}; 
 (3,22);(3,24); **[black][|(2)]\dir{-} ?(0)*[black][|(1)]\dir{<}?;
 (-1,22)*{};(3,22)*{} **[black][|(.8)]\crv{(-1,18) & (3,18)} ?(.65)*[black][|(1)]\dir{>}?;
 (-3,24)*{\scs 2};
 (1,24)*{\scs 2};
%%%%%
(14,20)*+{\scs \left(\begin{array}{l}
\hskip-1mm
{\scs -1 }\hskip-1mm
\\ \hskip-1mm {\scs -(e_1)_2} \hskip-1mm\end{array}\right)};
%%%%
(-50,5)*+{{\scs \F^{(2)}\E^{(2)}\onen\la  -2n-2\ra}};
(-4,5)*+{
\left(
\begin{array}{c} 
{\scs \F^{(2)}\E^{(2)}\onen\la  -2n-4\ra w_1}\\
{\scs \F^{(2)}\E^{(2)}\onen\la  -2n-6\ra w_2}
\end{array}\right)
};
(-35,5);(-25,5)**\dir{-} ?(1)*\dir{>};
(-30,13)*+{\scs \left(\begin{array}{l}
\hskip-1mm
{\scs c_1} \hskip-1mm
\\ \hskip-1mm {\scs c_2} \hskip-1mm\end{array}\right)};
(68,5)*+{{\scs\F^{(2)}\E^{(2)}\onen\la  -2n-4\ra w_1\wedge w_2}};
(22,5);(45,5)**\dir{-} ?(1)*\dir{>};
(32,10)*+{ \left(\begin{array}{ll}
\hskip-1mm
{\scs -c_2} & {\scs c_1}\hskip-1mm\end{array}\right)};
%%%%%%%%%%%%%%%%%%%%%%%%%%%%%%%%%%%%%%%%%%%%%%
(-50,0);(-50,-20); **\dir{-} ?(0)*\dir{<};
(-58,-10);(-58,-16); **[black][|(2)]\dir{-} ?(0)*[black][|(1)]\dir{<}?;
 (-58,-10);(-58,-8); **[black][|(2)]\dir{-};
 (-54,-10);(-54,-16); **[black][|(2)]\dir{-}; 
 (-54,-10);(-54,-8); **[black][|(2)]\dir{-} ?(0)*[black][|(1)]\dir{<}?;
 (-58,-10)*{};(-54,-10)*{} **[black][|(.8)]\crv{(-58,-14) & (-54,-14)} ?(.65)*[black][|(1)]\dir{>}?;
 (-60,-7)*{\scs 3};
 (-56,-7)*{\scs 3};
%%%%%
(-5,0);(-5,-20); **\dir{-} ?(0)*\dir{<};
(-1,-10);(-1,-16); **[black][|(2)]\dir{-} ?(0)*[black][|(1)]\dir{<}?;
 (-1,-10);(-1,-8); **[black][|(2)]\dir{-};
 (3,-10);(3,-16); **[black][|(2)]\dir{-}; 
 (3,-10);(3,-8); **[black][|(2)]\dir{-} ?(0)*[black][|(1)]\dir{<}?;
 (-1,-10)*{};(3,-10)*{} **[black][|(.8)]\crv{(-1,-14) & (3,-14)} ?(.65)*[black][|(1)]\dir{>}?;
 (-3,-7)*{\scs 3};
 (1,-7)*{\scs 3};
%%%%%
(20,-12)*+{ \left(\begin{array}{cc}
\hskip-1mm
\scs -1 &\scs 0 \hskip-1mm\\
\hskip-1mm
\scs 0 &\scs -1 \hskip-1mm
\\ \hskip-1mm \scs (e_2)_2& \scs -(e_1)_2 \hskip-1mm\end{array}\right)};
%%%%
(-50,-30)*+{{\scs \F^{(3)}\E^{(3)}\onen\la  -3n-3\ra}};
%%%%%%
(-4,-30)*+{
\left(
\begin{array}{c} 
\scs \F^{(3)}\E^{(3)}\onen \la  -3n-5\ra w_1\\
\scs \F^{(3)}\E^{(3)}\onen\la  -3n-7\ra w_2\\
\scs \F^{(3)}\E^{(3)}\onen\la  -3n-9\ra w_3\\
\end{array}\right)
};
(-35,-30);(-25,-30)**\dir{-} ?(1)*\dir{>};
%%%%%
(-30,-20)*+{\scs \left(\begin{array}{l}
\hskip-1mm
\scs c_1 \hskip-1mm
\\ \hskip-1mm\scs c_2 \hskip-1mm\\
\hskip-1mm \scs c_3 \hskip-1mm\end{array}\right)};
%%%%%%
(70,-30)*+{\left(\begin{array}{l}
\scs \F^{(3)}\E^{(3)}\onen\la  -3n-9\ra w_1\wedge w_2\\
\scs \F^{(3)}\E^{(3)}\onen\la  -3n-11\ra w_1\wedge w_3\\
\scs \F^{(3)}\E^{(3)}\onen\la  -3n-13\ra w_2\wedge w_3\end{array}\right)
};
(22,-30);(45,-30)**\dir{-} ?(1)*\dir{>};
(32,-40)*+{\scs \left(\begin{array}{ccc}
\hskip-1mm
\scs -c_2 &\scs c_1&\scs 0\hskip-1mm\\
\hskip-1mm \scs -c_3&\scs 0&\scs c_1\hskip-1mm\\
\hskip-1mm \scs 0&\scs -c_3&\scs c_2\hskip-1mm
\end{array}\right)};
%%%%%%%%%%
(65,0);(65,-20); **\dir{-} ?(0)*\dir{<};
%%%
(69,-10);(69,-16); **[black][|(2)]\dir{-} ?(0)*[black][|(1)]\dir{<}?;
 (69,-10);(69,-8); **[black][|(2)]\dir{-};
 (73,-10);(73,-16); **[black][|(2)]\dir{-}; 
 (73,-10);(73,-8); **[black][|(2)]\dir{-} ?(0)*[black][|(1)]\dir{<}?;
 (69,-10)*{};(73,-10)*{} **[black][|(.8)]\crv{(69,-14) & (73,-14)} ?(.65)*[black][|(1)]\dir{>}?;
 (67,-7)*{\scs 3};
 (71,-7)*{\scs 3};
%%%%%
(83,-12)*+{ \left(\begin{array}{c}
\hskip-1mm
\scs -1  \hskip-1mm\\
\hskip-1mm
\scs (e_1)_2 \hskip-1mm
\\ \hskip-1mm \scs (e_2)_2 \hskip-1mm\end{array}\right)};
%%%
(95,-30);(100,-30)**\dir{.}?(1)*[black][|(1)]\dir{>}?;
(105,-30)*+{C_{3,3}};
%\begin{array}{l}\F^{(3)}\E^{(3)}\onen\\{\scs\la  -3n-17\ra w_1\wedge w_2\wedge w_3}
%\end{array}
(-50,-40);(-50,-50)**\dir{.} ?(1)*[black][|(1)]\dir{>}?;
(-50,-60)*+{C_{4,0}};
(-5,-40);(-5,-50)**\dir{.} ?(1)*[black][|(1)]\dir{>}?;
(-5,-60)*+{C_{4,1}};
(65,-40);(65,-50)**\dir{.}?(1)*[black][|(1)]\dir{>}?;
(65,-60)*+{C_{4,2}};
\endxy
\]
\centerline{{\bf Figure 1} The ribbon bicomplex}

\begin{proof}
The first formula $(d^H)^2=0$
is immediate from the definition, since $c\wedge c=0$.

For the next equation we have to check that
the following square anticommutes:
\[ \xymatrix{C_{k,l}\ar[r]^{c\wedge}\ar[d]^{d^V_{k,l}} &C_{k,l+1} \ar[d]^{d^V_{k,l+1}}
	    	  \\
      C_{k+1,l}\ar[r]^{c\wedge}&  C_{k+1,l+1}
		}
       \]

%The anti-commutativity of the vertical and the horizontal differentials
% is guaranteed by the choice of signs and the following property:
%In \cite{BHZ} we proved that $c_i$ belong to the Drinfeld center 
%of $\UcatD$, i.e.
%for any $f\in 
%\Hom_{\UcatD} (\F^{(k)} \E^{(k)}\onen, \F^{(k')} \E^{(k')}\onen)$ and any 
%positive integer $j$
%  $$ f c_j|_{\F^{(k)} \E^{(k)}\onen}= c_j|_{\F^{(k')} \E^{(k')}\onen} f\, .$$
After moving all $c_j$ down,  by using their centrality,
the 
anticommutativity reduces to 
$$\sum^{k+1}_{i_0=1}  w_{i_0} \wedge d^V_{k,l}\, c_{i_0}  
+ \sum^{k}_{i_0=1}  d^V_{k,l+1} \wedge w_{i_0}\, c_{i_0} =0$$
or to
the following identity in $\Dot (\F^{(k)}\E^{(k)}\onen) \Lambda^l W_k$
\begin{align*}
\sum^k_{i_0=1} c_{i_0} w_{i_0}\wedge w_{i_1} \wedge \dots\wedge w_{i_l}
-
\sum^{l}_{j=0}(-1)^{k+1-i_j+l-j} \sum^k_{i_0=1} c_{i_0} (e_{k+1-i_j})_2 w_{i_0}\wedge \dots
\hat w_{i_j} \dots\wedge  w_{i_l} \wedge w_{k+1}=\\
\sum^{k+1}_{i_0=1} c_{i_0} w_{i_0}\wedge w_{i_1} \wedge \dots\wedge w_{i_l}
-\sum^{l}_{j=1}(-1)^{k+1-i_j+l-j} \sum^k_{i_0=1} c_{i_0} (e_{k+1-i_j})_2 w_{i_0}\wedge \dots
\hat w_{i_j} \dots\wedge  w_{i_l}\wedge w_{k+1}
\end{align*}
where the lower index $2$ indicates the stand where this endomorphism
acts. This  is easily seen to hold after accomplishing the second summand 
for $j=0$ to zero by means of the formula
\begin{equation}\label{d21}
 \sum^{k+1}_{i_0=1} (-1)^{k+1-i_0} c_{i_0} (e_{k+1-i_0})_2=0 \in \Dot (\F^{(k)}\E^{(k)}\onen)
\end{equation}
or pictorially
\[
\sum^{k+1}_{i=1} (-1)^{k+1-i} \xy
(-4,8);(-4,-8); **[black][|(2)]\dir{-} ?(0)*[black][|(1)]\dir{<};
 (4,8);(4,-8); **[black][|(2)]\dir{-} ?(1)*[black][|(1)]\dir{>};
%(-12,10);(-12,6); **[black][|(1)]\dir{-} ?(1)*[black][|(1)]\dir{>};
(-6,8)*{\scriptstyle {k}};
% (1,8)*{\scriptstyle {k}};
(6,8)*{\scriptstyle {k}};
%(2,-8)*{\scriptstyle {k}};
(0,-3)*{\bigb{\;\;\; c_{i}\;\;\;}};
(4,3)*{\bigb{e_{k+1-i}}};
\endxy\quad=\;\;0 \, .
\]

  The formula \eqref{d21} follows from the fact that
$e_{k+1} \in \Dot(\E^{(k)}\onen)$ (or $e_{k+1}\in \Dot(\F^{(k)}\onen)$)
is zero and
\begin{equation}\label{d2}
\sum^d_{j=0} (-1)^{d-j} c_j (e_{d-j})_2= 
\sum^d_{j=0} (-1)^j \sum^{d-j}_{l=0} (-1)^{d-j-l} (e_j)_1 (e_{d-j-l})_2 (h_l)_2=
(c_d)_1 \in \Dot (\F^{(r)}\E^{(s)}\onen)
\end{equation}
for any $d$, $r$ and $s$, since
 $\sum^p_{i=0} (-1)^i e_i  h_{p-i}=0$ for any $p>0$. Here $(c_d)_1=(c_d)_{\F^{(r)}}$
is equal to $(-1)^d e_d$ sitting on the first strand.

The proof of  $d^V_{k,l}\circ d^V_{k-1,l}=0$ is based on the following identities
\[ 
\xy
 (-2,8);(-2,-12); **[black][|(2)]\dir{-} ?(.55)*[black][|(1)]\dir{<}?;
 (6,-12);(6,8); **[black][|(2)]\dir{-} ?(0)*[black][|(1)]\dir{<}?;
 (6,6)*{};(-2,6)*{} **[black][|(1)]\crv{(6,2) & (-2,2)} ?(.25)*[black][|(1)]\dir{<}?;
 (-6,9)*{\scriptstyle k+1};
 (10,9)*{\scriptstyle k+1};
 %(-5,-9)*{\scriptstyle k};
 %(8,-9)*{\scriptstyle k};
% (6,0)*{\bigb{e_{j+1}}};
  (6,-4)*{};(-2,-4)*{} **[black][|(.8)]\crv{(6,-8) & (-2,-8)} ?(.25)*[black][|(1)]\dir{<}?;
%(6,-9)*{\bullet}+(3,2)*{e_{s}};
%(6,-11)*{\bigb{e_{s}}};
%(12,-3)*{n};
\endxy
\;\;=\;\;0
\quad\text{and}\quad
\xy
 (-2,8);(-2,-12); **[black][|(2)]\dir{-} ?(.55)*[black][|(1)]\dir{<}?;
 (6,-12);(6,8); **[black][|(2)]\dir{-} ?(0)*[black][|(1)]\dir{<}?;
 (6,6)*{};(-2,6)*{} **[black][|(.8)]\crv{(6,2) & (-2,2)} ?(0.25)*[black][|(1)]\dir{<}?;
 (-6,9)*{\scriptstyle k+1};
 (10,9)*{\scriptstyle k+1};
 %(-5,-9)*{\scriptstyle k};
 %(8,-9)*{\scriptstyle k};
 (6,-1)*{\bullet} +(4,2)*{e_{t+1}};
  (6,-4)*{};(-2,-4)*{} **[black][|(.8)]\crv{(6,-8) & (-2,-8)} ?(.25)*[black][|(1)]\dir{<}?;
(6,-9)*{\bullet}+(3,2)*{e_{s}};
(8,-13)*{};
\endxy\;\;
-
\xy
 (-2,8);(-2,-12); **[black][|(2)]\dir{-} ?(.55)*[black][|(1)]\dir{<}?;
 (6,-12);(6,8); **[black][|(2)]\dir{-} ?(0)*[black][|(1)]\dir{<}?;
 (6,6)*{};(-2,6)*{} **[black][|(.8)]\crv{(6,2) & (-2,2)} ?(.25)*[black][|(1)]\dir{<}?;
 (-6,9)*{\scriptstyle k+1};
 (10,9)*{\scriptstyle k+1};
 %(-5,-9)*{\scriptstyle k};
 %(8,-9)*{\scriptstyle k};
 %(6,0)*{\bigb{e_{s+1}}};
 (6,-1)*{\bullet} +(4,2)*{e_{s+1}};
  (6,-4)*{};(-2,-4)*{} **[black][|(.8)]\crv{(6,-8) & (-2,-8)} ?(.25)*[black][|(1)]\dir{<}?;
%(6,-11)*{\bigb{e_{t}}};
(6,-9)*{\bullet}+(3,2)*{e_{t}};
%(12,-3)*{n};
\endxy\;\; =\;\;0 \, .
\]
Let us explain this  in more details. 
The formula \eqref{alpha} contains one summand without dots,
let us call it $A$, and the other summands called $B$.
Putting $A$ on top of $d^V_{k-1, l}$ is zero by the first identity
printed above. Now any square resulting from putting $B$
on top of $d^V_{k-1, l}$ gives rise to the second identity
with $t=k-i_j$ and $s=k-i_{j'}$ or $s=0$.

The first identity can be proved as follows.
By using the associativity of the
trivalent vertices in the thick calculus (Proposition 2.4 in \cite{KLMS})
we can attach the second horizontal line to the first one 
\[\xy
 (-2,8);(-2,-4); **[black][|(2)]\dir{-} ?(.85)*[black][|(1)]\dir{<}?;
 (6,-4);(6,8); **[black][|(2)]\dir{-} ?(0)*[black][|(1)]\dir{<}?;
 (6,6)*{};(-2,6)*{} **[black][|(.8)]\crv{(6,2) & (-2,2)};
 ?(.5)*{\cir<5pt>{d_u}};
%(5,5);(-3,5); **\crv{(2,3)}; 
(-6,9)*{\scriptstyle k+1};
 (10,9)*{\scriptstyle k+1};
 %(-5,-9)*{\scriptstyle k};
 %(8,-9)*{\scriptstyle k};
% (6,0)*{\bigb{e_{j+1}}};
 % (6,-4)*{};(-2,-4)*{} **[black][|(.8)]\crv{(6,-8) & (-2,-8)} ?(.25)*[black][|(1)]\dir{<}?;
%(6,-9)*{\bullet}+(3,2)*{e_{s}};
%(6,-11)*{\bigb{e_{s}}};
%(12,-3)*{n};
\endxy
\;\;=\;\;0
\]
and then 
apply eq. (2.70) in \cite{KLMS}. The proof of the second identity 
is similar after sliding $e_{t+1}$ and $e_{s+1}$ down by 
using the comultiplication rule in the ring of symmetric polynomials
(or eq. (2.67) in \cite{KLMS}).

%Further details are left to the reader.
\end{proof}

\subsection{The inverse bicomplex}
The bicomplex $r^{-1}\onen$ is defined as a kind of ``left adjoint''
 of $r\onen$. It has the form
\[ \xymatrix{C^L_{00}
	     &	& &	  \\
      C^L_{10}\ar[u]_{(d^V_{00})^L}&C^L_{11} \ar[l]_{(d^H_{10})^L}&&\\
		 C^L_{20}\ar[u]_{(d^V_{10})^L}&
C^L_{21}\ar[l]_{(d^H_{20})^L} \ar[u]_{(d^V_{11})^L}&
C^L_{22}\ar[l]_{(d^H_{21})^L}&\\
C^L_{30}\ar[u]_{(d^V_{20})^L} &
C^L_{31}\ar[l]_{(d^H_{30})^L}\ar[u]_{(d^V_{21})^L}
 & C^L_{32}\ar[l]_{(d^H_{31})^L}\ar[u]_{(d^V_{22})^L}&
C^L_{33}\ar[l]_{(d^H_{32})^L}\\
\dots\ar[u]_{(d^V_{30})^L} &\dots\ar[u]_{(d^V_{31})^L} &\dots \ar[u]_{(d^V_{32})^L}& \dots
\ar[u]_{(d^V_{33})^L}
}
       \]
%where $C^L_{k,l}$ is the left adjoint of $C_{k,l}$.
%Let us compute them. From \cite{Lau1} we have
%\begin{align*}
%(\E^{(k)}\onen \la s\ra)^L &= \F^{(k)}{\mathbf 1}_{n+2k}\la -kn-k^2-s\ra =
%(\E^{(k)}\onen \la s\ra)^R\\
%(\F^{(k)}{\mathbf 1}_{n+2k} \la s\ra)^L &= \E^{(k)}\onen\la kn+ k^2-s\ra =
%(\F^{(k)}{\mathbf 1}_{n+2k} \la s\ra)^R \, .
%\end{align*}

where
 $$C^L_{k,l}= \F^{(k)} \E^{(k)}\onen
\la kn+k\ra\otimes \Lambda^l {\bar W}_k = C^R_{k,l}
$$
with the total $q$-degree shift $\la\frac{n^2}{2}+n\ra$
and the homological shift $[-n/2,-n/2]$.
Here ${\bar W}_k=\Span_\Z\{\bar w_1,\dots, \bar w_k\}   $ and
 ${\rm deg}(\bar w_j)=2j$.

The  differentials are obtained by rotating
the original differentials by $\pi$. We have
\begin{align*}
( d^H_{k,l})^L: C^L_{k,l+1}&\to C^L_{k,l}\\ 
\bar w_{i_1}\wedge\dots\wedge \bar w_{i_{l+1}}
& \mapsto \sum^l_{j=1} (-1)^{j-1} c^L_j \otimes  \bar w_{i_1}\wedge 
\hat {\bar w}_{i_j} \wedge {\bar w}_{i_{l+1}}\end{align*}
where 
 $(c^L_d)_{\F^{(k)}\E^{(k)}\onen}=\sum^d_{j=0}(-1)^j  h_{d-j}\otimes e_j$.
The vertical differential is
\[( d^V_{k,l})^L:= (-1)^l 
\xy
 (-2,8);(-2,-8); **[black][|(2)]\dir{-} ?(0)*[black][|(1)]\dir{<}?;
 (6,-8);(6,8); **[black][|(2)]\dir{-} ?(0)*[black][|(1)]\dir{<}?;
 (6,-6)*{};(-2,-6)*{} **[black][|(1)]\crv{(6,-2) & (-2,-2)} ?(.5)*[black][|(1)]\dir{>}?;
 (-6,-9)*{\scriptstyle k+1};
 (10,-9)*{\scriptstyle k+1};
 (-5,9)*{\scriptstyle k};
 (8,9)*{\scriptstyle k};
 (-2,3)*{\bigb{\alpha_{k,l}}};
% (12,-3)*{n};
\endxy\; :  C^L_{k+1,l}\to C^L_{k,l}\, \quad 
\]
with $\alpha_{k,l}$ defined as before by \eqref{alpha}.

%It is easy to theck that $(r \onen)^R=(r\onen)^L=r^{-1}\onen$.

\subsection{Intermediate bicomplex}
Recall that  $\rib\E\onen$ is the bicomplex obtained by composing 
 $\rib\onenn{n+2}$ to the left of  $\E\onen$. Its ``chain groups'' are $ 
\F^{(k)}\E^{(k)} \E \onen\la-kn-3k\ra  \Lambda^l W_k$
with the total $q$-degree shift $\la -\frac{n^2}{2}-3n-4\ra$
and homological shift $[n/2+1,n/2+1]$. We will
 denote them by $C_{k,l}\E$.
The differentials are those
of $\rib\onenn{n+2}$ extended by identity on $\E\onen$.
 The Euler characteristic of $\rib\E\onen$ is
\begin{equation}
{\rm r}E1_n= q^{-\frac{n^2}{2}-3n-4} \sum^\infty_{k=0} (-1)^k
q^{-kn-2k}(q^{-4};q^{-2})_k F^{(k)} E^{(k+1)}1_n \, .
\label{Euler}
\end{equation}

The intermediate bicomplex 
$$\Int=(\oplus_{k,l} C'_{k,l}, d'^H_{k,l}, d'^V_{k,l})$$
is an indecomposable 
summand of $\rib\E\onen$ defined as follows.
The ``chain groups'' are
$$C'_{k,l}=\F^{(k)}\E^{(k+1)}  \onen \la -kn-2k\ra\otimes \Lambda^l W'_k\, $$
with $W'_k:=\Span_\Z\{w_2, \dots, w_{k+1}\}$ and the same total shifts
as for $\rib\E\onen$.

The horizontal differential $d'^H_{k,l}$ sends $x$ to $c'\wedge x$ where
\[c':= \sum^{k+1}_{j=2} c_{j-1,1}  w_j :=
\sum^{k+1}_{j=2}  \left( \sum_{\lambda,\mu\subset (j-1,1)} (-1)^{|\lambda|}
 N^{(j-1,1)}_{\lambda\, \mu} \;\;\;
\xy
 (-8,6);(-8,-6); **[black][|(2)]\dir{-} ?(0)*[black][|(1)]\dir{<};
 (0,6);(0,-6); **[black][|(2)]\dir{-} ?(1)*[black][|(1)]\dir{>};
% (6,4)*{n};
 (0,7)*{\scriptstyle {k+1}};
 (-10,7)*{\scriptstyle {k}};
 (8,0)*{};
 (-8,0)*{};
(0,0)*{\bigb{s_{\mu}}};
(-8,0)*{\bigb{s_{\lambda^t}}};
\endxy\right)  w_j \in \Dot (\F^{(k)}\E^{(k+1)}\onen)
  W'_k \; .
\]
%where $N^\nu_{\lambda\mu}$ are the Littlewood-Richardson coefficients.

The vertical differential $d'^V_{k,l}: C'_{k,l}\to C'_{k+1,l}$ is
defined as follows 
\[ d'^V_{k,l}:= (-1)^l 
\xy
 (-2,8);(-2,-8); **[black][|(2)]\dir{-} ?(.55)*[black][|(1)]\dir{<}?;
 (6,-8);(6,8); **[black][|(2)]\dir{-} ?(0)*[black][|(1)]\dir{<}?;
 (6,6)*{};(-2,6)*{} **[black][|(1)]\crv{(6,2) & (-2,2)} ?(.25)*[black][|(1)]\dir{<}?;
 (-6,9)*{\scriptstyle k+1};
 (10,9)*{\scriptstyle k+2};
 (-5,-9)*{\scriptstyle k};
 (10,-9)*{\scriptstyle k+1};
 (6,-2)*{\bigb{\alpha'_{k,l}}};
% (12,-3)*{n};
\endxy
\]
where 
\begin{equation*}
\alpha'_{k,l}:\Lambda^l W'_k \to 
\Dot\left( \E^{(k+1)}\onen \right)
  \Lambda^l W'_{k+1}
\end{equation*}
is defined  similar to \eqref{alpha}. For
$1<i_1<i_2<\dots<i_l\leq k+1$,
we have
\begin{equation} \label{alpha'} 
\alpha'_{k,l}(w_{i_1}\wedge w_{i_2}\wedge \dots\wedge w_{i_l})=
 w_{i_1}\wedge \dots \wedge w_{i_l}-
\sum^l_{j=1} (-1)^{k-i_j+l-j}\;
\xy
% (-10,6);(-10,-6); **[black][|(2)]\dir{-} ?(0)*[black][|(1)]\dir{<};
 (0,6);(0,-6); **[black][|(2)]\dir{-} ?(1)*[black][|(1)]\dir{>};
(2,5)*{\scriptstyle {k}};
 (-7,5)*{\scriptstyle {}};
(-0,0)*{\bigb{e_{k+2-i_j}}};
\endxy  \quad\;
   w_{i_i}\wedge \dots \hat w_{i_j}\dots \wedge w_{i_l}
\wedge w_{k+2}
\end{equation}

It can be verified similarly to the previous case that $\Int$ is
a bicomplex. The identity which replaces \eqref{d21} here is
\[
\sum^r_{j=0} (-1)^j c_{r+1-j,1}(e_j)_2 =0 \in \Dot(\F^{(r)}\E^{(r+1)}\onen) \, .
\]
It can be proven  by using the comultiplication of Schur
functions implying that
$$\Delta (c_{k,1})=
\sum^{k-1}_{i=0} c_i\otimes c_{k-i,1}+\sum^k_{i=1} c_i \otimes c_{k+1-i}
+\sum^{k+1}_{i=2} c_{i-1,1} \otimes c_{k+1-i} \, .
$$

\section{Chain maps between $\rib \E\onen$ and $\Int$}
In this section we will define the chain maps between
$\rib\E\onen$ and the intermediate complexes.

The map $f=\oplus_{k,l} f_{k,l}: \rib\E\onen \to \Int$ 
%with $f_{k,l}: 
is defined by
\[ f_{k,l}:=\;
\xy (-8,-10);(-8,8) **[black][|(2)]\dir{-} ?(.6)*[black][|(1)]\dir{<};
 (0,0);(0,8) **[black][|(2)]\dir{-} ?(.6)*[black][|(1)]\dir{>};
 (-4,-10);(0,0)*{} **[black][|(2)]\crv{(-4,-1)} ?(.3)*[black][|(1)]\dir{>};
 (4,-10);(0,0)*{} **[black][|(1)]\crv{(4,-1)} ?(.3)*[black][|(1)]\dir{>};
 (-10,6)*{\scriptstyle k};
% (7,2)*{n};
 (4,6)*{\scriptstyle k+1};
 (-8,0)*{}; (8,0)*{};
(-4,-5)*{\bigb{\;\;\;\quad\beta_{k,l}\quad\;\;\;}};
\endxy
\;\;:C_{k,l}\E\to C'_{k,l}
\]
where $\beta_{k,l}:\Lambda^l W_k\to
\Dot (\F^{(k)}\E^{(k)}\E\onen)
  \Lambda^l W'_{k}\, $ will be specified below.
Recall that $W'_k=\Span_\Z\{w_2, \dots, w_{k+1}\}$.

Let us first
define a linear map
$$\beta_k:W_k\to \Dot (\F^{(k)}\E^{(k)}\E\onen)  W'_{k}\, $$
by
$$\beta_k(w_i)=-w_i + \delta_{i,1} \sum^{k+1}_{j=2} c_{j-1} w_j
+(-1)^{k+1-i}\;
\xy
 (-8,6);(-8,-6); **[black][|(2)]\dir{-} ?(0)*[black][|(1)]\dir{<};
 (0,6);(0,-6); **[black][|(2)]\dir{-} ?(1)*[black][|(1)]\dir{>};
(8,6);(8,-6); **[black][|(1)]\dir{-} ?(1)*[black][|(1)]\dir{>};
(-6,7)*{\scriptstyle {k}};
 (2,7)*{\scriptstyle {k}};
(0,0)*{\bigb{e_{k+1-i}}}
\endxy \quad w_{k+1}$$
where  $(c_d)_{\F^{(k)}\E^{(k)}\E\onen}=\sum^d_{i=0}\sum^{d-i}_{j=0}
(-1)^i e_i\circ h_{j}\circ h_{d-i-j}$ and $\delta_{i,j}$ 
is the Kronecker delta-function. 

In the matrix form, this map can be represented as follows:
\[
M(\beta_k)=\begin{pmatrix} c_1&-{\bf 1}& \dots& 0&0\\
                         c_2& 0&\dots& 0&0\\
                        \dots&\dots&\dots&\dots&\dots\\
                        c_{k-1}& 0&\dots&0&-{\bf 1}\\
                         c_k+(-1)^k (e_k)_2& (-1)^{k-1}(e_{k-1})_2&\dots&(e_2)_2&-(e_1)_2
\end{pmatrix}
\]
where as before $ (e_i)_2$ 
means that  $e_i$  sits on the second strand.

The map $\beta_k$ extends to an algebra homomorphism
$$\beta_k:\Lambda^\bullet W_k\to 
\Dot (\F^{(k)}\E^{(k)}\E\onen)
  \Lambda^\bullet W'_{k}\, $$
which in the matrix form can be written
as $\beta_{k}=\oplus^k_{l=0} \beta_{k,l}$. This defines the
matrix $\beta_{k,l}$ as the matrix of $(l,l)$-minors of the matrix 
$\beta_{k,1}=M(\beta_k)$ or, alternatively,
$$\beta_{k,l}(w_{i_1}\wedge w_{i_2}\wedge \dots \wedge w_{i_l})=
\beta_{k}(w_{i_1})\wedge \beta_{k}(w_{i_2})\wedge \dots\wedge
\beta_{k}(w_{i_l})\, .$$
We set $\beta_{k,0}=1$.

For example, with the same notation as before 
\[\beta_{3,2}=\begin{pmatrix}
c_2 &-c_1&\bf 1\\
c_1(e_2)_2+c_3-(e_3)_2& -c_1(e_1)_2&(e_1)_2\\
c_2(e_2)_2& -c_2(e_1)_2+c_3-(e_3)_2& (e_2)_2
\end{pmatrix}
\]
where each entry of this matrix is a determinant of the corresponding
$2\times 2$ matrix of $\beta_{3,1}$.

One can easily prove that
\begin{equation}
\label{detbeta}
\det(\beta_{k,1})=
\sum^k_{j=0}
(-1)^j \;\;
\xy
 (-6,8);(-6,-8); **[black][|(2)]\dir{-} ?(0)*[black][|(1)]\dir{<};
 (0,8);(0,-8); **[black][|(2)]\dir{-} ?(1)*[black][|(1)]\dir{>};
(6,8);(6,-8); **[black][|(1)]\dir{-} ?(1)*[black][|(1)]\dir{>};
(0,-4)*{\bigb{e_{j}}};
(0,3)*{\bigb{\quad c_{k-j}\quad}};
\endxy\;\,=\,
\sum^k_{j=0}
(-1)^j \quad
\xy
 (-6,8);(-6,-8); **[black][|(2)]\dir{-} ?(0)*[black][|(1)]\dir{<};
 (0,8);(0,-8); **[black][|(2)]\dir{-} ?(1)*[black][|(1)]\dir{>};
(6,8);(6,-8); **[black][|(1)]\dir{-} ?(1)*[black][|(1)]\dir{>};
(-6,0)*{\bigb{e_{j}}};
(6,0)*{\bigb{ h_{k-j}}};
\endxy
\end{equation}

The inverse map $\bar f=\oplus_{k,l} \bar f_{k,l}: 
\Int\to \rib\E\onen$ is defined by
\[ \bar f_{k,l}=(-1)^k\quad\quad
\xy
 (-8,-8);(-8,10) **[black][|(2)]\dir{-} ?(0)*[black][|(1)]\dir{<};
% (-8,-8);(-8,10) **[|(2)]\dir{-} ?(0)*\dir{<};
 (0,-8);(0,0) **[black][|(2)]\dir{-} ?(.6)*[black][|(1)]\dir{>};
 (0,0);(-4,10)*{} **[black][|(2)]\crv{(-4,-1)} ?(1)*[black][|(1)]\dir{>};
 (0,0);(4,10)*{} **[black][|(1)]\crv{(4,-1)} ?(1)*[black][|(1)]\dir{>};
 (-6,-6)*{\scriptstyle k};
% (8,2)*{n};
 (4,-6)*{\scriptstyle k+1};
 (-8,0)*{}; (8,0)*{};
(-4,5)*{\bigb{\;\;\quad\bar \beta_{k,l}\quad\;\;\;}};
\endxy\]
where $$\bar\beta_{k,l}:= {\rm Adj}(\beta_{k,l})=
\det(\beta_{k,1}) \beta^{-1}_{k,l}$$ 
is given by the
adjugate matrix, or the transpose of the cofactor  matrix 
 for $\beta_{k,l}$.

For example, 
$$ \beta_{3,1}=\begin{pmatrix}
c_1 &-{\bf 1}& 0\\
c_2 & 0& -{\bf 1}\\
c_3- (e_3)_2& ( e_2)_2& -(e_1)_2
\end{pmatrix}
$$
and
$$\bar \beta_{3,1}={\rm Adj}( \beta_{3,1})=\begin{pmatrix}
(e_2)_2 &-(e_1)_2& {\bf 1}\\
c_2 (e_1)_2-c_3+(e_3)_2& -c_1 (e_1)_2& c_1\\
c_2 (e_2)_2& -c_1 (e_2)_2-c_3+(e_3)_2& c_2
\end{pmatrix}
$$

\begin{prop}\label{4.1}
The map $f \bar f: \Int\to\Int$ is equal to identity.
\end{prop}
\begin{proof}
We have
\[
f_{k,l}\bar f_{k,l}=  (-1)^k \quad
\xy
 (-9,-11);(-9,11) **[black][|(2)]\dir{-} ?(0)*[black][|(1)]\dir{<};
 (0,6);(0,11) **[black][|(2)]\dir{-}?(1)*[black][|(1)]\dir{>};
 (0,-6);(0,-11) **[black][|(2)]\dir{-};
 (-6,-8)*{\scriptstyle k};
 (0,-6);(0,6)*{} **[black][|(2)]\crv{(-6,-4) & (-6,4)}?(.5)*{};
 % (0,3)*{\bigb{\quad\beta_{k,l}\quad}};
% (0,-3)*{\bigb{\quad\bar\beta_{k,l}\quad}};
 (0,-6);(0,6)*{} **\crv{(6,-4) & (6,4)}?(.5)*{};
(-3,0)*{\bigb{\;\;\det \beta_{k,1}\;\;}}; 
% (7,7)*{n};
 (4,-8)*{\scriptstyle k+1};
\endxy \;=\;
\sum^k_{j=0} (-1)^{k+j}\quad
\xy
 (-9,-11);(-9,11) **[black][|(2)]\dir{-} ?(0)*[black][|(1)]\dir{<};
 (0,6);(0,11) **[black][|(2)]\dir{-}?(1)*[black][|(1)]\dir{>};
 (0,-6);(0,-11) **[black][|(2)]\dir{-};
 (-6,-8)*{\scriptstyle k};
 (0,-6);(0,6)*{} **[black][|(2)]\crv{(-6,-4) & (-6,4)}?(.5)*{};
 (-9,0)*{\bigb{e_j}};
% (0,3)*{\bigb{\quad\beta_{k,l}\quad}};
% (0,-3)*{\bigb{\quad\bar\beta_{k,l}\quad}};
 (0,-6);(0,6)*{} **\crv{(6,-4) & (6,4)}?(.5)*{\bigb{h_{k-j}}};
 (8,1)*{};
 (8,-1)*{};
% (7,7)*{n};
 (4,-8)*{\scriptstyle k+1};
\endxy \;\;\; = \;\;
\xy
 (-6,-10);(-6,10); **[black][|(2)]\dir{-} ?(0)*[black][|(1)]\dir{<};
 (0,-10);(0,10); **[black][|(2)]\dir{-} ?(1)*[black][|(1)]\dir{>};
(-4,-8)*{\scriptstyle {k}};
 (4,-8)*{\scriptstyle {k+1}};
%(3,7)*{n};
\endxy
\]

\end{proof}

\begin{prop}\label{4.2}
The map $f: \rib \E\onen \to \Int $
%and $\bar f:\Int \to \rib\E\onen$
is a  chain map between bicomplexes.
 
\end{prop}

\begin{proof}
Let us first check the commutativity of the horizontal square:
\[ \xymatrix{C'_{k,l}\ar[r]^{c'\wedge} &C'_{k,l+1}
	    	  \\
      C_{k,l}\E\ar[r]^{(c)_{12}\wedge}\ar[u]^{f_{k,l}}& C_{k,l+1}\E\ar[u]^{f_{k,l+1}}
		}
       \]
%We will write down the proof  for $l=1$,
%the other cases are similar and hence will be left to the reader. 
We have to show  that for any $1\leq i_1<i_2<\dots<i_l\leq k$
$$  \sum^{k+1}_{i=2} c_{i-1,1} \, w_i \wedge \beta_k(w_{i_1})
\wedge \dots \wedge \beta_k(w_{i_l})
=\sum^k_{i=1, i\neq i_j} (c_{i})_{12} \,\beta_k(w_i)\wedge \beta_k(w_{i_1})\wedge
\dots \wedge \beta_k(w_{i_l})  \;\;
\in \Dot(\F^{(k)}\E^{(k)}\E)  \Lambda^l  W_k\, . $$ 
%For simplicity, we will check this identity in the case $l=1$.
%The general case is completely analogous

Let us first assume $i_1>1$.
After the substitution of maps this identity can be rewritten as follows:
\begin{align*}
(-1)^{l}\sum^{k+1}_{i=2} c_{i-1,1} w_i\wedge w_{i_1}&
\wedge\dots \wedge w_{i_l} + (-1)^{l-1}\sum^l_{j=1} \sum^{k+1}_{i=2}
(-1)^{k+1-i_j} (e_{k+1-i_j})_2 c_{i-1,1} w_i\wedge\dots \hat w_{i_j}\dots\wedge w_{i_l}\wedge w_{k+1} =\\
(-1)^{l+1}\sum^k_{i=1} (c_i)_{12} w_i \wedge w_{i_1}&
\wedge\dots \wedge w_{i_l}
+(-1)^l\sum^{k+1}_{i=2} (c_1)_{12} c_{i-1} w_i \wedge w_{i_1}
\wedge\dots \wedge w_{i_l} \\& +(-1)^{l-1}\sum^l_{j=1}\sum^{k+1}_{i=2} (-1)^{k+1-i_j} (c_1)_{12} c_{i-1} (e_{k+1-i_j})_2 
w_i \wedge w_{i_1} \wedge\dots \hat w_{i_j}
\dots \wedge w_{i_l}\\ &
+ (-1)^l \sum^l_{j=0}\sum^k_{i_0=1, i_0\neq i_j} (-1)^{k+1-i_j} (c_{i_0})_{12} (e_{k+1-i_j})_2 
 w_{i_0}
\wedge\dots\hat w_{i_j} \dots\wedge w_{i_l}\wedge w_{k+1}
%+(-1)^{k-j} \sum^k_{i=1, i\neq j} (c_i)_{12} (e_{k+1-j})_2 w_i\wedge w_{k+1}
%+(-1)^{k+1-j} \sum^{k+1}_{i=2} (c_1)_{12} c_{i-1} (e_{k+1-j})_2 w_i\wedge w_{k+1}
\end{align*}
After  accomplishing the $j=0$ summand of the last term 
 to zero by using \eqref{d2}
and
applying twice
\begin{equation} \label{horsquare}
 c_{i-1,1} +(c_i)_{12} -(c_1)_{12} c_{i-1}=0
\, \end{equation}
 it is not difficult to verify that it holds.
The case $i_1=1$ reduces to the same identities and is left to the reader.

To see that $f_{k+1,l}(d^V_{k,l}\E)=  d'^V_{k,l} f_{k,l}$ we 
 need to verify  that 
\[
\xy (-8,-10);(-8,6) **[black][|(2)]\dir{-} ?(.6)*[black][|(1)]\dir{<};
 (0,0);(0,6) **[black][|(2)]\dir{-} ?(1)*[black][|(1)]\dir{>};
 (-2,-10);(0,0)*{} **[black][|(2)]\crv{(-2,-1)} ?(.3)*[black][|(1)]\dir{>};
 (2,-10);(0,0)*{} **[black][|(1)]\crv{(2,-1)} ?(.5)*[black][|(1)]\dir{-};
 (-14,6)*{\scriptstyle k+1};
% (7,2)*{n};
 (4,6)*{\scriptstyle k+2};
 (-8,0)*{}; (8,0)*{};
(-4,-5)*{\bigb{\;\;\;\beta_{k+1}\;\;\;}};
(-8,-10);(-8,-20); **[black][|(2)]\dir{-};
 (-2,-10);(-2,-20); **[black][|(2)]\dir{-} ?(.85)*[black][|(1)]\dir{-}?;
(2,-10);(2,-20); **[black][|(1)]\dir{-} ?(.9)*[black][|(1)]\dir{>}?;
 (-8,-9)*{};(-2,-9)*{} **[black][|(1)]\crv{(-8,-13) & (-2,-13)};
% (-6,9)*{\scriptstyle k+1};
% (10,9)*{\scriptstyle k+1};
% (-12,-22)*{\scriptstyle k};
% (-6,-22)*{\scriptstyle k};
 (-2,-16)*{\bigb{\alpha_{k}}};
% (12,-3)*{n};
\endxy \quad=\quad
\xy (-8,-20);(-8,-6) **[black][|(2)]\dir{-} ?(.6)*[black][|(1)]\dir{<};
 (0,-10);(0,-6) **[black][|(2)]\dir{-} ?(.6)*[black][|(1)]\dir{-};
 (-2,-20);(0,-10)*{} **[black][|(2)]\crv{(-2,-15)} ?(.3)*[black][|(1)]\dir{>};
 (2,-20);(0,-10)*{} **[black][|(1)]\crv{(2,-15)} ?(.3)*[black][|(1)]\dir{-};
% (-10,-12)*{\scriptstyle k};
% (7,2)*{n};
% (4,6)*{\scriptstyle k+1};
% (-8,0)*{}; (8,0)*{};
(-4,-16)*{\bigb{\;\;\;\;\beta_{k}\;\;\;\;}};
(-8,-6);(-8,6); **[black][|(2)]\dir{-};
 (0,-6);(0,6); **[black][|(2)]\dir{-} ?(0)*[black][|(1)]\dir{<}?;
%(4,-10);(4,6); **[black][|(1)]\dir{-} ?(.5)*[black][|(1)]\dir{>}?;
 (-8,4)*{};(0,4)*{} **[black][|(1)]\crv{(-8,0) & (0,0)};
 (-14,6)*{\scriptstyle k+1};
 (4,6)*{\scriptstyle k+2};
% (-12,-22)*{\scriptstyle k};
% (-6,-22)*{\scriptstyle k};
 (0,-4)*{\bigb{\alpha'_{k}}};
% (12,-3)*{n};
\endxy \quad :
 W_k \to \Dot (\F^{(k)}\E^{(k+1)})  W'_{k+1}\]
which can be easily seen  after the substitution of maps. 
 Hence, also the induced maps
$\Lambda^l W_k \to \Dot (\F^{(k)}\E^{(k+1)}) \Lambda^l W'_{k+1}$
have to coincide.
\end{proof}

\begin{thm}\label{re}
The bicomplex $\Int$  is a strong deformation retract of $\rib\E\onen$
in $\Kom(\UcatD)$.
\end{thm}

\section{Proof of Theorem \ref{re}}
This section is devoted to the proof of Theorem \ref{re}.
After defining the homotopies, Propositions \ref{pro-re},
\ref{4.1} and \ref{4.2}
establish all properties of the strong deformation retract.

\subsection{Homotopies}
The horizontal homotopy $h^H_{k,l}: C_{k,l}\E \to C_{k,l-1}\E$
is defined as follows:
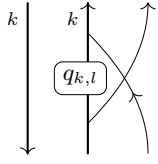
\[
\xy
 (-8,10);(-8,-10); **[black][|(2)]\dir{-} ?(0)*[black][|(1)]\dir{<};
 (0,10);(0,-10); **[black][|(2)]\dir{-} ?(1)*[black][|(1)]\dir{>};
(8,10);(0,-6); **\crv{(8,2) } ?(1)*[black][|(1)]\dir{>};
(8,-10);(0,6); **\crv{(8,-2) } ?(.5)*[black][|(1)]\dir{<};
(-2,8)*{\scriptstyle {k}};
 (-10,8)*{\scriptstyle {k}};
(-1,0)*{\bigb{ q_{k,l}}};
\endxy
\]
where
$ q_{k,l}:\Lambda^l W_k \to \Dot(\E^{(k-1)}\onen)
  \Lambda^{l-1} W_{k}$
\[q_{k,l}(w_{i_1}\wedge \dots \wedge w_{i_l})=\sum^l_{j=1}(-1)^{j-1} 
q_{k,1}(w_{i_j})
w_{i_1}\wedge \dots \hat w_{i_j}\dots \wedge w_{i_l}
\]
with $q_{k,1}(w_i)=(-1)^{i-1}e_{k-i}$.
The vertical homotopy  $h^V_{k,l}:C_{k,l}\E$ to $C_{k-1,l}\E$ is set to be zero.

Now Theorem \ref{re} reduces to the following proposition.
\begin{prop}\label{pro-re} We have
\begin{enumerate} 
\item
$ f_{k,l-1} { h}^H_{k,l}=0$;
\item
${ h}^H_{k,l-1}   { h}^H_{k,l}=0$;
\item
$  h^H_{k+1,l} (d^V_{k,l}\E) + (d^V_{k,l-1} \E)h^H_{k,l}=0$;
\item
$ h^H_{k,l+1} (d^H_{k,l}\E)  + (d^H_{k,l-1}\E)
 h^H_{k,l} ={\bf 1} - \bar f_{k,l} f_{k,l}$.
\end{enumerate}
\end{prop}

\begin{proof}
%The proof of the first 3 identities is similar to those of Lemma 9.4.
$(1)$ The proof of the first identity requires to show that for any
$1\leq i,j\leq k$, we have
\[\xy
 (0,8);(0,4); **[black][|(2)]\dir{-}?(1)*[black][|(1)]\dir{>};
(0,4);(-4,-8); **[black][|(2)]\crv{ (-4,1)}; 
?(0.8)*[black][|(1)]\dir{>} ?(0.4)*{\bullet}+(-4,-1)*{\scs e_{k-i}};
(0,4);(-4,-5); **[black][|(1)]\crv{ (8,-2)}; 
%?(0.3)*[black][|(1)]\dir{>};
(-3,0);(4,-8); **[black][|(1)]\crv{ (4,-2)} ?(0.2)*[black][|(1)]\dir{>};
 (2,8)*{\scriptstyle k};
\endxy\;\;=0
\quad \quad\text{and}\quad\quad
\xy
 (0,8);(0,4); **[black][|(2)]\dir{-}?(1)*[black][|(1)]\dir{>};
(0,4);(-4,-8); **[black][|(2)]\crv{ (-4,1)}; 
?(0.1)*[black][|(1)]\dir{>} ?(0.4)*{\bullet}+(-4,-1)*{\scs e_{k-i}}
?(0.7)*{\bullet}+(-6,0)*{\scs e_{k+1-j}};
(0,4);(-4,-5); **[black][|(1)]\crv{ (8,-2)}; 
%?(0.3)*[black][|(1)]\dir{>};
(-3,0);(4,-8); **[black][|(1)]\crv{ (4,-2)} ?(0.2)*[black][|(1)]\dir{>};
(2,8)*{\scriptstyle k};
\endxy\;\;-\;\;
\xy
 (0,8);(0,4); **[black][|(2)]\dir{-}?(1)*[black][|(1)]\dir{>};
(0,4);(-4,-8); **[black][|(2)]\crv{ (-4,1)}; 
?(0.1)*[black][|(1)]\dir{>} ?(0.4)*{\bullet}+(-4,-1)*{\scs e_{k-j}}
?(0.7)*{\bullet}+(-6,0)*{\scs e_{k+1-i}};;
(0,4);(-4,-5); **[black][|(1)]\crv{ (8,-2)}; 
%?(0.3)*[black][|(1)]\dir{>};
(-3,0);(4,-8); **[black][|(1)]\crv{ (4,-2)} ?(0.2)*[black][|(1)]\dir{>};
(2,8)*{\scriptstyle k};
\endxy\;\;=0
\]
To prove the first statement we use the associativity of the trivalent vertices.
%by induction on $j=k-i<k$. The inductive step
%goes as follows: 
$$\xy
 (0,8);(0,4); **[black][|(2)]\dir{-}?(1)*[black][|(1)]\dir{>};
(0,4);(-4,-8); **[black][|(2)]\crv{ (-4,1)}; 
?(0.8)*[black][|(1)]\dir{>} ?(0.4)*{\bullet}+(-2,-1)*{\scs e_j};
(0,4);(-4,-5); **[black][|(1)]\crv{ (8,-2)}; 
%?(0.3)*[black][|(1)]\dir{>};
(-3,0);(4,-8); **[black][|(1)]\crv{ (4,-2)} ?(0.2)*[black][|(1)]\dir{>};
\endxy\;\;=\;\;
\xy
 (0,8);(0,4); **[black][|(2)]\dir{-}?(1)*[black][|(1)]\dir{>};
(0,4);(-4,-8); **[black][|(2)]\crv{ (-4,1)}; 
?(0.1)*[black][|(1)]\dir{>} ?(0.4)*{\bullet}+(-2,0)*{\scs e_j};
(0,4);(-4,-5); **[black][|(1)]\crv{ (8,-2)}; 
%?(0.3)*[black][|(1)]\dir{>};
(3,0);(4,-8); **[black][|(1)]\crv{ (-4,-2)} ?(0.2)*[black][|(1)]\dir{>};
\endxy\;\;=\;\;0
$$
%where the last term is zero by Propositions 2.5, 2.9 in \cite{KLMS},
%and the second is zero after sliding down the dot  by the induction hypothesis,
%the smoothing term after sliding is zero again by Proposition 2.9.

The second statement works similarly. We first move  dots 
down, then cancel terms that coincide and finally use the same trick.  
%by induction on the total number
%of dots. The inductive step is as follows:

(2) Here we need to show that
for all $0\leq s,u\leq k-1$ 
\[
\begin{DGCpicture}
%[scale={.7,.7}]
%\DGCstrand[d](0,0.2)(0,2)\DGCdot <{0}\DGCdot.{2.2}[l]{$\scs k$}
\DGCstrand[d](0,0.2)(0,2.2)\DGCdot >{2.2}\DGCdot.{2.2}[l]{$\scs k$}
\DGCdot{1.7}[l]{$ e_s$} \DGCdot{.6}[l]{$ e_u$}
\DGCstrand(0,0.8)/ur/(1,2.2)/u/\DGCdot >{2.2}
\DGCstrand(0,1.4)/dr/(1,0.2)/d/\DGCdot <{1}
\DGCstrand(0,0.4)/r/(1,1.2)(0,2)/l/
\end{DGCpicture}
-
\begin{DGCpicture}
%[scale={.7,.7}]
\DGCstrand[d](0,0.2)(0,2.2)\DGCdot >{2.2}\DGCdot.{2.2}[l]{$\scs k$}
\DGCdot{1.7}[l]{$ e_u$} \DGCdot{.6}[l]{$ e_s$}
\DGCstrand(0,0.8)/ur/(1,2.2)/u/\DGCdot >{2.2}
\DGCstrand(0,1.4)/dr/(1,0.2)/d/\DGCdot <{1}
\DGCstrand(0,0.4)/r/(1,1.2)(0,2)/l/
\end{DGCpicture}=0 \; .
\]
After sliding the upper $e_i$'s down and using the 
invariance under the third Reidemeister move, the left hand side of this
identity become
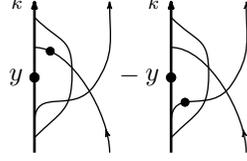
\[
\begin{DGCpicture}
%\DGCstrand[black](0,0)(0,2.4)\DGCdot <{0}\DGCdot.{2.4}[l]{$\scs k$}
\DGCstrand[d](0,.2)(0,2.2)\DGCdot >{2.2}\DGCdot.{2.2}[l]{$\scs k$}
\DGCdot{1.2}[l]{$ y$}
\DGCstrand(0,0.6)(0.8,1.1)(1,2.2)/u/\DGCdot >{2.2}
\DGCstrand(0,1.6)/r/(1,0.2)/d/\DGCdot <{0.5}\DGCdot{1.55}
\DGCstrand(0,0.4)/ur/(.5,1.25)(0,2)/ul/
\end{DGCpicture}
-
\begin{DGCpicture}
\DGCstrand[d](0,.2)(0,2.2)\DGCdot >{2.2}\DGCdot.{2.2}[l]{$\scs k$}
\DGCdot{1.2}[l]{$ y$}
\DGCstrand(0,0.6)(0.8,1.1)(1,2.2)/u/\DGCdot >{2.2}\DGCdot{0.87}
\DGCstrand(0,1.6)/r/(1,0.2)/d/\DGCdot <{0.5}
\DGCstrand(0,0.4)/ur/(.5,1.25)(0,2)/ul/
\end{DGCpicture}
\]
with $y:=e_{s-1}e_u-e_{u-1}e_s$. Using
associativity as in case (1) we can easily see
that it is zero.

(3) The third equation follows
from the following identity:

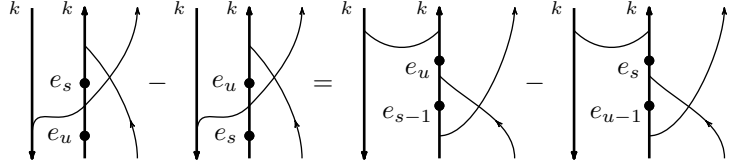
\[
\begin{DGCpicture}
%[scale={.7,.7}]
\DGCstrand[d](0,0)(0,2)\DGCdot <{0}\DGCdot.{2}[l]{$\scs k$}
\DGCstrand[d](.7,0)(.7,2)\DGCdot >{2}\DGCdot.{2}[l]{$\scs k$}
\DGCdot{1}[l]{$ e_s$} \DGCdot{.3}[l]{$ e_u$}
\DGCstrand(0,0.3)(.7,.7)(1.4,2)\DGCdot >{2}
\DGCstrand(.7,1.5)/dr/(1.4,0)/d/\DGCdot <{0.5}
\end{DGCpicture}
-
\begin{DGCpicture}
%[scale={.7,.7}]
\DGCstrand[d](0,0)(0,2)\DGCdot <{0}\DGCdot.{2}[l]{$\scs k$}
\DGCstrand[d](.7,0)(.7,2)\DGCdot >{2}\DGCdot.{2}[l]{$\scs k$}
\DGCdot{1}[l]{$ e_u$} \DGCdot{.3}[l]{$ e_s$}
\DGCstrand(0,0.3)(.7,.7)(1.4,2)\DGCdot >{2}
\DGCstrand(.7,1.5)/dr/(1.4,0)/d/\DGCdot <{0.5}
\end{DGCpicture}
=
\begin{DGCpicture}
%[scale={.7,.7}]
\DGCstrand[d](0,0)(0,2)\DGCdot <{0} \DGCdot.{2}[l]{$\scs k$}
\DGCstrand[d](1,0)(1,2)\DGCdot >{2}\DGCdot.{2}[l]{$\scs k$}
\DGCdot{.7}[dll]{$ e_{s-1}$} \DGCdot{1.3}[dll]{$ e_u$}
\DGCstrand(0,1.7)/dr/(1,1.7)/ur/
\DGCstrand(1,1.1)/dr/(2,0)/d/\DGCdot <{0.5}
\DGCstrand(1,.3)/r/(2,2)\DGCdot >{2}
\end{DGCpicture}
-
\begin{DGCpicture}
%[scale={.7,.7}]
\DGCstrand[d](0,0)(0,2)\DGCdot <{0} \DGCdot.{2}[l]{$\scs k$}
\DGCstrand[d](1,0)(1,2)\DGCdot >{2}\DGCdot.{2}[l]{$\scs k$}
\DGCdot{.7}[dll]{$ e_{u-1}$} \DGCdot{1.3}[dll]{$ e_s$}
\DGCstrand(0,1.7)/dr/(1,1.7)/ur/
\DGCstrand(1,1.1)/dr/(2,0)/d/\DGCdot <{0.5}
\DGCstrand(1,.3)/r/(2,2)\DGCdot >{2}
\end{DGCpicture}
\]
which holds for all $0\leq s,u\leq k$ (here $e_{-1}$ are assumed to be zero).
 This identity is easy to prove by using the sliding
rules from  Appendix.

(4) Let us compute all terms of this equation. The ``$dh$''-part
contains the diagonal  term of the form
\begin{equation}\label{ABterms}
\sum^k_{s=1} (-1)^{s-1}
\begin{DGCpicture}
%[scale={.7,.7}]
\DGCstrand[d](0,0)(0,2)\DGCdot <{0} \DGCdot.{2}[l]{$\scriptstyle k$}
\DGCstrand[d](.7,0)(.7,2)\DGCdot >{2}\DGCdot.{2}[l]{$\scriptstyle k$}
%\DGCdot{.7}[dl]{$\scs e_{u-1}$} 
\DGCdot{1.2}[dl]{$\scriptstyle e_{k-s}$}
%\DGCstrand(0,1.7)/dr/(1,1.7)/ur/
\DGCstrand(.7,1.6)/r/(1.4,0)/d/\DGCdot <{0.5}
\DGCstrand(.7,.9)/r/(1.4,2)\DGCdot >{2}
\DGCcoupon(-.1,.6)(.8,.2){$ c_s$}
\end{DGCpicture} =
(-1)^{k-1}
\begin{DGCpicture}
%[scale={.7,.7}]
\DGCstrand[d](0,0)(0,2)\DGCdot <{0} \DGCdot.{2}[l]{$\scriptstyle k$}
\DGCstrand[d](.7,0)(.7,2)\DGCdot >{2}\DGCdot.{2}[l]{$\scriptstyle k$}
%\DGCdot{.7}[dl]{$\scriptstyle e_{u-1}$} 
%\DGCdot{1.2}[dl]{$\scriptstyle e_{k-s}$}
%\DGCstrand(0,1.7)/dr/(1,1.7)/ur/
\DGCstrand(.7,1.6)/r/(1.4,0)/d/\DGCdot <{0.5}
\DGCstrand(.7,.4)/ur/(1.4,2)\DGCdot >{2}
\DGCcoupon(-.2,1.2)(1.1,.65){$ (c_k)_{13}$}
\end{DGCpicture}
=
\begin{DGCpicture}
%[scale={.7,.7}]
\DGCstrand[d](0,0)(0,2)\DGCdot <{0} \DGCdot.{2}[l]{$\scriptstyle k$}
\DGCstrand[d](.6,0)(.6,2)\DGCdot >{2}\DGCdot.{2}[l]{$\scriptstyle k$}
%\DGCdot{.7}[dl]{$\scriptstyle e_{u-1}$} 
%\DGCdot{1.2}[dl]{$\scriptstyle e_{k-s}$}
%\DGCstrand(0,1.7)/dr/(1,1.7)/ur/
\DGCstrand(1.2,0)(1.2,2)\DGCdot >{2}
%\DGCstrand(1,.4)/ur/(2,2)\DGCdot >{2}
%\DGCcoupon[rsolid](-.1,1.2)(1.4,.7){$ (c_k)_{13}$}
\end{DGCpicture} +
(-1)^{k-1}
\begin{DGCpicture}
%[scale={.7,.7}]
\DGCstrand[d](0,0)(0,2)\DGCdot <{0} \DGCdot.{2}[l]{$\scriptstyle k$}
\DGCstrand[d](.7,0)(.7,2)\DGCdot >{2}\DGCdot.{2}[l]{$\scriptstyle k$}
%\DGCdot{.7}[dl]{$\scriptstyle e_{u-1}$} 
%\DGCdot{1.2}[dl]{$\scriptstyle e_{k-s}$}
%\DGCstrand(0,1.7)/dr/(1,1.7)/ur/
\DGCstrand(.7,1.1)/r/(1.4,0)/d/\DGCdot <{.5}
\DGCstrand(.7,.4)/r/(1.4,2)\DGCdot >{2}
\DGCcoupon(-.2,1.3)(1.6,1.75){$ (c_k)_{13}$}
\end{DGCpicture} \, .
\end{equation}

The off-diagonal terms of ``$dh$'' and ``$hd$'' are 
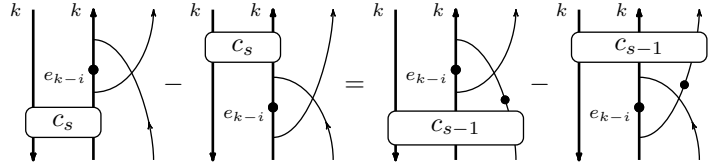
\[\begin{DGCpicture}
\DGCstrand[d](0,0)(0,2)\DGCdot <{0} \DGCdot.{2}[l]{$\scs k$}
\DGCstrand[d](.8,0)(.8,2)\DGCdot >{2}\DGCdot.{2}[l]{$\scs k$}
%\DGCdot{.7}[dl]{$ e_{u-1}$} 
\DGCdot{1.2}[dll]{$\scs e_{k-i}$}
%\DGCstrand(0,1.7)/dr/(1,1.7)/ur/
\DGCstrand(.8,1.6)/r/(1.6,0)/d/\DGCdot <{0.5}
\DGCstrand(.8,.9)/r/(1.6,2)\DGCdot >{2}
\DGCcoupon(-.1,.7)(.9,.3){$ c_s$}
\end{DGCpicture}-
\begin{DGCpicture}
\DGCstrand[d](0,0)(0,2)\DGCdot <{0} \DGCdot.{2}[l]{$\scs k$}
\DGCstrand[d](.8,0)(.8,2)\DGCdot >{2}\DGCdot.{2}[l]{$\scs k$}
%\DGCdot{.7}[dl]{$\scs e_{u-1}$} 
\DGCdot{.7}[dll]{$\scs e_{k-i}$}
%\DGCstrand(0,1.7)/dr/(1,1.7)/ur/
\DGCstrand(.8,1.1)/r/(1.6,0)/d/\DGCdot <{0.5}
\DGCstrand(.8,.3)/r/(1.6,2)\DGCdot >{2} 
%\DGCdot{1}
\DGCcoupon(-.1,1.3)(.9,1.7){$ c_{s}$}
\end{DGCpicture}
=
\begin{DGCpicture}
\DGCstrand[d](0,0)(0,2)\DGCdot <{0} \DGCdot.{2}[l]{$\scs k$}
\DGCstrand[d](.8,0)(.8,2)\DGCdot >{2}\DGCdot.{2}[l]{$\scs k$}
%\DGCdot{.7}[dl]{$\scs e_{u-1}$} 
\DGCdot{1.2}[dll]{$\scs e_{k-i}$}
%\DGCstrand(0,1.7)/dr/(1,1.7)/ur/
\DGCstrand(.8,1.6)/r/(1.6,0)/d/\DGCdot <{0.5}\DGCdot{0.8}
\DGCstrand(.8,.9)/r/(1.6,2)\DGCdot >{2}
\DGCcoupon(-.1,.65)(1.7,.2){$c_{s-1}$}
\end{DGCpicture}-
\begin{DGCpicture}
\DGCstrand[d](0,0)(0,2)\DGCdot <{0} \DGCdot.{2}[l]{$\scs k$}
\DGCstrand[d](.8,0)(.8,2)\DGCdot >{2}\DGCdot.{2}[l]{$\scs k$}
%\DGCdot{.7}[dl]{$\scs e_{u-1}$} 
\DGCdot{.7}[dll]{$\scs e_{k-i}$}
%\DGCstrand(0,1.7)/dr/(1,1.7)/ur/
\DGCstrand(.8,1.1)/r/(1.6,0)/d/\DGCdot <{0.5}
\DGCstrand(.8,.3)/r/(1.6,2)\DGCdot >{2} \DGCdot{1}
\DGCcoupon(-.1,1.3)(1.7,1.7){$ c_{s-1}$}
\end{DGCpicture}
\]
%For each such term we will also have a contribution from the
%``$hd$'' part of the form

Finally, from the chain maps we get
\begin{equation}\label{maps}
\begin{DGCpicture}
\DGCstrand[d](0,-0.5)(0,2)\DGCdot <{0} \DGCdot.{2}[l]{$\scs k$}
\DGCstrand[d](.8,-.5)(.8,2)\DGCdot >{2}\DGCdot.{2}[l]{$\scs k$}
%\DGCdot{.7}[dl]{$\scs e_{u-1}$} 
%\DGCdot{.7}[dl]{$\scs e_{k-i}$}
%\DGCstrand(0,1.7)/dr/(1,1.7)/ur/
\DGCstrand(.8,1)/r/(1.6,-.5)/d/\DGCdot <{0.5}
\DGCstrand(.8,.4)/r/(1.6,2)\DGCdot >{2} 
\DGCcoupon(-.1,1.2)(1.7,1.7){$ \bar \beta_{k,l}$}
\DGCcoupon(-.1,-.3)(1.7,0.2){$  \beta_{k,l}$}
\end{DGCpicture}=
\begin{DGCpicture}
\DGCstrand[d](0,-0.5)(0,2)\DGCdot <{0} \DGCdot.{2}[l]{$\scs k$}
\DGCstrand[d](1,-.5)(1,2)\DGCdot >{2}\DGCdot.{2}[l]{$\scs k$}
%\DGCdot{.7}[dl]{$\scs e_{u-1}$} 
%\DGCdot{.7}[dl]{$\scs e_{k-i}$}
%\DGCstrand(0,1.7)/dr/(1,1.7)/ur/
\DGCstrand(1,1.1)/r/(2,-.5)/d/\DGCdot <{0.5}
\DGCstrand(1,.3)/r/(2,2)\DGCdot >{2} 
\DGCcoupon(-.1,1.3)(2.1,1.8){$ \bar \beta_{k,l} \circ \beta_{k,l}$}
%\DGCcoupon(-.1,-.4)(2.1,0.1){$ f$}
\end{DGCpicture}
+\begin{DGCpicture}
\DGCstrand[d](0,-0.5)(0,2)\DGCdot <{0} \DGCdot.{2}[l]{$\scs k$}
\DGCstrand[d](.8,-.5)(.8,2)\DGCdot >{2}\DGCdot.{2}[l]{$\scs k$}
%\DGCdot{.7}[dl]{$\scs e_{u-1}$} 
%\DGCdot{.7}[dl]{$\scs e_{k-i}$}
%\DGCstrand(0,1.7)/dr/(1,1.7)/ur/
\DGCstrand(.8,1)/r/(1.6,-.5)/d/\DGCdot <{0.5}
\DGCstrand(.8,.4)/r/(1.6,2)\DGCdot >{2} 
\DGCcoupon(-.1,1.2)(1.7,1.7){$ \bar \beta_{k,l}$}
\DGCcoupon(-.1,-.3)(1.7,0.2){$  \beta_{k,l}$}
\end{DGCpicture} -
\begin{DGCpicture}
\DGCstrand[d](0,-0.5)(0,2)\DGCdot <{0} \DGCdot.{2}[l]{$\scs k$}
\DGCstrand[d](1,-.5)(1,2)\DGCdot >{2}\DGCdot.{2}[l]{$\scs k$}
%\DGCdot{.7}[dl]{$\scs e_{u-1}$} 
%\DGCdot{.7}[dl]{$\scs e_{k-i}$}
%\DGCstrand(0,1.7)/dr/(1,1.7)/ur/
\DGCstrand(1,1.1)/r/(2,-.5)/d/\DGCdot <{0.5}
\DGCstrand(1,.3)/r/(2,2)\DGCdot >{2} 
\DGCcoupon(-.1,1.3)(2.1,1.8){$ \bar \beta_{k,l} \circ \beta_{k,l}$}
%\DGCcoupon(-.1,-.4)(2.1,0.1){$ f$}
\end{DGCpicture}\quad .
\end{equation}
Due to  \eqref{detbeta}, the first term in \eqref{maps}
cancels with the last term 
in \eqref{ABterms}
and the remaining terms cancel with the off-diagonal contributions
from ``$dh$'' and ``$hd$''.

Further details are left to the reader.

\end{proof}

\section{Chain maps between  $\E\rib\onen$ and $\Int$}
Let us denote by $\E C_{k,l}$ the ``chain groups'' of the complex
$\E\rib \onen$, obtained by composing $\E\onen$ with $\rib\onen$. 
These groups are
$\E\F^{(k)}\E^{(k)}\onen \otimes \Lambda^l W_k \la -kn-k\ra$ 
with the total degree shift $-\frac{n^2}{2}-3n-4$. 
The differentials are those of $\rib\onen$ extended by identity on
$\E\onen$.
The Euler 
characteristic of this complex is given by \eqref{Euler}.

Define  an algebra homomorphism
$$\gamma_k: W_k \to
\Dot (\E\F^{(k-1)}\E^{(k-1)}\onen)
 \otimes  W'_{k-1}$$
by $\gamma_k(w_1)=-\sum^k_{j=2} c_{j-1} w_j$ and
 $\gamma_k(w_i)=w_i$ for $2\leq i\leq k$. This extends to
$$\gamma_k: \Lambda^l W_k \to
\Dot (\E\F^{(k-1)}\E^{(k-1)}\onen)
 \otimes  \Lambda^l W'_{k-1}\, $$ defined by
$\gamma_k (w_{i_1}\wedge \dots\wedge w_{i_l})= \gamma_k(w_{i_1})\wedge \dots
\wedge \gamma_k
(w_{i_l})$.
We will also need another map
$a_k: W_k \to \Dot (\E\F^{(k-1)}\E^{(k-1)}\onen)$ 
with
$$a_k(w_i)=(-1)^{k-i} (e_{k-i})_3 =(-1)^{k-i} \;\;
\xy
 (-4,6);(-4,-6); **[black][|(2)]\dir{-} ?(0)*[black][|(1)]\dir{<};
 (4,6);(4,-6); **[black][|(2)]\dir{-} ?(1)*[black][|(1)]\dir{>};
(-10,6);(-10,-6); **[black][|(1)]\dir{-} ?(1)*[black][|(1)]\dir{>};
(-1,6)*{\scriptstyle {k-1}};
 (7,6)*{\scriptstyle {k-1}};
(4,0)*{\bigb{e_{k-i}}};
\endxy $$
Let  $\gamma_{k,l}:\Lambda^\bullet W_k
\to \Dot (\E\F^{(k-1)}\E^{(k-1)}\onen)
\otimes \Lambda^{\bullet -1}W'_{k}$ 
 be the derivation along the homomorphism
$\gamma_k$ induced by $a_k$, i.e.  $\gamma_{k,0}(1)=0$ and
for 
$l\geq 0$ we have
$$\gamma_{k,l}(w_{i_1}\wedge \dots \wedge w_{i_l})=
\sum^{l}_{j=1}(-1)^{j-1}
a_k(w_{i_j}) \otimes \gamma_{k}
(w_{i_1}\wedge \dots\hat w_{i_j}\dots \wedge w_{i_l})$$
Then we define the map 
%\to \F^{(k-1)}\E^{(k) } \onen\otimes \Lambda^{l-1}W'_{k-1}$ 
\[ g_{k,l}:=
\xy
 (-4,10);(-4,-14); **[black][|(2)]\dir{-} ?(0)*[black][|(1)]\dir{<};
 (4,10);(4,-14); **[black][|(2)]\dir{-} ?(1)*[black][|(1)]\dir{>};
%(-8,6);(-8,-6); **[black][|(1)]\dir{-} ?(1)*[black][|(1)]\dir{>};
(-8,8)*{\scriptstyle {k-1}};
 (1,8)*{\scriptstyle {k}};
(-6,-12)*{\scriptstyle {k}};
(2,-12)*{\scriptstyle {k}};
(-12,-14);(4,7)*{} **\crv{(-12,3) & (4,3)}?(.1)*\dir{>};
(-4,-10);(4,-10)*{} **\crv{(-4,-6) & (4,-6)}?(.5)*\dir{<};
(-2,-3)*{\bigb{\;\quad\gamma_{k,l}\quad \;}};
\endxy\quad\;\; : \E C_{k,l} \to C'_{k-1, l-1}
\]
and set  $g=\oplus_{k,l} g_{k,l}:\E\rib\onen \to \Int$.
The inverse map 
$p=\oplus_{k,l} p_{k,l}:  \Int\to \E\rib\onen$ 
%\F^{(k-1)}\E^{(k)}\onen \otimes \Lambda^{l-1} W'_{k-1}\to \E\F^{(k)}\E^{(k) } \onen\otimes \Lambda^{l}W_{k}$
is defined as follows:
\[ p_{k,l}:= 
\xy
 (-4,10);(-4,-10); **[black][|(2)]\dir{-} ?(0)*[black][|(1)]\dir{<};
 (4,10);(4,-10); **[black][|(2)]\dir{-} ?(1)*[black][|(1)]\dir{>};
(-12,10);(-12,6); **[black][|(1)]\dir{-} ?(1)*[black][|(1)]\dir{>};
(-6,8)*{\scriptstyle {k}};
 (1,8)*{\scriptstyle {k}};
(-8,-8)*{\scriptstyle {k-1}};
(2,-8)*{\scriptstyle {k}};
(-12,6);(-4,6)*{} **\crv{(-12,2) & (-4,2)}?(.1)*{};
%(-4,-10);(4,-10)*{} **\crv{(-4,-6) & (4,-6)}?(.5)*\dir{<};
(0,-3)*{\bigb{\;\; b_{k,l}\;\;}};
\endxy \quad \;\; : C'_{k-1,l-1}\to \E C_{k,l}
\]
where
$b_{k,l}(w_{i_2}\wedge \dots w_{i_l})= (-1)^{k-1} \sum^k_{i=1} c_{i-1}
\otimes  w_i\wedge w_{i_2}\wedge \dots \wedge w_{i_l}$.

\begin{thm}\label{Er}
%The chain maps $g$ and $p$
% defined above are mutually inverse up to homotopy,
The bicomplex $\Int$ is a strong deformation retract of
 the bicomplex $\E\rib\onen$ 
in $\Kom(\UcatD)$.
\end{thm}

\section{Proof of Theorem \ref{Er}}
 We will split the proof of Theorem \ref{Er}
into lemmas and prove them separately.

\begin{lem} $g: \E r\onen\to \Int$ is a chain map between bicomplexes, i.e.
\begin{align*}
d'^H_{k-1,l-1} g_{k,l} &=  g_{k,l+1} (\E d^H_{k,l})\\
d'^V_{k-1,l-1} g_{k,l} &=  f_{k+1,l} (\E d^V_{k,l})\, .
\end{align*}
\end{lem}

\begin{proof}
Let us check the commutativity of a general horizontal square:
\[ \xymatrix{C'_{k-1,l-1}\ar[r]^{c'\wedge} &C'_{k-1,l}
	    	  \\
     \E C_{k,l}\ar[r]^{(c)_{23}\wedge}\ar[u]^{g_{k,l}}&\E C_{k,l+1}\ar[u]^{g_{k,l+1}}
		}
       \]

We assume $k\geq 2$, otherwise there is nothing to check.
We start with $$\E\F^{(k)}\E^{(k)}  w_{i_1}\wedge \dots \wedge w_{i_l} \in
\E C_{k,l}$$ and first apply the map $g_{k,l}$ followed by  $d'^H_{k-1,l-1}$.
Then we get
\begin{equation}\label{111a}
  \sum^k_{t=2} \sum^l_{j=1} (-1)^{j-1+k-i_j} c_{t-1,1} (e_{k-i_j})_3
w_t \wedge \gamma_k(w_{i_1}\wedge \dots\hat w_{i_j}\dots \wedge w_{i_l}
)\, .
\end{equation}
Applying first the differential and then the map, we get
\begin{equation}\label{222a}
  \sum^l_{j=0}\sum^k_{i_0=1}
 (-1)^{j+k-i_j} (c_{i_0})_{23} (e_{k-i_j})_3
 \gamma_k(w_{i_0}\wedge w_{i_1}\wedge \dots\hat w_{i_j}\dots \wedge w_{i_l}
)\, .
\end{equation}
We claim that these expressions are equal in
$\Dot(\F^{(k-1)}\E^{(k)}\onen)\Lambda^l W'_{k-1}$. Indeed,
assume $1\neq i_1< i_2<\dots < i_l$ and $t\neq i_s$ for all $1\leq s \leq l$.
Then collecting the coefficients in front of
 $w_t \wedge w_{i_1}\wedge \dots\hat w_{i_j}\dots \wedge w_{i_l}$
with $j>1$ in the both formulas we get 
$(-1)^{j+k-i_j} (e_{k-i_j})_3$ times 
\begin{equation}\label{ttt}
 c_{t-1,1} +(c_t)_{23} -(c_1)_{23} c_{t-1}=0
\end{equation}
where the last term comes from setting $i_0=1$ and picking
the $t$-th summand in $\gamma_k(w_1)=-
\sum^k_{j=2}c_{j-1}w_j$.

Allowing $i_1=1$, but $i_j\neq 1$, leads to the same identity. 

Let us consider the case $i_j=1$. Collecting the coefficients  of
$w_t \wedge w_{i_2}\wedge \dots \wedge w_{i_l}$ in \eqref{111a} we get
$$ (-1)^{k-1} c_{t-1,1} (e_{k-1})_3 +\sum^l_{j=2} (-1)^{k-i_j+1} 
(e_{k-i_j})_3 \left( c_{t-1,1} c_{i_j-1} - c_{t-1} c_{i_j-1, 1}\right)
$$
and in \eqref{222a}
$$(-1)^k (c_t)_{23} (e_{k-1})_{3} -\sum^l_{j=2} (-1)^{k-i_j+1}
(e_{k-i_j})_3 (c_t)_{23} c_{i_j-1} -\sum_{p\neq 1, i_2, \dots, i_l} (-1)^{k-p} (e_{k-p})_3 (c_p)_{23} c_{t-1}\, .
$$
Using \eqref{ttt} few times, we can reduce the claim to \eqref{d2}.
%in this case to the following true identity:
%$$\sum^k_{l=1}(-1)^{k-l} (c_l)_{23} (e_{k-l})_3 =0$$
%which is equivalent 

Let us consider the following vertical square
\[ \xymatrix{\E C_{k,l}\ar[r]^{g_{k,l}}\ar[d]^{\E d^V_{k,l}} 
&C'_{k-1,l-1}\ar[d]^{d'^V_{k-1,l-1}}\\
    \E C_{k+1,l}\ar[r]^{g_{k+1,l}}& C'_{k,l-1}		}
       \]
Similar considerations as before lead in all cases to the following true
identity:
$$ 
\xy
 (-4,14);(-4,-10); **[black][|(2)]\dir{-} ?(0)*[black][|(1)]\dir{<};
 (4,14);(4,-10); **[black][|(2)]\dir{-} ?(1)*[black][|(1)]\dir{>};
%(-6,-10);(4,-10); **[black][|(1)]\dir{-} ?(0.5)*[black][|(1)]\dir{<};
(-4,-2);(4,-2)*{} **\crv{(-4,2) & (4,2)}?(.5)*\dir{<};
(-4,-5);(4,-5)*{} **\crv{(-4,-9) & (4,-9)}?(.5)*\dir{>};
(-10,-10);(4,12)*{} **\crv{(-10,6) & (0,6)}?(.2)*\dir{>};
%(-6,8)*{\scriptstyle {k}};
% (2,8)*{\scriptstyle {k}};
(-6,14)*{\scriptstyle {k}};
(6,14)*{\scriptstyle {k}};
(4,6)*{\bigb{ e_j}};
\endxy \;\;-\;\; \xy
 (-4,14);(-4,-10); **[black][|(2)]\dir{-} ?(0)*[black][|(1)]\dir{<};
 (4,14);(4,-10); **[black][|(2)]\dir{-} ?(1)*[black][|(1)]\dir{>};
%(-6,-10);(4,-10); **[black][|(1)]\dir{-} ?(0.5)*[black][|(1)]\dir{<};
(-4,2);(4,2)*{} **\crv{(-4,6) & (4,6)}?(.5)*\dir{<};
(-4,-1);(4,-1)*{} **\crv{(-4,-5) & (4,-5)}?(.5)*\dir{>};
(-10,-10);(4,12)*{} **\crv{(-10,6) & (0,6)}?(.2)*\dir{>};
%(-6,8)*{\scriptstyle {k}};
% (2,8)*{\scriptstyle {k}};
(-6,14)*{\scriptstyle {k}};
(6,14)*{\scriptstyle {k}};
(4,-7)*{\bigb{ e_j}};
\endxy
\;=\;
\xy
 (-4,14);(-4,-10); **[black][|(2)]\dir{-} ?(0)*[black][|(1)]\dir{<};
 (4,14);(4,-10); **[black][|(2)]\dir{-} ?(1)*[black][|(1)]\dir{>};
%(-6,-10);(4,-10); **[black][|(1)]\dir{-} ?(0.5)*[black][|(1)]\dir{<};
(-4,12);(4,12)*{} **\crv{(-4,8) & (4,8)}?(.5)*\dir{>};
(-4,-8);(4,-8)*{} **\crv{(-4,-4) & (4,-4)}?(.5)*\dir{<};
(-10,-10);(4,8)*{} **\crv{(-10,4) & (0,4)}?(.2)*\dir{>};
%(-6,8)*{\scriptstyle {k}};
% (2,8)*{\scriptstyle {k}};
(-6,14)*{\scriptstyle {k}};
(6,14)*{\scriptstyle {k}};
(4,1)*{\bigb{ e_{j-1}}};
\endxy
$$
\end{proof}

\begin{lem}
We have $gp=1\in \End(C'_{k-1,l-1})$.
\end{lem}
\begin{proof}
Putting $g$ on the top of $p$ we get for any $1< i_2< \dots < i_l$
$$
gp(w_{i_2}\wedge \dots \wedge w_{i_l}) =
\sum^l_{j=1} \sum^k_{i_1=1} (-1)^{j-i_j}
\xy
 (-4,12);(-4,-10); **[black][|(2)]\dir{-} ?(0)*[black][|(1)]\dir{<};
 (4,12);(4,-10); **[black][|(2)]\dir{-} ?(1)*[black][|(1)]\dir{>};
%(-12,14);(-12,10); **[black][|(1)]\dir{-} ?(1)*[black][|(1)]\dir{>};
%(-6,8)*{\scriptstyle {k}};
% (1,8)*{\scriptstyle {k}};
(-8,-8)*{\scriptstyle {k-1}};
(7,-8)*{\scriptstyle {k}};
(4,-2);(4,11)*{} **\crv{(-20,5) & (6,5)}?(.1)*{};
%(-4,-10);(4,-10)*{} **\crv{(-4,-6) & (4,-6)}?(.5)*\dir{<};
(0,-5)*{\bigb{\; c_{i_1-1}\;}};
(4,3)*{\bigb{ e_{{k-i_j}}}};
\endxy \;\;\; \gamma_{k, l-1} (w_{i_1}\wedge w_{i_2}\wedge \dots \hat w_{i_j}\dots
 \wedge w_{i_l})$$
which is equal to
$$ \sum^k_{i=1, i\neq i_2, \dots ,i_l} (-1)^{i-1}
\xy
 (-4,12);(-4,-10); **[black][|(2)]\dir{-} ?(0)*[black][|(1)]\dir{<};
 (4,12);(4,-10); **[black][|(2)]\dir{-} ?(1)*[black][|(1)]\dir{>};
%(-12,14);(-12,10); **[black][|(1)]\dir{-} ?(1)*[black][|(1)]\dir{>};
%(-6,8)*{\scriptstyle {k}};
% (1,8)*{\scriptstyle {k}};
(-8,-8)*{\scriptstyle {k-1}};
(7,-8)*{\scriptstyle {k}};
(4,-2);(4,11)*{} **\crv{(-20,5) & (6,5)}?(.1)*{};
%(-4,-10);(4,-10)*{} **\crv{(-4,-6) & (4,-6)}?(.5)*\dir{<};
(0,-5)*{\bigb{\; c_{i-1}\;}};
(4,3)*{\bigb{ e_{{k-i}}}};
\endxy \;\;\;  w_{i_2}\wedge \dots \wedge w_{i_l} \; +\;
\sum^l_{j=2}\;\;\sum^k_{i=2, i\neq i_2, \dots ,\hat i_j,\dots , i_l} (-1)^{i_j-j}
\xy
 (-4,12);(-4,-10); **[black][|(2)]\dir{-} ?(0)*[black][|(1)]\dir{<};
 (4,12);(4,-10); **[black][|(2)]\dir{-} ?(1)*[black][|(1)]\dir{>};
%(-12,14);(-12,10); **[black][|(1)]\dir{-} ?(1)*[black][|(1)]\dir{>};
%(-6,8)*{\scriptstyle {k}};
% (1,8)*{\scriptstyle {k}};
(-8,-8)*{\scriptstyle {k-1}};
(7,-8)*{\scriptstyle {k}};
(4,-2);(4,11)*{} **\crv{(-20,5) & (6,5)}?(.1)*{};
%(-4,-10);(4,-10)*{} **\crv{(-4,-6) & (4,-6)}?(.5)*\dir{<};
(0,-5)*{\bigb{\; c_{i-1}\;}};
(4,3)*{\bigb{ e_{{k-i_j}}}};
\endxy \;\;\; w_i\wedge w_{i_2}\wedge \dots\hat w_{i_j}\dots \wedge w_{i_l}\;
+
$$
$$ +
\sum^l_{j=2}\;\;\sum^k_{i=2, i\neq i_2, \dots, i_l} (-1)^{i_j-j-1}
\xy
 (-4,12);(-4,-10); **[black][|(2)]\dir{-} ?(0)*[black][|(1)]\dir{<};
 (4,12);(4,-10); **[black][|(2)]\dir{-} ?(1)*[black][|(1)]\dir{>};
%(-12,14);(-12,10); **[black][|(1)]\dir{-} ?(1)*[black][|(1)]\dir{>};
%(-6,8)*{\scriptstyle {k}};
% (1,8)*{\scriptstyle {k}};
(-8,-8)*{\scriptstyle {k-1}};
(7,-8)*{\scriptstyle {k}};
(4,-2);(4,11)*{} **\crv{(-20,5) & (6,5)}?(.1)*{};
%(-4,-10);(4,-10)*{} **\crv{(-4,-6) & (4,-6)}?(.5)*\dir{<};
(0,-5)*{\bigb{\; c_{i-1}\;}};
(4,3)*{\bigb{ e_{{k-i_j}}}};
\endxy \;\;\; w_i\wedge w_{i_2}\wedge \dots\hat w_{i_j}\dots \wedge w_{i_l}
 \, .$$
After cancellation we get
$$gp(w_{i_2}\wedge \dots \wedge w_{i_l})=
\sum^k_{i=1} (-1)^{i-1}
\xy
 (-4,12);(-4,-10); **[black][|(2)]\dir{-} ?(0)*[black][|(1)]\dir{<};
 (4,12);(4,-10); **[black][|(2)]\dir{-} ?(1)*[black][|(1)]\dir{>};
%(-12,14);(-12,10); **[black][|(1)]\dir{-} ?(1)*[black][|(1)]\dir{>};
%(-6,8)*{\scriptstyle {k}};
% (1,8)*{\scriptstyle {k}};
(-8,-8)*{\scriptstyle {k-1}};
(7,-8)*{\scriptstyle {k}};
(4,-2);(4,11)*{} **\crv{(-20,5) & (6,5)}?(.1)*{};
%(-4,-10);(4,-10)*{} **\crv{(-4,-6) & (4,-6)}?(.5)*\dir{<};
(0,-5)*{\bigb{\; c_{i-1}\;}};
(4,3)*{\bigb{ e_{{k-i}}}};
\endxy \;\;\;  w_{i_2}\wedge \dots \wedge w_{i_l}
$$
Using the Reidemeister move listed in Appendix
%Stosic Formula (Corollary 5.8  in \cite{KLMS})
%for $b=k$, $a=1$ 
we can see that the only non-zero term without bubbles
is the desired identity, and all the bubble terms cancel since
\begin{equation}\label{e-c}
\sum^{k-1}_{i=0} (-1)^{i} \xy
(-4,8);(-4,-8); **[black][|(2)]\dir{-} ?(0)*[black][|(1)]\dir{<};
 (4,8);(4,-8); **[black][|(2)]\dir{-} ?(1)*[black][|(1)]\dir{>};
%(-12,10);(-12,6); **[black][|(1)]\dir{-} ?(1)*[black][|(1)]\dir{>};
(-7,8)*{\scriptstyle {k-2}};
% (1,8)*{\scriptstyle {k}};
(8,8)*{\scriptstyle {k-1}};
%(2,-8)*{\scriptstyle {k}};
(0,-3)*{\bigb{\;\;\; c_{i}\;\;\;}};
(4,3)*{\bigb{e_{k-1-i}}};
\endxy\;=\;0 \, 
\end{equation}
where the last identity is equivalent to \eqref{d21}. Here we are 
again using centrality of $c_i$'s.
 \end{proof}

\subsection{Horizontal homotopy}
Let $q_{k,1}: W_k \to \Dot (\E^{(k-1)}\onen)$
be the map defined by $q_{k,1}(w_i)=(-1)^{i-1} e_{k-i}$. Then let
$q_{k,l}: \Lambda^l W_k \to \Dot (\E^{(k-1)}\onen)
 \otimes \Lambda^{l-1} W_k$
be the derivation induced by $q_{k,1}$, i.e.
$$ q_{k,l} (w_{i_1}\wedge \dots \wedge w_{i_l})=\sum^l_{j=1} (-1)^{j-1}
q_{k,1}(w_{i_j}) w_{i_1}\wedge \dots \hat w_{i_j}\dots \wedge w_{i_l} \; .
$$
We set 
\[{\rm h}^H_{k,l}:=\; 
\xy
 (-4,10);(-4,-10); **[black][|(2)]\dir{-} ?(0)*[black][|(1)]\dir{<};
 (4,10);(4,-10); **[black][|(2)]\dir{-} ?(1)*[black][|(1)]\dir{>};
(-6,8)*{\scriptstyle {k}};
 (2,8)*{\scriptstyle {k}};
(-6,-8)*{\scriptstyle {k}};
(2,-8)*{\scriptstyle {k}};
(-12,-10);(4,7)*{} **\crv{(-12,0) & (-3,0)}?(.1)*\dir{>};
(-12,10);(4,-7)*{} **\crv{(-12,0) & (-3,0)}?(0)*\dir{<};
(4,0)*{\bigb{q_{k,l}}}
\endxy \quad : \E C_{k,l}\to \E C_{k,l-1} \, .
\]

\begin{lem}\label{lem9.4}
We have 
\begin{align} \label{11}
g_{k,l-1} {\rm h}^H_{k,l}&=0\\ \label{22}
{\rm h}^H_{k,l-1} \circ  {\rm h}^H_{k,l}&=0\\ \label{33}
{\rm h}^H_{k+1,l}(\E d^V_{k,l}) + (\E d^V_{k,l-1}){\rm h}^H_{k,l}&=0
\end{align}
\end{lem}

\begin{proof}
Equation \eqref{11}  reduces to
\[\xy
 (-4,12);(-4,-10); **[black][|(2)]\dir{-} ?(0)*[black][|(1)]\dir{<};
 (4,12);(4,-10); **[black][|(2)]\dir{-} ?(1)*[black][|(1)]\dir{>}
?(0.3)*{\bullet}+(3,0)*{ e_{i}}
?(0.7)*{\bullet}+(3,0)*{ e_{j}};
(4,-1);(-4,3) **\crv{(0,0) & (-3,5)} ?(.3)*[black][|(1)]\dir{>};
%(-6,8)*{\scriptstyle {k}};
% (1,8)*{\scriptstyle {k}};
(0,-10)*{\scriptstyle {k-1}};
(7,-10)*{\scriptstyle {k}};
(4,-8);(4,11)*{} **\crv{(-30,5) & (6,5)}?(.5)*\dir{>};
(-10,-10);(4,4)*{} **\crv{(-10,2) & (5,2)}?(.1)*\dir{>};
%(4,-4)*{\bigb{e_i}};
%(4,5)*{\bigb{ e_j}};
\endxy \;\;-\;\;
\xy
 (-4,12);(-4,-10); **[black][|(2)]\dir{-} ?(0)*[black][|(1)]\dir{<};
 (4,12);(4,-10); **[black][|(2)]\dir{-} ?(1)*[black][|(1)]\dir{>}
?(0.3)*{\bullet}+(3,0)*{ e_{j}}
?(0.7)*{\bullet}+(3,0)*{ e_{i}};
(4,-1);(-4,3) **\crv{(0,0)&(-3,5)} ?(.3)*[black][|(1)]\dir{>};
%(-12,-10);(4,3); **[black][|(1)]\dir{-} ?(1)*[black][|(1)]\dir{>};
%(-6,8)*{\scriptstyle {k}};
% (1,8)*{\scriptstyle {k}};
(0,-10)*{\scriptstyle {k-1}};
(7,-10)*{\scriptstyle {k}};
(4,-8);(4,11)*{} **\crv{(-30,5) & (6,5)}?(.5)*\dir{>};
(-10,-10);(4,4)*{} **\crv{(-10,2) & (5,2)}?(.1)*\dir{>};
%(4,-4)*{\bigb{e_j}};
%(4,5)*{\bigb{ e_i}};
\endxy\;\;=\;\;0
\]
for $1\leq i<j\leq k-1$. Indeed, moving  the line starting at the lowest left corner up through the 
3-valent vertex (note that the bubble terms cancel) and then moving
the $e_j$ and $e_i$ to the middle of the strand, 
 we get
that the left hand side is equal to
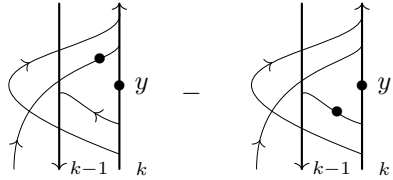
\[
\xy
 (-4,12);(-4,-10); **[black][|(2)]\dir{-} ?(0)*[black][|(1)]\dir{<};
 (4,12);(4,-10); **[black][|(2)]\dir{-} ?(1)*[black][|(1)]\dir{>}
?(0.5)*{\bullet}+(3,0)*{y};
%?(0.7)*{\bullet}+(3,0)*{ e_{j}};
(4,-4);(-4,0) **\crv{(0,-3) & (-3,1)} ?(.3)*[black][|(1)]\dir{>};
%(-12,-10);(4,1); **[black][|(1)]\dir{-} ?(1)*[black][|(1)]\dir{>};
%(-6,8)*{\scriptstyle {k}};
% (1,8)*{\scriptstyle {k}};
(0,-10)*{\scriptstyle {k-1}};
(7,-10)*{\scriptstyle {k}};
(4,-8);(4,11)*{} **\crv{(-30,5) & (6,5)}?(.5)*\dir{>};
(-10,-10);(4,7)*{} **\crv{(-10,4) & (5,4)}?(.1)*\dir{>}
?(0.7)*{\bullet};
%(4,-4)*{\bigb{e_j}};
%(4,5)*{\bigb{ e_i}};
\endxy  \quad -\quad
\xy
 (-4,12);(-4,-10); **[black][|(2)]\dir{-} ?(0)*[black][|(1)]\dir{<};
 (4,12);(4,-10); **[black][|(2)]\dir{-} ?(1)*[black][|(1)]\dir{>}
?(0.5)*{\bullet}+(3,0)*{y};
%?(0.7)*{\bullet}+(3,0)*{ e_{j}};
%(4,-1);(-4,3) **\crv{(0,0)} 
(4,-4);(-4,0) **\crv{(0,-3) & (-3,1)} ?(0.3)*{\bullet};
%(-12,-10);(4,3); **[black][|(1)]\dir{-} ?(1)*[black][|(1)]\dir{>};
%(-6,8)*{\scriptstyle {k}};
% (1,8)*{\scriptstyle {k}};
(0,-10)*{\scriptstyle {k-1}};
(7,-10)*{\scriptstyle {k}};
(4,-8);(4,11)*{} **\crv{(-30,5) & (6,5)}?(.5)*\dir{>};
%(-10,-10);(4,4)*{} **\crv{(-10,2) & (5,2)}?(.1)*\dir{>};
(-10,-10);(4,7)*{} **\crv{(-10,4) & (5,4)}?(.1)*\dir{>};
%(4,-4)*{\bigb{e_j}};
%(4,5)*{\bigb{ e_i}};
\endxy
\]
where $y:=e_{j-1}e_i-e_{i-1}e_j$. 
Now  using  the associativity 
and the invariance under the 3. Reidemeister move with
all strands going in the same direction we can see that both
summands  vanish.

Equation \eqref{22} reduces to the following identity
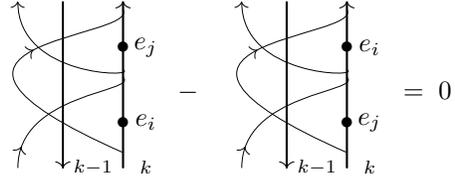
\[\xy
 (-4,12);(-4,-10); **[black][|(2)]\dir{-} ?(0)*[black][|(1)]\dir{<};
 (4,12);(4,-10); **[black][|(2)]\dir{-} ?(1)*[black][|(1)]\dir{>};
%(-12,-10);(4,3); **[black][|(1)]\dir{-} ?(1)*[black][|(1)]\dir{>};
%(-6,8)*{\scriptstyle {k}};
% (1,8)*{\scriptstyle {k}};
(0,-10)*{\scriptstyle {k-1}};
(7,-10)*{\scriptstyle {k}};
(4,-8);(4,11)*{} **\crv{(-30,7) & (6,7)}?(.5)*\dir{>};
(-10,-10);(4,2)*{} **\crv{(-10,0) & (6,0)}?(.1)*\dir{>};
(-10,12);(4,3)*{} **\crv{(-10,2) & (6,2)}?(0)*\dir{<};
(4,-4)*{\bullet}+(3,0)*{e_i};
(4,6)*{\bullet}+(3,0)*{ e_j};
\endxy \;\;-\;\;
\xy
 (-4,12);(-4,-10); **[black][|(2)]\dir{-} ?(0)*[black][|(1)]\dir{<};
 (4,12);(4,-10); **[black][|(2)]\dir{-} ?(1)*[black][|(1)]\dir{>};
%(-12,-10);(4,3); **[black][|(1)]\dir{-} ?(1)*[black][|(1)]\dir{>};
%(-6,8)*{\scriptstyle {k}};
% (1,8)*{\scriptstyle {k}};
(0,-10)*{\scriptstyle {k-1}};
(7,-10)*{\scriptstyle {k}};
(4,-8);(4,11)*{} **\crv{(-30,7) & (6,7)}?(.5)*\dir{>};
(-10,-10);(4,2)*{} **\crv{(-10,0) & (6,0)}?(.1)*\dir{>};
(-10,12);(4,3)*{} **\crv{(-10,2) & (6,2)}?(0)*\dir{<};
(4,-4)*{\bullet}+(3,0)*{e_j};
(4,6)*{\bullet}+(3,0)*{ e_i};
\endxy\;\;=\;0
\]
for $1\leq i<j\leq k-1$, which can be proved similarly.

Finally, equation \eqref{33} follows from
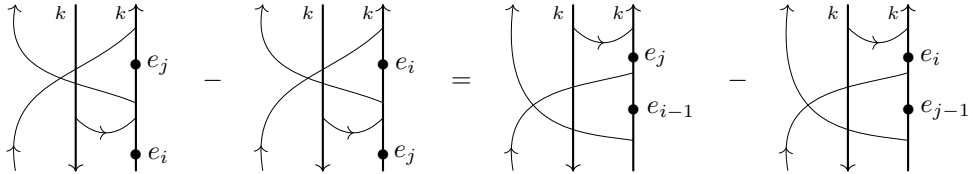
\[
\xy
 (-4,10);(-4,-12); **[black][|(2)]\dir{-} ?(0)*[black][|(1)]\dir{<};
 (4,10);(4,-12); **[black][|(2)]\dir{-} ?(1)*[black][|(1)]\dir{>};
(-6,9)*{\scriptstyle {k}};
 (2,9)*{\scriptstyle {k}};
%(-6,-13)*{\scriptstyle {k}};
%(2,-13)*{\scriptstyle {k}};
(-12,-12);(4,7)*{} **\crv{(-14,0) & (-3,0)}?(.1)*\dir{>};
(-12,10);(4,-3)*{} **\crv{(-14,0) & (-3,0)}?(0)*\dir{<};
(4,2)*{\bullet}+(3,0)*{e_j};
(4,-10)*{\bullet}+(3,0)*{e_i};
(-4,-5);(4,-5)*{} **\crv{(-0,-9)}?(.5)*\dir{>};
\endxy
\quad -\quad
\xy
 (-4,10);(-4,-12); **[black][|(2)]\dir{-} ?(0)*[black][|(1)]\dir{<};
 (4,10);(4,-12); **[black][|(2)]\dir{-} ?(1)*[black][|(1)]\dir{>};
(-6,9)*{\scriptstyle {k}};
 (2,9)*{\scriptstyle {k}};
%(-6,-13)*{\scriptstyle {k}};
%(2,-13)*{\scriptstyle {k}};
(-12,-12);(4,7)*{} **\crv{(-14,0) & (-3,0)}?(.1)*\dir{>};
(-12,10);(4,-3)*{} **\crv{(-14,0) & (-3,0)}?(0)*\dir{<};
(4,2)*{\bullet}+(3,0)*{e_i};
(4,-10)*{\bullet}+(3,0)*{e_j};
(-4,-5);(4,-5)*{} **\crv{(-0,-9)}?(.5)*\dir{>};
\endxy \quad
=\quad
\xy
 (-4,10);(-4,-12); **[black][|(2)]\dir{-} ?(0)*[black][|(1)]\dir{<};
 (4,10);(4,-12); **[black][|(2)]\dir{-} ?(1)*[black][|(1)]\dir{>};
(-6,9)*{\scriptstyle {k}};
 (2,9)*{\scriptstyle {k}};
%(-6,-13)*{\scriptstyle {k}};
%(2,-13)*{\scriptstyle {k}};
(-12,-12);(4,1)*{} **\crv{(-14,-1) & (-3,-1)}?(.1)*\dir{>};
(-12,10);(4,-8)*{} **\crv{(-14,-7) & (-3,-7)}?(0)*\dir{<};
(4,3)*{\bullet}+(3,0)*{e_j};
(4,-4)*{\bullet}+(5,0)*{e_{i-1}};
(-4,7);(4,7)*{} **\crv{(0,3)}?(.5)*\dir{>};
\endxy\quad -\quad
\xy
 (-4,10);(-4,-12); **[black][|(2)]\dir{-} ?(0)*[black][|(1)]\dir{<};
 (4,10);(4,-12); **[black][|(2)]\dir{-} ?(1)*[black][|(1)]\dir{>};
(-6,9)*{\scriptstyle {k}};
 (2,9)*{\scriptstyle {k}};
%(-6,-13)*{\scriptstyle {k}};
%(2,-13)*{\scriptstyle {k}};
(-12,-12);(4,1)*{} **\crv{(-14,-1) & (-3,-1)}?(.1)*\dir{>};
(-12,10);(4,-8)*{} **\crv{(-14,-7) & (-3,-7)}?(0)*\dir{<};
(4,3)*{\bullet}+(3,0)*{e_i};
(4,-4)*{\bullet}+(5,0)*{e_{j-1}};
(-4,7);(4,7)*{} **\crv{(0,3)}?(.5)*\dir{>};
\endxy
\]
which holds for any $0\leq i<j\leq k$. 
\end{proof}

\subsection{Vertical homotopy}
We set
%\F^{(k)}\E^{(k)}\onen \otimes \Lambda^l W_k \to 
%\E\F^{(k-1)}\E^{(k-1)}\onen \otimes \Lambda^{l} W_{k-1}$
$${h}^V_{k,l}:=
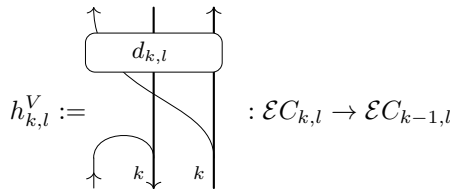
\xy
 (-4,14);(-4,-10); **[black][|(2)]\dir{-} ?(0)*[black][|(1)]\dir{<};
 (4,14);(4,-10); **[black][|(2)]\dir{-} ?(1)*[black][|(1)]\dir{>};
(-12,-10);(-12,-6); **[black][|(1)]\dir{-} ?(0.5)*[black][|(1)]\dir{<};
%(-12,-6);(-4,-6); **\crv{(-12,-2) & (-4,-2)}? (1)*{};
(-12,-6);(-4,-6)*{} **\crv{(-12,-2) & (-4,-2)}?(.5)*{};
(4,-6);(-12,14)*{} **\crv{(4,2) & (-12,2)}?(1)*\dir{>};
%(-6,8)*{\scriptstyle {k}};
% (2,8)*{\scriptstyle {k}};
(-6,-8)*{\scriptstyle {k}};
(2,-8)*{\scriptstyle {k}};
(-4,8)*{\bigb{\quad\;\; d_{k,l}\;\quad\;\;}};
\endxy \quad :  \E C_{k,l} \to \E C_{k-1,l}
$$
where
$d_{k,l}: \Lambda^l W_k \to 
\Dot (\E\F^{(k)}\E^{(k)}\onen)
 \otimes \Lambda^{l} W_{k-1} $
is given by
\begin{align*}
d_{k,0}(1)& =(-1)^{k-1}(c_{k-1})_{12}= (-1)^{k-1}\quad
\xy
(-8,-5);(-8,5); **[black][|(1)]\dir{-} ?(0)*[black][|(1)]\dir{<};
 (-4,5);(-4,-5); **[black][|(2)]\dir{-} ?(0)*[black][|(1)]\dir{<};
(0,5);(0,-5); **[black][|(2)]\dir{-} ?(1)*[black][|(1)]\dir{>};
(-6,0)*{\bigb{ c_{k-1}}};
\endxy
\\[.1cm] 
d_{k,1}(w_i) &=(-1)^k (c_{k-1})_{12} \; w_i
+ (-1)^{i-1}(e_{k-i})_3 \; (w_1+c_1 w_2+\dots+c_{k-2} w_{k-1})
\;\;\;{\text {for}}\; i<k \\[.1cm]
d_{k,1}(w_k) &=
 (-1)^{k-1}( w_1+c_1 w_2+\dots+c_{k-2} w_{k-1})
\;\;\;{\text {and}}\;\;\\[1mm]
%\begin{align*}
d_{k,l}(w_{i_1}\wedge \dots w_{i_l}) &= (-1)^{l-1} 
\left((-1)^{k} (c_{k-1})_{12}\right)^{-l+1} d_{k,1}(w_{i_1})\wedge \dots \wedge 
d_{k,1}(w_{i_l})\\
&= (-1)^{l-1}\left(
(-1)^k (c_{k-1})_{12} \; w_{i_1}\wedge \dots \wedge w_{i_l}+
\sum^l_{j=1} (-1)^{i_j-1} (e_{k-i_j})_3\; w_{i_1}\wedge \dots (X)_j \dots \wedge w_{i_l}
\right)
\end{align*}
where $X:= \sum^{k-1}_{i=1} c_{i-1} w_i$ and $(X)_j$ means that we
replace $w_{i_j}$ with $X$. 
As before the lower indices indicate the strands
on which the morphism is acting. 

%Observe that $d_{k,l}$ is well defined, since its image
%can contain

Let us illustrate this definition with few examples.
\[ d_{3,1}=
\begin{pmatrix} -(c_2)_{12} + (e_2)_3&-(e_1)_3& {\bf 1}\\
(e_2)_3c_1& -(c_2)_{12}-(e_1)_3 c_1& c_1 \end{pmatrix}
\]

\[ d_{3,1}\wedge d_{3,1}= \left( \det\begin{pmatrix}  
-(c_2)_{12} + (e_2)_3&-(e_1)_3\\
(e_2)_3c_1& -(c_2)_{12}-(e_1)_3 c_1
\end{pmatrix},
\det \begin{pmatrix} -(c_2)_{12} + (e_2)_3& {\bf 1}\\
(e_2)_3c_1& c_1 \end{pmatrix}, 
\begin{pmatrix} -(e_1)_3& {\bf 1}\\
 -(c_2)_{12}-(e_1)_3 c_1& c_1 \end{pmatrix}
\right)
\]
Hence, we get 
\begin{align*}
d_{3,1}\wedge d_{3,1} &= -(c_2)_{12}(-c_2, c_1, {\bf 1})\\
d_{3,2} &= (c_2, -c_1,{\bf 1})\, .\\
\end{align*}

Similarly, $d_{k,k-1}= (c_{k-1}, -c_{k-2}, \dots, (-1)^{k-2}c_1, (-1)^{k-1}{\bf 1})$.

\begin{lem}
We have 
\begin{align*} 
g_{k-1,l} { h}^V_{k,l}&=0\\
h^V_{k,l} p_{k,l}&=0\\
{ h}^V_{k-1,l} \circ  { h}^V_{k,l}&=0\\  
{\rm h}^H_{k-1,l} h^V_{k,l} + h^V_{k,l-1}{\rm h}^H_{k,l}&=0\\
h^V_{k,l+1} (\E d^H_{k,l})+ (\E d^H_{k-1,l}) h^V_{k,l}&=0\\ 
{\rm h}^H_{k,l+1}(\E d^H_{k,l}) + h^V_{k+1,l}(\E d^V_{k,l}) +(\E d^H_{k,l-1}) {\rm h}^H_{k,l} +(\E d^V_{k-1,l}) h^V_{k,l} &={\bf 1} - p_{k,l} g_{k,l}
\end{align*}
\end{lem}

\begin{proof}
 The proof of the last equality is based on 
the  two  identities  given in Lemmas 5,6 in Appendix.
The rest is  similar to the previous computations and hence
left to the reader.
\end{proof}

\subsection{Proof of Lemma \ref{mainlemma}}
Thus we proved, that 
the
 map \[\kappa_\E: \xymatrix{ 
\E r\onen  \ar[r]^{g} & \Int\ar[r]^{\bar f} & r\E\onen}
\]
has a homotopy  inverse
\[\bar \kappa_\E:\xymatrix{ 
\E r\onen & \ar[l]^{p}  \Int & \ar[l]^{ f}  r\E\onen}
\]
To construct $\kappa_\F$ we apply the symmetry $\sigma\omega$ to $\kappa_\E$
assuming that Theorem \ref{main} (Symmetry) holds. We get
\[\bar \kappa_\F: \xymatrix{ 
 r\F{\bf 1}_{n+2}  \ar[r]^{\sigma\omega(g)} & \Int\ar[r]^{\sigma\omega(\bar f)} 
& \F r{\bf 1}_{n+2}}
\]
together with its homotopy inverse
\[\kappa_\F: \xymatrix{ 
 r\F{\bf 1}_{n+2} & \ar[l]^{\sigma\omega(p)}  \Int & \ar[l]^{\sigma\omega( f)}  
 \F r{\bf 1}_{n+2}} \, .
\]
\qed

\subsection{Proof of Theorem \ref{main} (Centrality)}
\label{proof-main}
We assume that Theorem \ref{main} (Symmetry) holds.

Since any 1-morphism in $\UcatD$ is a direct sum of compositions
  of $\E\la t\ra$ and $\F\la t'\ra$ with $t,t'\in \Z$,
it is enough to check the statement for the generators. 
Lemma \ref{mainlemma} defines the maps $\kappa_\E$ and $\kappa_\F$
as well as their homotopy inverses. Applying symmetry,
we can define $\eta_\F=\psi (\bar\kappa_\F)$ and $\eta_\E=\psi (\kappa_\E)$.
The details are left to the reader.

\qed

\subsection{Comments on the naturality of maps $\kappa_X$}
To prove Conjecture 2 (Naturality) we need to show that
for  any chain map $f: X \to Y$ the squares below
commute up to chain homotopy.

\begin{equation*}
    \xy
   (-10,10)*+{X r\onen}="tl";
   (10,10)*+{r X\onen}="tr";
   (-10,-10)*+{Y r\onen}="bl";
   (10,-10)*+{r Y\onen}="br";
   {\ar^-{\kappa_{X}} "tl";"tr"};
   {\ar_{fr} "tl";"bl"};
   {\ar^{\kappa_{Y}} "bl";"br"};
   {\ar^{r f} "tr";"br"};
  \endxy
  \qquad
      \xy
   (-10,10)*+{rX\onen}="tl";
   (10,10)*+{X r\onen}="tr";
   (-10,-10)*+{rY\onen}="bl";
   (10,-10)*+{Y r\onen}="br";
   {\ar^-{\bar{\kappa}_{X}} "tl";"tr"};
   {\ar_{rf} "tl";"bl"};
   {\ar^{\bar{\kappa}_{Y}} "bl";"br"};
   {\ar^{fr} "tr";"br"};
  \endxy
\end{equation*}
The commutativity of similar diagrams for $\eta_X$ will follow 
then by applying  symmetry functors.
It is enough to check the commutativity for 
short chain complexes $f:X\to Y$, where $X$, $Y$ are  $\E\onen$,  $\E^2\onen$, 
 $\F\E\onen$ or $\onen$ and
the differential is one of the generating 2-morphisms: dot,
crossing, cup or cap.
We leave this problem
 for  future investigations.

\section{Symmetry 2-functors}
The 2-category $\Ucat$ has the symmetry group ${\cal G}=(\Z/2\Z)^3$
generated by the involutive 2-functors $\omega, \sigma, \psi$ described below.

\subsection{2-functor $\omega$}
Consider the operation on the diagrammatic calculus that rescales
the crossing $\Ucross \mapsto -\Ucross$ for all $n \in \Z$, inverts
the orientation of each strand and sends $n \mapsto -n$.

This gives a strict invertible 2-functor
$\omega\maps \Ucat \to \Ucat$ given by
\begin{eqnarray}
  \omega \maps \Ucat &\to& \Ucat \nn \\
  n &\mapsto&  -n \nn \\
  \onem\cal{E}^{\alpha_1} \cal{F}^{\beta_1}\cal{E}^{\alpha_2} \cdots
 \cal{E}^{\alpha_k}\cal{F}^{\beta_k}\onen\{s\}
 &\mapsto &
 \mathbf{1}_{-m} \cal{F}^{\alpha_1} \cal{E}^{\beta_1}\cal{F}^{\alpha_2} \cdots
\cal{F}^{\alpha_k}\cal{E}^{\beta_k}\mathbf{1}_{-n}\{s\}.
\end{eqnarray}
This 2-functor extends to a 2-functor
\begin{eqnarray}
  \omega \maps Kom(\Ucat) &\to& Kom(\Ucat) \nn \\
  n &\mapsto&  -n \nn \\
 (X^{\bullet},d)
 &\mapsto &
   \xy
    (-50,0)*+{\cdots}="1";
    (-30,0)*+{\omega(X^{i-1})}="2";
    (0,0)*+{\omega(X^i)}="3";
    (30,0)*+{ \omega(X^{i+1})}="4";
    (50,0)*+{\cdots}="5";
    {\ar^-{} "1";"2"};
    {\ar^-{\omega(d_{i-1})} "2";"3"};
    {\ar^-{\omega(d_i)} "3";"4"};
    {\ar^-{} "4";"5"};
   \endxy \nn \\
   f_i \maps X^{\bullet} \to Y^{\bullet} & \mapsto &
   \omega(f_i) \maps \omega(X^{\bullet}) \to \omega(Y^{\bullet}).
\end{eqnarray}
Finally, this operation extends to $\UcatD$. The images of
the idempotents
$e_a$, $e'_a$ under  the 2-functor $\omega$ are new idempotents which are
equivalent to the old ones and leave   thick calculus invariant.

\subsection{ 2-functor $\sigma$}
The operation on diagrams that rescales the crossing $\Ucross
\mapsto -\Ucross$ for all $n \in \Z$, reflects a diagram across the
vertical axis, and sends $n$ to $-n$ leaves invariant the relations on the
2-morphisms of $\Ucat$. 

This operation
is contravariant for composition of 1-morphisms, covariant for
composition of 2-morphisms, and preserves the degree of a diagram.
This symmetry gives an invertible 2-functor
\begin{eqnarray}
  \sigma \maps \Ucat &\to& \Ucat^{\op} \nn \\
  n &\mapsto&  -n \nn \\
  \onem\cal{E}^{\alpha_1} \cal{F}^{\beta_1}\cal{E}^{\alpha_2} \cdots
 \cal{E}^{\alpha_k}\cal{F}^{\beta_k}\onen\{s\}
 &\mapsto &
 \mathbf{1}_{-n} \cal{F}^{\beta_k} \cal{E}^{\alpha_k}\cal{F}^{\beta_{k-1}} \cdots
\cal{F}^{\beta_1}\cal{E}^{\alpha_1}\mathbf{1}_{-m}\{s\} \nn
\end{eqnarray}
that acts on 2-morphisms via the symmetry described above.  This 2-functor extends to a 2-functor
\begin{eqnarray}
  \sigma \maps Kom(\Ucat) &\to& Kom(\Ucat) \nn \\
  n &\mapsto&  -n \nn \\
 (X^{\bullet},d)
 &\mapsto &
   \xy
    (-50,0)*+{\cdots}="1";
    (-30,0)*+{\sigma(X^{i-1})}="2";
    (0,0)*+{\sigma(X^i)}="3";
    (30,0)*+{ \sigma(X^{i+1})}="4";
    (50,0)*+{\cdots}="5";
    {\ar^-{} "1";"2"};
    {\ar^-{\sigma(d_{i-1})} "2";"3"};
    {\ar^-{\sigma(d_i)} "3";"4"};
    {\ar^-{} "4";"5"};
   \endxy \nn \\
   f_i \maps X^{\bullet} \to Y^{\bullet} & \mapsto &
   \sigma(f_i) \maps \sigma(X^{\bullet}) \to \sigma(Y^{\bullet}).
\end{eqnarray}
Note that $\sigma$ acts contravariantly on 1-morphisms in $Kom(\Ucat)$.

Furthermore, $\sigma$ extends to $\UcatD$. The images
of the idempotents
$e_a$, $e'_a$
under $\sigma$  are equivalent idempotents, leaving thick calculus invariant.

\subsection{2-functor $\psi$}
This operation reflects across the  horizontal axis and invert orientation.
  This gives an invertible 2-functor defined by
\begin{eqnarray}
  \psi \maps \Ucat &\to& \Ucat^{\co} \nn \\
  n &\mapsto&  n \nn \\
  \onem\cal{E}^{\alpha_1} \cal{F}^{\beta_1}\cal{E}^{\alpha_2} \cdots
 \cal{E}^{\alpha_k}\cal{F}^{\beta_k}\onen\{s\}
 &\mapsto &
 \onem\cal{E}^{\alpha_1} \cal{F}^{\beta_1}\cal{E}^{\alpha_2} \cdots
\cal{E}^{\alpha_k}\cal{F}^{\beta_k}\onen\{-s\}
\end{eqnarray}
and on 2-morphisms $\psi$ reflects the diagrams representing
summands across the $x$-axis and inverts the orientation.

Since $\psi$ is contravariant on 2-morphisms in $\Ucat$, this 2-functor extends to a 2-functor
\begin{eqnarray}
  \psi \maps Kom(\Ucat) &\to& Kom(\Ucat) \nn \\
  n &\mapsto&  n \nn \\
 (X^{\bullet},d)
 &\mapsto &
   \xy
    (-50,0)*+{\cdots}="1";
    (-30,0)*+{\psi(X^{i+1})}="2";
    (0,0)*+{\psi(X^i)}="3";
    (30,0)*+{ \psi(X^{i-1})}="4";
    (50,0)*+{\cdots}="5";
    {\ar^-{} "1";"2"};
    {\ar^-{\psi(d_{i})} "2";"3"};
    {\ar^-{\psi(d_{i-1})} "3";"4"};
    {\ar^-{} "4";"5"};
   \endxy \nn \\
   f_i \maps X^{\bullet} \to Y^{\bullet} & \mapsto &
   \psi(f_i) \maps \psi(Y^{\bullet}) \to \psi(X^{\bullet}).
\end{eqnarray}
Notice that $\psi$ inverts the homological degree so that $\psi$ acts on a complex $(X^{\bullet}\onen,\partial)$ in $Kom(\Ucat)$ by $\psi(X^i) = (\psi X^{\bullet})^{-i}$.

These 2-functors commute with each other `on-the-nose':
\begin{equation}
  \omega \sigma = \sigma\omega , \qquad \sigma \psi = \psi \sigma, \qquad
  \omega \psi = \psi \omega.
\end{equation}

Again, this 2-functor extends straightforward
to the Karoubi envelope $\UcatD$. The images of the idempotents
$e_a$, $e'_a$
under symmetry functors are equivalent idempotents
with the same properties as before.

\section{Symmetries of the ribbon bicomplex}\label{sym}
In this section we describe the behavior of the ribbon 
bicomplex under the symmetry 2-functors.

\subsection{The image under $\omega$}
 The ``chain groups'' of the bicomplex $\omega(r\one)$
are
 $$\omega(C_{k,l}):= \E^{(k)} \F^{(k)}\onen\la kn-k\ra\otimes \Lambda^l W_k,\quad
W_k=\Span_\Z\{w_1,\dots, w_k\}, \quad \deg(w_j)=-2j\, $$
with the total shifts $\la -\frac{n^2}{2}+n\ra$ and $[n/2,n/2]$.
The horizontal differential
$ \omega( d^H_{k,l}):\omega( C_{k,l})\to \omega(C_{k,l+1})$ 
sends $x \mapsto \omega(c)\wedge x$ where
\[\omega(c):= 
\sum^k_{j=1}  \left( \sum^j_{i=0} (-1)^i\quad
\xy
 (-8,6);(-8,-6); **[black][|(2)]\dir{-} ?(1)*[black][|(1)]\dir{>};
 (0,6);(0,-6); **[black][|(2)]\dir{-} ?(0)*[black][|(1)]\dir{<};
% (6,4)*{n};
 (2,5)*{\scriptstyle {k}};
 (-6,5)*{\scriptstyle {k}};
% (8,0)*{};
 (-8,0)*{};
(0,0)*{\bigb{h_{j-i}}};
(-8,0)*{\bigb{e_{i}}};
\endxy \;\right) \otimes w_j \in \Dot( \E^{(k)}\F^{(k)}\onen) \otimes W_k
\]
%Here we use the notation
% $c_d|_{\E^{(k)}\F^{(k)}\onen}=\sum^d_{j=0}(-1)^{d-j} e_j\otimes h_{d-j}$.

Similarly,  the vertical differential is
\[ \omega(d^V_{k,l}):= (-1)^l 
\xy
 (-2,-8);(-2,8); **[black][|(2)]\dir{-} ?(0)*[black][|(1)]\dir{<}?;
 (6,-8);(6,8); **[black][|(2)]\dir{-} ?(1)*[black][|(1)]\dir{>}?;
 (6,6)*{};(-2,6)*{} **[black][|(1)]\crv{(6,2) & (-2,2)} ?(.5)*[black][|(1)]\dir{>}?;
 (-6,9)*{\scriptstyle k+1};
 (10,9)*{\scriptstyle k+1};
 (-5,-9)*{\scriptstyle k};
 (8,-9)*{\scriptstyle k};
 (6,-2)*{\bigb{\alpha^{\omega}_{k,l}}};
% (12,-3)*{n};
\endxy\; :  \omega(C_{k,l})\to \omega(C_{k+1,l})\, \quad \text{where}
\]
\begin{align} \label{omega-alpha} 
\alpha^{\omega}_{k,l}(w_{i_1}\wedge w_{i_2}\wedge \dots\wedge w_{i_l})&=\\\nn
{\bf 1} \otimes w_{i_1}\wedge \dots \wedge w_{i_l}&-
\sum^l_{j=1} (-1)^{k+1-i_j+l-j}\; \;
\xy
% (-10,6);(-10,-6); **[black][|(2)]\dir{-} ?(0)*[black][|(1)]\dir{<};
 (0,6);(0,-6); **[black][|(2)]\dir{-} ?(0)*[black][|(1)]\dir{<};
(2,5)*{\scriptstyle {k}};
 (-7,5)*{\scriptstyle {}};
(-0,0)*{\bigb{e_{k+1-i_j}}};
\endxy  \quad\;
 \otimes (w_{i_i}\wedge \dots \wedge\hat w_{i_j}\wedge\dots \wedge w_{i_l})
\wedge w_{k+1}
\end{align}

Since all the relations in $\UcatD$ are invariant under symmetries,
$\omega(r\one)$ is a bicomplex.

\subsection{The image under $\sigma$}
The bicomplex $\sigma(r\one )$ has the same ``chain groups'' as $\omega(r\one)$
(i.e. $\sigma(C_{k,l})=\omega(C_{k,l})$)
with the differentials defined as follows:
The horizontal differential
$ \sigma( d^H_{k,l}):\sigma( C_{k,l})\to \sigma(C_{k,l+1})$ 
sends $x \mapsto \sigma(c)\wedge x$ where
\[\sigma(c):= 
\sum^k_{j=1}  \left( \sum^j_{i=0} (-1)^i\quad\quad
\xy
 (-8,6);(-8,-6); **[black][|(2)]\dir{-} ?(1)*[black][|(1)]\dir{>};
 (0,6);(0,-6); **[black][|(2)]\dir{-} ?(0)*[black][|(1)]\dir{<};
% (6,4)*{n};
 (2,5)*{\scriptstyle {k}};
 (-6,5)*{\scriptstyle {k}};
% (8,0)*{};
 (-8,0)*{};
(-8,0)*{\bigb{h_{j-i}}};
(0,0)*{\bigb{e_{i}}};
\endxy \;\right) \otimes w_j \in \Dot( \E^{(k)}\F^{(k)}\onen) \otimes W_k
\]
%Here we use the notation
% $c_d|_{\E^{(k)}\F^{(k)}\onen}=\sum^d_{j=0}(-1)^{d-j} e_j\otimes h_{d-j}$.

Similarly,  the vertical differential is
\[ \sigma(d^V_{k,l}):= (-1)^l 
\xy
 (-2,-8);(-2,8); **[black][|(2)]\dir{-} ?(0)*[black][|(1)]\dir{<}?;
 (6,-8);(6,8); **[black][|(2)]\dir{-} ?(1)*[black][|(1)]\dir{>}?;
 (6,6)*{};(-2,6)*{} **[black][|(1)]\crv{(6,2) & (-2,2)} ?(.5)*[black][|(1)]\dir{>}?;
 (-6,9)*{\scriptstyle k+1};
 (10,9)*{\scriptstyle k+1};
 (-5,-9)*{\scriptstyle k};
 (8,-9)*{\scriptstyle k};
 (-2,-2)*{\bigb{\alpha_{k,l}}};
% (12,-3)*{n};
\endxy\; :  \sigma(C_{k,l})\to \sigma(C_{k+1,l})\, \quad \text{where}
\]
where $\alpha_{k,l}\in \Dot(\E^{(k)}\onen)\otimes \Lambda^l W_{k+1}$ is defined 
by  \eqref{alpha}.

\subsection{The image under $\psi$}
The 2-functor $\psi$ is contravariant on the 2-morphisms in $\UcatD$, hence
the bicomplex $\psi(r\onen)$ looks as follows:

\[ \xymatrix{C^L_{00}
	     &	& &	  \\
      C^L_{10}\ar[u]&C^L_{11} \ar[l]&&\\
		 C^L_{20}\ar[u]&
C^L_{21}\ar[l] \ar[u]&
C^L_{22}\ar[l]&\\
C^L_{30}\ar[u] &
C^L_{31}\ar[l]\ar[u]
 & C^L_{32}\ar[l]\ar[u]&
C^L_{33}\ar[l]\\
\dots\ar[u] &\dots\ar[u] &\dots \ar[u]& \dots
\ar[u]
}
       \]
where the ``chain groups'' are 
 $$C^L_{k,l}:=\psi(C_{k,l})= \F^{(k)} \E^{(k)}\onen
\la kn+k\ra\otimes \Lambda^l {\bar W}_k
$$
with the total  shifts $\la\frac{n^2}{2}+n\ra$ and $[-n/2,-n/2]$.
The horizontal differential 
$\psi( d^H_{k,l}): C^L_{k,l+1}\to C^L_{k,l}$ sends 
$$ \bar w_{i_1}\wedge\dots\wedge \bar w_{i_{l+1}}
 \mapsto \sum^l_{j=1} (-1)^{j-1} c_j \otimes  \bar w_{i_1}\wedge 
\hat {\bar w}_{i_j} \wedge \bar w_{i_{l+1}}$$ 
where as before
 $(c_d)_{\F^{(k)}\E^{(k)}\onen}=\sum^d_{j=0}(-1)^j e_j\otimes h_{d-j}$.
The vertical differential is
\[\psi( d^V_{k,l}):= (-1)^l 
\xy
 (-2,8);(-2,-8); **[black][|(2)]\dir{-} ?(0)*[black][|(1)]\dir{<}?;
 (6,-8);(6,8); **[black][|(2)]\dir{-} ?(0)*[black][|(1)]\dir{<}?;
 (6,-6)*{};(-2,-6)*{} **[black][|(1)]\crv{(6,-2) & (-2,-2)} ?(.5)*[black][|(1)]\dir{>}?;
 (-6,-9)*{\scriptstyle k+1};
 (10,-9)*{\scriptstyle k+1};
 (-5,9)*{\scriptstyle k};
 (8,9)*{\scriptstyle k};
 (6,3)*{\bigb{\alpha_{k,l}}};
% (12,-3)*{n};
\endxy\; :  C^L_{k+1,l}\to C^L_{k,l}\, \quad 
\]
with $\alpha_{k,l}$ defined as before by \eqref{alpha}.

It is easy to see that $\sigma\omega \psi (r\onen)=r^{-1}\onen$.

%For any $g \in \cal G$ let us denote by $g(r \onen)$ the 
%image of $r\onen$ after applying the symmetry functor $g$.

%For 
%the chain groups of $g(r\onen)$ are 
%$\E^{(k)}\F^{(k)}\onen \otimes \Lambda^l W_k$ up to grading shift.

\section{Isomorphic bicomplex}\label{old}
This section provides a construction of the bicomplex $\tilde r\onen$,
which is isomorphic to the ribbon bicomplex and invariant under
$\sigma\omega$.

\subsection{The isomorphism $H$}
%Let $A=\Z[h_1, h_2, \dots]=\Z[e_1, e_2, \dots]$ be the ring of the
%symmetric functions. 
Let $W=\Span_\Z\{w_1,w_2, \dots\}$ and
$W_A=A\otimes W$.

Then there  exists an isomorphism
$H: W_A\to W_A$ with $$ H(w_m):=\sum_{j\geq m} h_{j-m} w_j\, .$$ 
Its inverse is defined by replacing $h_i$ with its 
antipode $(-1)^i e_i$,
hence we have
$$H^{-1}(w_m)=\sum_{j\geq m} (-1)^{j-m} e_{j-m} w_j\, .$$

We can use this map to define a non-trivial transformation
\[
H_{k,l}: C_{k,l} \to C_{k,l}\, 
\]
of the ``chain groups'' $C_{k,l}=\F^{(k)}\E^{(k)}\onen \otimes \Lambda^l W_k$
of the ribbon bicomplex as follows.

For $l=1$ we set
\begin{equation}\label{Hk1}
H_{k,1}:=\left( \begin{array} {rrrrrr}
{\bf 1} & 0 & 0& 0 &\dots&0\\
(h_1)_2 & {\bf 1} & 0& 0 &\dots&0\\
(h_2)_2 & (h_1)_2 & {\bf 1}& 0 &\dots&0\\
(h_3)_2 & (h_2)_2 & (h_1)_2& {\bf 1} &\dots&0\\
\dots &\dots &\dots &\dots &\dots &0\\
(h_{k-1})_2& (h_{k-2})_2 & (h_{k-3})_2&(h_{k-4})_2&\dots &{\bf 1}
\end{array}\right)=
\left(  (h_{i-j})_2 \right)_{1\leq i,j\leq k}
\end{equation}
where the matrix is written in the basis $w_1, w_2, \dots, w_k$
and all symmetric polynomials are sitting on the second strand.

This map obviously extends to $\Lambda^l W_k$ by setting
$$H_{k,l}(w_{i_1}\wedge w_{i_2}\wedge \dots \wedge w_{i_l}):=
\left(\Lambda^l H_{k,1}\right) (w_{i_1}\wedge w_{i_2}\wedge \dots \wedge w_{i_l})=
H_{k,1}(w_{i_1})\wedge H_{k,1}(w_{i_2})\wedge \dots \wedge
H_{k,1}(w_{i_l}) \, .$$

Inserting \eqref{Hk1}, we get
$$H_{k,l}(w_{i_1}\wedge w_{i_2}\wedge \dots \wedge w_{i_l}):=
\sum_{1\leq j_1<j_2<\dots<j_l\leq k}
\sum_{\sigma\in S_l} (-1)^\sigma
(h_{j_1-i_{\sigma(1)}})_2 (h_{j_2-i_{\sigma(2)}})_2
\dots (h_{j_l-i_{\sigma(l)}})_2
w_{i_1}\wedge w_{i_2}\wedge \dots \wedge w_{i_l}\, . $$
Hence, for $1\leq i_1<i_2 \dots<i_l\leq k$ and $1\leq j_1<j_2 \dots<j_l\leq k$
\[
\left(H_{k,l}\right)^{i_1,i_2,\dots ,i_l}_{j_1,j_2, \dots , j_l}
= 
\sum_{\sigma\in S_l} (-1)^\sigma
(h_{j_1-i_{\sigma(1)}})_2 (h_{j_2-i_{\sigma(2)}})_2
\dots (h_{j_l-i_{\sigma(l)}})_2
\]

The inverse map is defined in a similar way by using
\begin{equation}\label{H-1}
H^{-1}_{k,1}:=\left( \begin{array} {rrrrrr}
{\bf 1} & 0 & 0& 0 &\dots&0\\
-(e_1)_2 & {\bf 1} & 0& 0 &\dots&0\\
(e_2)_2 & -(e_1)_2 & {\bf 1}& 0 &\dots&0\\
-(e_3)_2 & (e_2)_2 & -(e_1)_2& {\bf 1} &\dots&0\\
\dots &\dots &\dots &\dots &\dots &0\\
(-1)^{k-1}(e_{k-1})_2& (-1)^{k-2}(e_{k-2})_2 & (-1)^{k-3}(e_{k-3})_2&(-1)^{k-4}(e_{k-4})_2&\dots &{\bf 1}
\end{array}\right)
\end{equation}
 or
$$ \left( H^{-1}_{k,1}\right)_{i,j}
= \left( (-1)^{i-j} (e_{i-j})_2 \right)_{1\leq i,j\leq k} \, .
$$
\vspace*{.2cm}
For $1\leq i_1<i_2 \dots<i_l\leq k$ and $1\leq j_1<j_2 \dots<j_l\leq k$ 
we have
\[
\left(H^{-1}_{k,l}\right)^{i_1,i_2,\dots ,i_l}_{j_1,j_2, \dots , j_l}
= (-1)^{\sum_s i_s+j_s}
\sum_{\sigma\in S_l} (-1)^\sigma
(e_{j_1-i_{\sigma(1)}})_2 (e_{j_2-i_{\sigma(2)}})_2
\dots (e_{j_l-i_{\sigma(l)}})_2 \, .
\]

\subsection{The bicomplex $\tilde r\onen$}
Let us denote by $\tilde r\onen$ the image of $r\onen$
under applying the isomorphism $H_{k,l}$ to each ``chain group'' $C_{k,l}$.
The horizontal and vertical
differentials of $\tilde r\onen$ are given by
\begin{equation}\label{diffVH}
 \tilde d^H_{k,l}:= H^{-1}_{k,l+1} d^H_{k,l} H_{k,l}\quad\quad
\tilde d^V_{k,l}= H^{-1}_{k+1,l} d^V_{k,l} H_{k,l}
\, .\end{equation}
Let us compute them.

We start with $\tilde d^H_{k,0}= H^{-1}_{k,1} d^H_{k,0}=
 H^{-1}_{k,1}\sum^k_{i=1}c_i w_i $.
Thus,  the $i$th entry of $\tilde d^H_{k,0}$  is
$$\sum^{i-1}_{j=0}(-1)^{i-j} (e_{i-j})_2 c_j= (c_i)_1 -(-1)^i (e_i)_2=
(-1)^i((e_i)_1-(e_i)_2  )=: y_i \in \Dot(\F^{(k)}\E^{(k)}\onen) \, .
$$
In the general case,
\begin{align*}
\tilde d^H_{k,l} (w_{i_1}\wedge\dots\wedge w_{i_l})&=
 (\Lambda^{l+1} H_{k,1})^{-1}\left(c\wedge (\Lambda^l H_{k,1})
(w_{i_1}\wedge \dots\wedge w_{i_l})\right)\\
&=(H_{k,1}^{-1} c)\wedge w_{i_1}\wedge \dots\wedge w_{i_l}\\
&= \sum^{k}_{i=1} y_i w_i \wedge w_{i_1}\wedge \dots\wedge w_{i_l}
%=y  \wedge w_{i_1}\wedge \dots\wedge w_{i_l}
\end{align*}
where we used the  centrality of $c$ and the previous computation.
Observe that $\tilde d^H_{k,l}=\sigma\omega(\tilde d^H_{k,l})$.

To compute the vertical differential we
will need the following notation:
\[t_p:=
\xy
 (-2,-6);(-2,6); **[black][|(2)]\dir{-} ?(1)*[black][|(1)]\dir{>}?;
 (6,-6);(6,6); **[black][|(2)]\dir{-} ?(0)*[black][|(1)]\dir{<}?;
 (6,2)*{};(-2,3)*{} **[black][|(1)]\crv{(6,-2) & (-2,-2)};
 ?(.5)*{\bullet}+(2,-2)*{\scs p}; 
 (-6,7)*{\scriptstyle k+1};
 (10,7)*{\scriptstyle k+1};
 (-5,-7)*{\scriptstyle k};
 (8,-7)*{\scriptstyle k};
% (-2,-2)*{\bigb{\alpha_{k,l}}};
% (12,-3)*{n};
\endxy\;
\]
Then it is easy to verify that $(e_j)_2 t_0=t_0 (e_j)_2 +t_1 (e_{j-1})_2$
 or more generally,
\begin{equation}\label{e-tcomut}
(e_{m_1 }\dots e_{m_l})_2 \, t_0=
\sum^l_{a=0}\;\; t_a\; \sum_{
\xy (0,2)*{\scs T\subset\{1,2,\dots, l\}}; (0,-1)*{\scs |T|=a};
\endxy}
(e_{m_1-\varepsilon_1(T)} \dots e_{m_l-\varepsilon_l(T)})_2
\end{equation}
where  $(xy)_2=(x)_2(y)_2$ and
$$\varepsilon_p(T)=\left\{\begin{array}{cc}
1&\quad p\in T\\
0&\quad p\notin T\end{array}\right. \, .
$$

For instance,
\begin{align*}
 (e_{m_1}e_{m_2}e_{m_3})_2 t_0 & =t_0  (e_{m_1}e_{m_2}e_{m_3})_2 + t_1\left(
 (e_{m_1-1}e_{m_2}e_{m_3})_2 +  (e_{m_1}e_{m_2-1}e_{m_3})_2 +  (e_{m_1}e_{m_2}e_{m_3-1})_2\right)\\
&+t_2 \left(  (e_{m_1-1}e_{m_2-1}e_{m_3})_2+  (e_{m_1-1}e_{m_2}e_{m_3-1})_2 +
 (e_{m_1}e_{m_2-1}e_{m_3-1})_2   
\right)\\ &  +t_3  (e_{m_1-1}e_{m_2-1}e_{m_3-1})_2 
\end{align*}

For simplicity, we put $t_0=t$. With this notation, let us first compute the
matrix for the vertical differential in the case when $l=1$.
Inserting \eqref{Hk1}, \eqref{H-1} into \eqref{diffVH}, we get
$$ (d^V_{k,1})_{i,j}=\left\{
\begin{array}{ll}
-\sum^{i-j}_{p=0}(-1)^p (e_p)_2 t (h_{i-j-p})_2=-t \delta_{i-j,0}+t_1\delta_{i-j,1} & \quad i-j\geq 0\\
0 &
\quad {\rm otherwise} \end{array}\right. \, 
$$
where $\delta_{i,j}$ is the Kronecker delta-function.
For example, for $k=3$ we have
\[ d^V_{k,1}=
\left( \begin{array} {rrrr}
{\bf 1} & 0 & 0& 0 \\
-(e_1)_2 & {\bf 1} & 0& 0\\
(e_2)_2 & -(e_1)_2 & {\bf 1}& 0 \\
-(e_3)_2 & (e_2)_2 & -(e_1)_2& {\bf 1}
\end{array}\right)\; (-t)
\left( \begin{array} {rrr}
{\bf 1} & 0 & 0 \\
0 & {\bf 1} & 0 \\
0 & 0 & {\bf 1} \\
(e_{3})_2& -(e_{2})_2 & (e_{1})_2
\end{array}\right)
\left( \begin{array} {rrr}
{\bf 1} & 0 & 0\\
(h_1)_2 & {\bf 1} & 0\\
(h_2)_2 & (h_1)_2 & {\bf 1}
\end{array}\right)
\]

In general, 
for $1\leq i_1<i_2 \dots<i_l\leq k$ and $1\leq j_1<j_2 \dots<j_l\leq k$ 
using the computation for $l=1$ case we get
\begin{align*}
\left(d^V_{k,l}\right)^{i_1,\dots ,i_l}_{j_1, \dots , j_l}&= (-1)^l
\sum_{1\leq p_1<\dots<p_l\leq k} \sum_{1\leq s_1<\dots<s_l\leq k+1}
 (H^{-1}_{k+1,l})^{s_1,\dots ,s_l}_{j_1, \dots , j_l}\;\; t\;\;(\alpha_{k,l})^{p_1, \dots , p_l}_{s_1, \dots , s_l}
(H_{k,l})^{i_1,\dots ,i_l}_{p_1, \dots , p_l}\\
& =\sum_{1\leq p_1<\dots<p_l\leq k+1}
 (H^{-1}_{k+1,l})^{p_1,\dots ,p_l}_{j_1, \dots , j_l}\;\; t\;\;
(H_{k+1,l})^{i_1,\dots ,i_l}_{p_1, \dots , p_l}
\end{align*}
By \eqref{e-tcomut}, we obtain
\begin{align*}
(H^{-1}_{k+1,l})^{p_1,\dots ,p_l}_{j_1, \dots , j_l}\; t &
= 
\sum^l_{a=0} t_a\; (-1)^{\sum_s (p_s-j_s) }
\sum_{\sigma\in S_l} (-1)^\sigma
\sum_{
\xy (0,2)*{\scs T\subset\{1,2,\dots, l\}}; (0,-1)*{\scs |T|=a};
\endxy} \prod^l_{s=1}
e_{j_s- k_{\sigma(s)}-\varepsilon_s(T)} \\
&= \sum^l_{a=0} (-1)^a t_a 
\sum_{
\xy (0,2)*{\scs T\subset\{1,2,\dots, l\}}; (0,-1)*{\scs |T|=a};
\endxy} 
(H^{-1}_{k+1,l})^{p_1,\dots ,p_l}_{j^T_1, \dots , j^T_l}\; 
\end{align*}
where $j^T_s:=j_s - \varepsilon_s(T)$. 
Hence,
\begin{align*}
\left(d^V_{k,l}\right)^{i_1,\dots ,i_l}_{j_1, \dots , j_l}=
& (-1)^l \sum^l_{a=0} (-1)^a t_a 
\sum_{
\xy (0,2)*{\scs T\subset\{1,2,\dots, l\}}; (0,-1)*{\scs |T|=a};
\endxy} \prod^l_{s=1} 
\delta_{j^T_s, i_s}
\\ &= \left\{ \begin{array}{cc}
(-1)^{l+a} t_a & {\rm if}\; (j_1, \dots, j_l)\; {\rm is\; obtained\; from }
\;(i_1,\dots, i_l)\; {\rm by\; shifting}\; a \;{\rm entries\; by}\; -1
\\ 0 & {\rm otherwise}\end{array}
\right.
\end{align*}
Observe that $\sigma\omega (t_a)=t_a$ and hence
$d^V_{k,l}=\sigma\omega (d^V_{k,l})$.

\section{Proof of the invertibility of the ribbon complex}
\def\1{\mathbbm{1}}%

\subsection{The 2-functor $\Gamma_N$}
%
% ===============================================================

Let us recall the definition of the 2-functor $\Gamma_N$ from \cite[Section 7]{Lau1}. On objects the 2-functor $\Gamma_{N}$ sends
$n$ to the ring $H_{k;N}$ whenever $n$ and $k$ are compatible:
\begin{eqnarray}
 \Gamma_{N} \maps \Ucat & \to & \Gr \nn \\
 n & \mapsto &
  \left\{\begin{array}{ccl}
    H_{k;N} & & \text{with $n=2k-N$ and $0\leq k \leq N$,} \\
    0  & & \text{otherwise.}
  \end{array} \right.
\end{eqnarray}
1-Morphisms of $\Ucat$ get mapped by $\Gamma_N$ to graded bimodules
\begin{eqnarray} \label{eq_GammaN_def}
 \Gamma_{N} \maps \Ucat & \to & \Gr  \\
  \onen\la s\ra & \mapsto &
  \left\{\begin{array}{ccl}
    H_{k;N}\la s\ra & & \text{with $n=2k-N$ and $0\leq k \leq N$,} \\
    0  & & \text{otherwise.}
  \end{array} \right.  \nn \\
  \cal{E}\onen\la s\ra & \mapsto &
  \left\{\begin{array}{ccl}
    H_{k+1,k;N}\la s+1-N+k\ra & & \text{with $n=2k-N$ and $0\leq k < N$,} \\
    0  & & \text{otherwise.}
  \end{array} \right. \nn \\
  \cal{F}\onen\la s\ra & \mapsto &
  \left\{\begin{array}{ccl}
    H_{k-1,k;N}\la s+1-k\ra & & \text{with $n=2k-N$ and $0< k \leq N$,} \\
    0  & & \text{otherwise}
  \end{array} \right. \nn
\end{eqnarray}
where
   the cohomology of the Grassmannian is given by
  \begin{equation} \nn
H_{k;N}:=
\Z[x_{1,n}, x_{2,n},\dots, x_{k,n};
y_{1,{n}}, \dots, y_{N-k, n}]/ I_{k;N}
\end{equation}
with $I_{k;N}$ the homogeneous ideal generated by elements equating powers of $t$ in the equation
\begin{equation} \nn
  (1+x_{1,n}t+\dots+x_{k,n}t^k)(1+y_{1,n}t+
\dots+y_{N-k,n} t^{N-k})=1.
\end{equation}

The cohomology of the $a$th iterated 1-step flag variety $H_{\underline k;N}$ with $\underline k=(k,k+1,k+2,\dots, k+a)$ is given by
\begin{equation} \nn
H_{\underline k;N}:=
\Z[x_{1,n}, x_{2,n},\dots, x_{k,n};\xi_1,\dots, \xi_a;
y_{1,{n+2a}}, \dots, y_{N-k-a, n+2a}]/ I_{\underline k;N}
\end{equation}
with $I_{\underline k;N}$ the homogeneous ideal generated by elements equating powers of $t$ in the equation
\begin{equation} \nn
  (1+x_{1,n}t+\dots+x_{k,n}t^k)(1+\xi_1 t)(1+\xi_2 t)\dots (1+\xi_a t)(1+y_{1,n+2a}t+
\dots+y_{N-k-a,n+2a} t^{N-k-a})=1.
\end{equation}

We will also use the cohomology of $a$-step flag variety corresponding to the sequence $\underline k=(k,k+a)$  given by
\begin{equation} \nn
H_{k,k+a;N}:=
\Z[x_{1,n}, x_{2,n},\dots, x_{k,n};
\varepsilon_1,\dots, \varepsilon_a;
y_{1,{n+2a}}, \dots, y_{N-k-a, n+2a}]/ I_{k,k+a;N}
\end{equation} \nn
with $I_{k,k+a;N}$ the homogeneous ideal generated by elements equating powers of $t$ in the equation
\begin{equation} \nn
  (1+x_{1,n}t+\dots+x_{k,n}t^k)(1+\varepsilon_1 t+\dots +\varepsilon_a t^a)(1+y_{1,n+2a}t+
\dots+y_{N-k-a,n+2a} t^{N-k-a})=1.
\end{equation}

It will be convenient in what follows to introduce a simplified notation in $\Gr$.
Corresponding to a fixed value of $N$, we set $n=2k-N$ and write
\begin{align} \label{eq_shorthand}
  \1_n^N &:= \Gamma_N(\onen) \nn \\
  \rE \1_n^N &= \1_{n+2}^N\rE = \1_{n+2}^N\rE\1_n^N := \Gamma_{N}(\cal{E}\onen) \nn \\
  \rF \1_n^N &= \1_{n-2}^N\rF = \1_{n-2}^N\rF\1_n^N := \Gamma_N(\cal{F}\onen)
\end{align}
as a shorthand for the various bimodules.  Juxtaposition of these symbols represents the tensor product of the corresponding bimodules. For example,
\[
\rF \rE\rE\1_n^N = H_{k+1,k+2;N} \otimes_{H_{k+2;N}} H_{k+2,k+1;N} \otimes_{H_{k+1;N}} H_{k+1,k;N}.
\]
Associated to a signed sequence $\ep$  is the $(H_{k+|\ep|},H_k)$-bimodule
\[
\rE_{\ep}\1_n^N := \rE_{\epsilon_1} \rE_{\epsilon_2}\dots \rE_{\epsilon_m}\1_n^N
\]
where $\rE_{+}:= \rE$ and $\rE_{-}:= \rF$. The 2-functor $\Gamma_N$ maps a composite $\cal{E}_{\ep}\onen$ of 1-morphisms in $\UcatD$ to the tensor product $\rE_{\ep}1_n^N$ in $\Gr$. Note that because tensor product of bimodules is only associative up to coherent isomorphism our notation is ambiguous unless we choose a parenthesization of the bimodules in question.  We employ the convention that all parenthesis are on the far left.  Hence, $\Gamma_N$ preserves composition of 1-morphisms only up to coherent 2-isomorphism.

It is sometimes convenient to use the following isomorphisms  from \cite{BL}.
\begin{align}
  \Gamma_N(\cal{E}^a\onen) &\cong H_{k+a,k+a-1, \dots, k;N} \la r_a\ra,
  \qquad r_a = \sum_{i=1}^{a}  i-N+k \\
  \Gamma_N(\cal{F}^a\onen) &\cong H_{k,k+1, \dots, k+a;N} \la r'_a \ra ,
  \qquad r'_a =\sum_{i=1}^{a} i-k.
\end{align}
We also define bimodules
\begin{align}
  \rE^{(a)}\1_n^N &:=H_{k+a,k;N}\la r_a -\frac{a(a-1)}{2}\ra, \\
  \rF^{(b)}\1_n^N &:=H_{k+a,k;N}\la r_a +\frac{a(a-1)}{2}\ra.
\end{align}
%For $x \in \B$ write $\rE(x)\1_n^N$ as in \eqref{eq_basis} for the corresponding tensor product of bimodules.

Let us denote by
$\rr\1^N_n$ the image under $\Gamma_N$ of the ribbon complex.

\subsection{Proof of Theorem \ref{flag}}
We first prove that tensoring on the left (or on the right)
with $\rr\1^N_n$ acts as a left (resp. right) multiplication with the identity 
on any left (resp. right) $H_{k;N}$-module up to degree shift.

We will consider the left action only, the right action can be proved similarly.
Note that it is enough to compute the left action of $\rr\1^N_n$  on $H_{k;N}$.
Let us first ignore the homological shift for simplicity.

Set $n=N$. Then $\rr\1^N_N=\1^N_N\la -N^2/2-N\ra$
in $\Kom(\Gr)$, simply because
$\rE\1^N_N=0$, and the result holds.

Assume $n=N-2k$, then
$$\rr\rF^{k}\1^N_N\simeq \rF^{k}\rr\1^N_N= \rF^{k}\1^N_N \la -N^2/2-N\ra$$
where the first homotopy equivalence holds due to centrality of $\rr\1^N_n$.
Now observe that
$$\rF^{k}\1^N_N\simeq\bigoplus_{[k]!} \rF^{(k)}\1^N_N$$
in $\Gr$. Hence, $\rr\1^N_n$ acts as a left multiplication with
$\1^N_n\la -N^2/2-N\ra$ on $\rF^{(k)}\1^N_N$ in $\Com(\Gr)$. But
$\rF^{(k)}\1^N_N$ is isomorphic to $H_{k;N}=\1^N_n$ as a left module over itself.
Hence we have the first statement.

Similarly, tensoring with $\rr^{-1}\1^N_n$ on the left and on the right
is homotopic to the identity functor shifted by $\la N^2/2+N\ra$,
which is inverse to $\rr\1^N_n$. Since homological shifts for $r\onen$ and
$r^{-1}\onen$ are inverse to each other, we have the result.

\subsection{The inverse limit of Schur quotients}
The Schur quotient $\UcatD_N$ of $\UcatD$ is defined by setting 
${\bf 1}_{N+2}=0$ (see \cite{MSV} for a more general definition).
Applying sl(2) relations, an easy induction argument shows 
that $\onen=0$ in $\UcatD_N$ for all $n<-N$ and $n>N$.
Moreover, this quotient is not empty, since
the functor
$\Gamma_N: \UcatD \to \Gr$ factorises through $\UcatD_N$ by its very 
definition.

For any $N'>N$ there is a natural projection
$$\tilde\Psi_{N',N}: \UcatD_{N'}\to \UcatD_{N}$$
defined by setting ${\bf 1}_{N+2}=0$. Taking all together, these maps define
an inverse system of 2-categories whose inverse limit is $\UcatD$
(compare \cite{BL}).

Now let us consider $\Com^b(\UcatD_N)$. The induced functor 
 $\Com^b(\UcatD)\to \Com^b(\Gr)$ factorises again through $\Com^b(\UcatD_N)$.
Using the natural projections
$$\Psi_{N',N}: \Com^b(\UcatD_{N'})\to \Com^b(\UcatD_{N})$$
for all $N'>N$ we can construct
$$^l\Com(\UcatD)=\varprojlim \Com^b (\UcatD_N)\, .$$
Observe that $r^{\pm 1}r^{\mp 1}\onen $ belongs to $^l\Com(\UcatD)$,
since its  projections to $\Com^b(\UcatD_N)$ are well-defined for all $N$
and compatible with each other.

\subsection{Proof of Theorem \ref{main} (Invertibility) }
%Using the fact that $\Gamma_N:\UcatD\to\Gr$ is locally full and 
%factorizes through $\UcatD_N$,
%it is easy
%to see that $\UcatD_N$ is well defined and non empty.
%By the same reasons, the construction of $^l\Kom(\UcatD)$ fits well
%with the definition of the 2-limit in $\Kom^b(\Gr)$ given in \cite{BL}. 

 By the universal property of the inverse 2-limit we get
 a 2-functor $\hat{\Gamma} \maps ^l\Com(\UcatD) \to \varprojlim 
\Com^b(\Grn{})$, 
unique up to 2-isomorphism in \cat{iBicat}, making the diagram
\[
 \xy
  (0,30)*+{\varprojlim\Com^b(\UcatD_N)}="tt";
  (0,15)*+{\varprojlim \Com^b(\Gr)}="t";
  (-20,0)*+{\Com^b(\Grn{N+2})}="bl";
  (20,0)*+{\Com^b(\Gr)}="br";
  {\ar_{\pi_{N+2}} "t";"bl"};
  {\ar^{\pi_{N}} "t";"br"};
  {\ar_{\Psi_{N+2,N}} "bl";"br"};
  {\ar^{\hat{\Gamma}} "tt";"t"};
  {\ar@/^1pc/^{\Gamma_N} "tt";"br"};
  {\ar@/_1pc/_{\Gamma_N} "tt";"bl"};
 \endxy
\]
commute up to 2-isomorphism in \cat{iBicat}. Note that \cat{iBicat}
is a bicategory with objects bicategories, morphisms  (pseudo) 2-functors
and 2-morphisms given by icons.
Now the arguments  in \cite{BL} (e.g. proof of Theorem 3.2) 
imply that this
  2-functor $\hat{\Gamma} \maps ^l\Com(\UcatD) \to \varprojlim \Com^b(\Grn{})$
is an equivalence of 2-categories in \cat{iBicat}. Since
$r r^{-1}\onen = \hat{\Gamma}^{-1} (\varprojlim \1^N_n)$ by Theorem \ref{flag},
% and the definition of the inverse limit, 
we get the desired result. $\hfill\Box$

\section{Appendix}
Let us collect the identities we need in the proofs.

\subsection{Sliding rules}
Generalizing the NilHecke algebra relations to the thick lines
 we get:
\begin{equation}
\xy
(-4,-6);(4,6); **[black][|(2)]\dir{-}; ?(0)*[black][|(1)]\dir{<}
?(.8)*{\bullet} +(-3,1)*{\scs h_n};
(-4,6);(4,-6); **[black][|(2)]\dir{-} ?(1)*[black][|(1)]\dir{>}?;
 (-6,6)*{\scriptstyle k}; 
 (6,6)*{\scriptstyle l}; 
\endxy\;\;=\;\;\xy
(-4,-6);(4,6); **[black][|(2)]\dir{-}; ?(0)*[black][|(1)]\dir{<}
?(.2)*{\bullet} +(2,-2)*{\scs h_n};
(-4,6);(4,-6); **[black][|(2)]\dir{-} ?(1)*[black][|(1)]\dir{>}?;
 (-6,6)*{\scriptstyle k}; 
 (6,6)*{\scriptstyle l}; 
\endxy\;\;+\;\;\sum_{i+j+k=n-1}
\xy
(-6,-6);(6,6); **[black][|(2)]\dir{-}; ?(0)*[black][|(1)]\dir{<}
?(.3)*{\bullet} +(-2,2)*{\scs h_i};
(-6,6);(6,-6); **[black][|(2)]\dir{-} ?(1)*[black][|(1)]\dir{>}?;
% (-8,6)*{\scriptstyle k}; 
% (1,6)*{\scriptstyle l}; 
(-4,-4);(-4,4); **[black][|(1)]\dir{-}; 
?(.5)*{\bullet} +(-3,1)*{\scs h_j};
(4,-4);(4,4); **[black][|(1)]\dir{-}; 
?(.5)*{\bullet} +(3,1)*{\scs h_k};
\endxy
\end{equation}

\begin{equation}
\xy
(-4,-6);(4,6); **[black][|(2)]\dir{-}; ?(0)*[black][|(1)]\dir{<}
?(.8)*{\bullet} +(-3,1)*{\scs e_n};
(-4,6);(4,-6); **[black][|(2)]\dir{-} ?(1)*[black][|(1)]\dir{>}?;
 (-6,6)*{\scriptstyle l}; 
 (6,6)*{\scriptstyle k}; 
\endxy\;\;=\;\;\xy
(-4,-6);(4,6); **[black][|(2)]\dir{-}; ?(0)*[black][|(1)]\dir{<}
?(.2)*{\bullet} +(2,-2)*{\scs e_n};
(-4,6);(4,-6); **[black][|(2)]\dir{-} ?(1)*[black][|(1)]\dir{>}?;
 (-6,6)*{\scriptstyle l}; 
 (6,6)*{\scriptstyle k}; 
\endxy\;\;+\;\;
\xy
(-6,-6);(6,6); **[black][|(2)]\dir{-}; ?(0)*[black][|(1)]\dir{<}
?(.3)*{\bullet}+(-2,3)*{\scs e_{n-1}};
(-6,6);(6,-6); **[black][|(2)]\dir{-} ?(1)*[black][|(1)]\dir{>}?;
% (-8,6)*{\scriptstyle k}; 
% (1,6)*{\scriptstyle l}; 
(-4,-4);(-4,4); **[black][|(1)]\dir{-}; 
%?(.5)*{\bullet} +(-3,1)*{\scs h_j};
(4,-4);(4,4); **[black][|(1)]\dir{-}; 
%?(.5)*{\bullet} +(3,1)*{\scs h_k};
\endxy
\end{equation}

In particular, if $l=1$ we have
\begin{equation}\label{hn-slide}
\xy
(-4,-6);(4,6); **[black][|(1)]\dir{-}; ?(0)*[black][|(1)]\dir{<}
?(.8)*{\bullet} +(-3,1)*{\scs h_n};
(-4,6);(4,-6); **[black][|(2)]\dir{-} ?(1)*[black][|(1)]\dir{>}?;
 (-6,6)*{\scriptstyle k}; 
% (6,6)*{\scriptstyle l}; 
\endxy\;\;=\;\;\xy
(-4,-6);(4,6); **[black][|(1)]\dir{-}; ?(0)*[black][|(1)]\dir{<}
?(.2)*{\bullet} +(2,-2)*{\scs h_n};
(-4,6);(4,-6); **[black][|(2)]\dir{-} ?(1)*[black][|(1)]\dir{>}?;
 (-6,6)*{\scriptstyle k}; 
% (6,6)*{\scriptstyle l}; 
\endxy\;\;+\;\;\sum_{i+j=n-1}
\xy
%(-6,-6);(6,6); **[black][|(2)]\dir{-}; ?(0)*[black][|(1)]\dir{<}
%?(.3)*{\bullet} +(-2,2)*{\scs h_i};
(-5,6);(5,-6); **[black][|(2)]\dir{-} ?(1)*[black][|(1)]\dir{>};
%?(.3)*{\bullet} +(-2,2)*{\scs h_i};
% (-8,6)*{\scriptstyle k}; 
% (1,6)*{\scriptstyle l}; 
(-5,-6);(-4,5); **[black][|(1)]\crv{(-1,0)}; 
?(.5)*{\bullet} +(-3,1)*{\scs h_i};
(4,-5);(5,6); **[black][|(1)]\crv{(1,0)};  ?(0)*\dir{<}
?(.5)*{\bullet} +(3,1)*{\scs h_j};
\endxy
\end{equation}

\begin{equation}\label{en-slide}
\xy
(-4,-6);(4,6); **[black][|(2)]\dir{-}; ?(0)*[black][|(1)]\dir{<}
?(.8)*{\bullet} +(-3,1)*{\scs e_n};
(-4,6);(4,-6); **[black][|(1)]\dir{-} ?(1)*[black][|(1)]\dir{>}?;
% (-6,6)*{\scriptstyle k}; 
 (6,6)*{\scriptstyle k}; 
\endxy\;\;=\;\;\xy
(-4,-6);(4,6); **[black][|(2)]\dir{-}; ?(0)*[black][|(1)]\dir{<}
?(.2)*{\bullet} +(2,-2)*{\scs e_n};
(-4,6);(4,-6); **[black][|(1)]\dir{-} ?(1)*[black][|(1)]\dir{>}?;
% (-6,6)*{\scriptstyle k}; 
 (6,6)*{\scriptstyle k}; 
\endxy\;\;+\;\;
\xy
(-5,-6);(5,6); **[black][|(2)]\dir{-}; ?(0)*[black][|(1)]\dir{<}
?(.4)*{\bullet}+(-1,3)*{\scs e_{n-1}};
%(-5,6);(5,-6); **[black][|(2)]\dir{-} ?(1)*[black][|(1)]\dir{>}?;
% (-8,6)*{\scriptstyle k}; 
% (1,6)*{\scriptstyle l}; 
(-4,-5);(-5,6); **[black][|(1)]\crv{(-1,0)} ?(0)*\dir{<};  
%?(.5)*{\bullet} +(-3,1)*{\scs h_i};
(4,-6);(4,5); **[black][|(1)]\crv{(1,0)};  
%?(.6)*{\bullet} +(3,1)*{\scs h_j};
\endxy
\end{equation}

In what follows  an $x$ labeled bullet on a thick line
will mean  $h_x$ inserted.

\subsection{Reidemeister moves}
From Corollary 5.8 in \cite{KLMS} (for $b=k$, $a=1$)
and the bubble slide rule (4.11), we get
\begin{equation}\label{R2}
\xy
 (-8,8);(-8,-8); **[black][|(2)]\dir{-} ?(0)*[black][|(1)]\dir{<}?;
 (-4,-8);(-4,8); **[black][|(1)]\dir{-} ?(0)*[black][|(1)]\dir{<}?;
 (-10,8)*{\scriptstyle k};
% (10,-9)*{\scriptstyle k+1};
% (-5,9)*{\scriptstyle k};
% (8,9)*{\scriptstyle k};
% (2,3)*{\bigb{\;\;\alpha_{k,l}\;\;}};
\endxy
= (-1)^k
\xy
 (-8,8)*{};(-8,-8)*{} **[black][|(2)]\crv{(0,0)} ?(1)*[black][|(1)]\dir{>}?;
 (-3,8)*{};(-3,-8)*{} **[black][|(1)]\crv{(-11,0)} ?(0)*[black][|(1)]\dir{<}?;
 (-10,8)*{\scriptstyle k};
\endxy\quad  + 
(-1)^{k-1} \sum_{x+u+s+t=k+n-2}
\xy 0;/r.17pc/: 
%(-20,12)*{};(-20,-12)*{} **\dir{-}?(0)*\dir{<};
%  (-20,4)*{\bullet}+(-3,2)*{\scs v};
(-13,0)*{\lbub{\scs x}};  
(-4,12);(-4,-12); **[black][|(2)]\dir{-} ?(1)*[black][|(1)]\dir{}?;
    (-4,0)*{\bullet}+(3,2)*{\scs u};
(-4,12)*{}="t1";
  (4,12)*{}="t2";
  "t1";"t2" **\crv{(-4,5) & (4,5)};  ?(1)*\dir{>}
  ?(0.5)*{\bullet}+(3,-2)*{\scs \;\; s};
  (-4,-12)*{}="t1";
  (4,-12)*{}="t2";
  "t2";"t1" **\crv{(4,-5) & (-4,-5)};  ?(1)*\dir{>}
  ?(.5)*{\bullet}+(3,2)*{\scs t};
\endxy
\end{equation}

Analogously, Theorem 5.6 and (4.12) in \cite{KLMS} imply
\[
\xy
 (-8,8);(-8,-8); **[black][|(2)]\dir{-} ?(1)*[black][|(1)]\dir{>}?;
 (-4,-8);(-4,8); **[black][|(1)]\dir{-} ?(1)*[black][|(1)]\dir{>}?;
 (-10,8)*{\scriptstyle k};
% (10,-9)*{\scriptstyle k+1};
% (-5,9)*{\scriptstyle k};
% (8,9)*{\scriptstyle k};
% (2,3)*{\bigb{\;\;\alpha_{k,l}\;\;}};
\endxy
= (-1)^k
\xy
 (-8,8)*{};(-8,-8)*{} **[black][|(2)]\crv{(0,0)} ?(0)*[black][|(1)]\dir{<}?;
 (-3,8)*{};(-3,-8)*{} **[black][|(1)]\crv{(-11,0)} ?(1)*[black][|(1)]\dir{>}?;
 (-10,8)*{\scriptstyle k};
\endxy\quad  + (-1)^{k-1} \sum_{x+u+s+t=k+n-2}
\xy 0;/r.17pc/: 
%(-20,12)*{};(-20,-12)*{} **\dir{-}?(0)*\dir{<};
%  (-20,4)*{\bullet}+(-3,2)*{\scs v};
(-13,0)*{\rbub{\scs x}};  
(-4,12);(-4,-12); **[black][|(2)]\dir{-} ?(1)*[black][|(1)]\dir{>}?;
    (-4,0)*{\bullet}+(3,2)*{\scs u};
(-4,12)*{}="t1";
  (4,12)*{}="t2";
  "t1";"t2" **\crv{(-4,5) & (4,5)};  ?(0.8)*\dir{<}
  ?(0.5)*{\bullet}+(3,-2)*{\scs \;\; s};
  (-4,-12)*{}="t1";
  (4,-12)*{}="t2";
  "t2";"t1" **\crv{(4,-5) & (-4,-5)};  ?(0)*\dir{<}
  ?(.5)*{\bullet}+(3,2)*{\scs t};
\endxy
\]

Similar, to the proof of Proposition 5.8 in \cite{Lau1} we can show that
for all colors $n \in \Z$ of the right most region
the following equation holds
\begin{equation} \label{R3}
 \vcenter{
 \xy 0;/r.17pc/:
    (-4,-4)*{};(4,4)*{} **[black][|(1)]\crv{(-4,-1) & (4,1)};
    (4,-4)*{};(-4,4)*{} **[black][|(2)]\crv{(4,-1) & (-4,1)}  ?(0)*\dir{<};
    (4,4)*{};(12,12)*{} **[black][|(1)]\crv{(4,7) & (12,9)};
    (12,4)*{};(4,12)*{} **\crv{(12,7) & (4,9)};
    (-4,12)*{};(4,20)*{} **[black][|(2)]\crv{(-4,15) & (4,17)};
    (4,12)*{};(-4,20)*{} **\crv{(4,15) & (-4,17)}?(1)*\dir{>};
    (-4,4)*{}; (-4,12) **[black][|(2)]\dir{-};
    (12,-4)*{}; (12,4) **\dir{-};
    (12,12)*{}; (12,20) **[black][|(1)]\dir{-}?(1)*\dir{>};;
  (6,20)*{\scriptstyle k};
\endxy}
-\;
   \vcenter{
 \xy 0;/r.17pc/:
    (4,-4)*{};(-4,4)*{} **[black][|(1)]\crv{(4,-1) & (-4,1)};
    (-4,-4)*{};(4,4)*{} **[black][|(2)]\crv{(-4,-1) & (4,1)}?(0)*\dir{<};;
    (-4,4)*{};(-12,12)*{} **[black][|(1)]\crv{(-4,7) & (-12,9)};
    (-12,4)*{};(-4,12)*{} **\crv{(-12,7) & (-4,9)};
    (4,12)*{};(-4,20)*{} **[black][|(2)]\crv{(4,15) & (-4,17)};
    (-4,12)*{};(4,20)*{} **\crv{(-4,15) & (4,17)}?(1)*\dir{>};
    (4,4)*{}; (4,12) **[black][|(2)]\dir{-};
    (-12,-4)*{}; (-12,4) **\dir{-};
    (-12,12)*{}; (-12,20) **[black][|(1)]\dir{-}?(1)*\dir{>};;
   (-2,20)*{\scriptstyle k};
\endxy}
  \; = \;
\sum_{\xy (0,2)*{\scs s+t+x+u+v}; (0,-1)*{\scs =n-k+1}\endxy} \; \xy 0;/r.17pc/:
    (-4,12)*{}="t1";
    (4,12)*{}="t2";
  "t2";"t1" **\crv{(5,5) & (-5,5)}; ?(.15)*\dir{} ?(1)*\dir{>}
  ?(0.6)*{\bullet}+(-2,-2)*{\scs s};
(4,12);(4,-12); **[black][|(2)]\dir{-} ?(1)*[black][|(1)]\dir{}?;
    (-4,-12)*{}="t1";
    (4,-12)*{}="t2";
  "t2";"t1" **\crv{(5,-5) & (-5,-5)};  ?(0)*\dir{<}
  ?(.6)*{\bullet}+(-2,2)*{\scs t};
    (12,1)*{\rbub{\scs x}{}};
    (4,2)*{\bullet}+(-3,-2)*{\scs u};
    (22,12)*{};(22,-12)*{} **\dir{-} ?(0)*\dir{<};
    (22,4)*{\bullet}+(3,2)*{\scs v};
 % (19,-6)*{n};
  \endxy
+\; (-1)^{k-1}
 \sum_{\xy (0,2)*{\scs s+t+x+u+v}; (0,-1)*{\scs =k-n-3}\endxy} \; 
\xy 0;/r.17pc/: 
(-20,12)*{};(-20,-12)*{} **\dir{-}?(0)*\dir{<};
  (-20,4)*{\bullet}+(-3,2)*{\scs v};
(-13,0)*{\lbub{\scs x}};  
(-4,12);(-4,-12); **[black][|(2)]\dir{-} ?(1)*[black][|(1)]\dir{}?;
    (-4,0)*{\bullet}+(3,2)*{\scs u};
(-4,12)*{}="t1";
  (4,12)*{}="t2";
  "t1";"t2" **\crv{(-4,5) & (4,5)};  ?(1)*\dir{>}
  ?(0.5)*{\bullet}+(3,-2)*{\scs \;\; s};
  (-4,-12)*{}="t1";
  (4,-12)*{}="t2";
  "t2";"t1" **\crv{(4,-5) & (-4,-5)};  ?(1)*\dir{>}
  ?(.5)*{\bullet}+(3,2)*{\scs t};
 % (12,1)*{\cbub{\scs \spadesuit+g}{}};
 % (24,6)*{n};
  \endxy
\end{equation}

\subsection{Further identities}
\begin{lemma}
\[
\xy 
(-6,8);(-6,-9); **[black][|(1)]\dir{-} ?(1)*[black][|(1)]\dir{>}?;
(1,7)*{\scs k};
(-1,8);(-1,-9); **[black][|(2)]\dir{-} ?(0)*[black][|(1)]\dir{<}?;
(4,8);(4,-9); **[black][|(2)]\dir{-} ?(1)*[black][|(1)]\dir{>}?;
(6,7)*{\scs k};
\endxy=
(-1)^{k-1}\;\; \xy
(-3,2);(6,-6); **[black][|(1)]\crv{(-6,-2) } 
?(1)*[black][|(1)]\dir{-}?;
(-6,8);(-3,2); **[black][|(1)]\crv{(-6,4)} ?(1)*[black][|(1)]\dir{>}?;
(2,7)*{\scs k};
(0,7);(6,7);**[black][|(1)]\crv{(3,3)}?(0.2)*[black][|(1)]\dir{<}?;
(0,8);(0,-9); **[black][|(2)]\dir{-} ?(0)*[black][|(1)]\dir{<}?;
(6,8);(6,-9); **[black][|(2)]\dir{-} ?(1)*[black][|(1)]\dir{>}?;
(8,7)*{\scs k};
(-6,-8);(0,-7);**[black][|(1)]\crv{(-3,0)}?(1)*[black][|(1)]\dir{-}?;
(-3,2)*{\bigb{ c_{k-1}}};
\endxy
+\;\; (-1)^k\;\;
\xy
 (-4,8);(-4,-9); **[black][|(2)]\dir{-} ?(0)*[black][|(1)]\dir{<};
 (4,8);(4,-9); **[black][|(2)]\dir{-} ?(1)*[black][|(1)]\dir{>};
%(-6,7)*{\scriptstyle {k}};
% (2,7)*{\scriptstyle {k}};
(-6,-7)*{\scriptstyle {k}};
(2,-7)*{\scriptstyle {k}};
(-11,-9);(4,5)*{} **\crv{(-6,0) & (5,0)}?(.1)*\dir{>};
(-11,8);(4,-5)*{} **\crv{(-6,-2) & (5,-2)}?(0)*\dir{<};
%(4,0)*{\bigb{q_{k,l}}};
(-6,4)*{\bigb{\; c_k\;}};
\endxy
+\;\; \sum^k_{i=1} (-1)^{i-1} 
\xy
 (-4,8);(-4,-9); **[black][|(2)]\dir{-} ?(0.5)*[black][|(1)]\dir{<};
 (4,8);(4,-9); **[black][|(2)]\dir{-} ?(1)*[black][|(1)]\dir{>};
(-6,7)*{\scriptstyle {k}};
 (2,7)*{\scriptstyle {k}};
%
%(2,-7)*{\scriptstyle {k}};
(-11,-9);(4,5)*{} **\crv{(-12,2) & (-3,2)}?(.1)*\dir{>};
(-11,8);(4,-3)*{} **\crv{(-12,-2) & (-3,-2)}?(0)*\dir{<};
(4,0)*{\bullet}+(5,0)*{ e_{k-i}};
(0,-6)*{\bigb{\,\;\; c_i\,\;\;}};
\endxy
\]
\label{lemma1}
\end{lemma}

\begin{proof}
We first simplify the last summand of this identity by using
the fact that $c_i$ is central and
$$\sum^k_{i=1} (-1)^{k-i} \;\;\;\xy
 (-4,-6);(-4,4); **[black][|(1)]\dir{-} ?(0)*[black][|(1)]\dir{<};
 (0,-6);(0,4); **[black][|(2)]\dir{-} ?(1)*[black][|(1)]\dir{>};
 (4,-6);(4,4); **[black][|(2)]\dir{-} ?(0)*\dir{<};
(4,2)*{\bullet};
(9,2)*{ e_{k-i}};
(8,-6)*{\scriptstyle {k-1}};
(-2,-6)*{\scriptstyle {k}};
(0,-2)*{\bigb{\quad c_i\quad}};
\endxy\;\; =\;\;\;\xy
 (-4,-6);(-4,4); **[black][|(1)]\dir{-} ?(0)*[black][|(1)]\dir{<};
 (0,-6);(0,4); **[black][|(2)]\dir{-} ?(1)*[black][|(1)]\dir{>};
 (4,-6);(4,4); **[black][|(2)]\dir{-} ?(0)*\dir{<};
%?(4,2)*{\bullet}+{2,2}*{ e_{k-i}};
(8,-6)*{\scriptstyle {k-1}};
(-2,-6)*{\scriptstyle {k}};
(-2,0)*{\;\bigb{c_k}\;};
\endxy
$$ 
Then we 
apply \eqref{R3} to the second summand on the right hand side.
Further we use
$$\xy
 (-4,-7);(4,4); **[black][|(1)]\crv{(2,-4)};
(4,4);(4,7);**\dir{-}?(0)*\dir{<}; 
(1,-7);(-3,7); **[black][|(1)]\dir{-} ?(0)*[black][|(1)]\dir{<};
 (4,-7);(0,7); **[black][|(2)]\dir{-} ?(1)*[black][|(1)]\dir{>};
(6,-8)*{\scriptstyle {k}};
(-1,3)*{\bigb{\, c_k\,}}
\endxy\;\;-\;\;\xy
 (-4,-7);(4,7); **\crv{(0,6)}?(0)*\dir{<};
%(4,4);(4,7); **\dir{-} ?(0)*\dir{<}; 
(1,-7);(-3,7); **[black][|(1)]\dir{-} ?(0)*[black][|(1)]\dir{<};
 (4,-7);(0,7); **[black][|(2)]\dir{-} ?(1)*[black][|(1)]\dir{>};
(6,-8)*{\scriptstyle {k}};
(2,-3)*{\bigb{\, c_k\,}}
\endxy\;\;=\;\;
\xy
(-4,-7);(-4,7);**\dir{-}?(0)*\dir{<}; 
(4,-7); (0,7); **[black][|(2)]\crv{(4,3)} ?(1)*[black][|(1)]\dir{>};
(1,-7);(4,7); **[black][|(1)]\crv{(1,3)} ?(0)*[black][|(1)]\dir{<};
(6,-8)*{\scriptstyle {k}};
(0,-3)*{\bigb{\; c_{k-1}\;}};
\endxy\;\; -\;\;
\xy
(1,-7);(-3,7); **[black][|(1)]\dir{-} ?(0)*[black][|(1)]\dir{<};
 (4,-7);(0,7); **[black][|(2)]\dir{-} ?(1)*[black][|(1)]\dir{>};
(6,-8)*{\scriptstyle {k}};
(0,0)*{\bigb{ c_{k-1}}};
 (-4,-7);(4,-7); **\crv{(0,-1)};
 (4,7);(0,7); **\crv{(2,2)}?(1)*\dir{>};
\endxy
$$
to simplify the resulting terms. 
The last equality follows from \eqref{hn-slide} and \eqref{en-slide}.
It is easy to check that 
all terms with bubbles sum to zero.
\end{proof}

\begin{lemma} \label{lemma2}
For $1\leq i,j\leq k$ we have
\[
\xy
(0,-8);(6,-8); **[black][|(1)]\crv{(3,-3)} ?;
(-6,8);(0,8); **[black][|(1)]\crv{(-3,4)} ?(1)*[black][|(1)]\dir{>}?;
(2,-8)*{\scs k};
%(0,-4);(6,2);**[black][|(2)]\crv{(3,-3)}?(0.2)*[black][|(1)]\dir{<}?;
(0,9);(0,-9); **[black][|(2)]\dir{-} ?(0)*[black][|(1)]\dir{<}?;
(6,9);(6,-9); **[black][|(2)]\dir{-} ?(1)*[black][|(1)]\dir{>}?;
(8,-8)*{\scs k};
(6,1);(-6,-9);**[black][|(1)]\crv{(-6,-5) }?(.3)*[black][|(1)]\dir{>}?;
(3,4)*{\bigb{\, c_{i-1}\,}};
(6,-3)*{\bullet}+(5,0)*{e_{j-1}};
\endxy \;\;-\;\;
\xy
(-3,2);(6,-6); **[black][|(1)]\crv{(-6,-2) } 
?(1)*[black][|(1)]\dir{-}?;
(-6,8);(-3,2); **[black][|(1)]\crv{(-6,4)} ?(1)*[black][|(1)]\dir{>}?;
(2,-8)*{\scs k};
(0,7);(6,7);**[black][|(1)]\crv{(3,3)};
%?(0.2)*[black][|(1)]\dir{<}?;
(0,8);(0,-9); **[black][|(2)]\dir{-} ?(0)*[black][|(1)]\dir{<}?;
(6,8);(6,-9); **[black][|(2)]\dir{-} ?(1)*[black][|(1)]\dir{>}?;
(8,-8)*{\scs k};
(-6,-8);(0,-7);**[black][|(1)]\crv{(-4,0)}?(1)*[black][|(1)]\dir{-}?;
(0,2)*{\bigb{\;\;\, c_{i-1}\,\;\;}};
(6,-3)*{\bullet}+(5,0)*{e_{j-1}};
\endxy
=
\xy
 (-4,8);(-4,-9); **[black][|(2)]\dir{-} ?(0)*[black][|(1)]\dir{<};
 (4,8);(4,-9); **[black][|(2)]\dir{-} ?(.3)*[black][|(1)]\dir{>};
%(-6,7)*{\scriptstyle {k}};
% (2,7)*{\scriptstyle {k}};
(-6,-8)*{\scriptstyle {k}};
(2,-8)*{\scriptstyle {k}};
(-11,-9);(4,2)*{} **\crv{(-12,0) & (-3,0)}?(.1)*\dir{>};
(-11,8);(4,-6)*{} **\crv{(-12,-4) & (-3,-4)}?(0)*\dir{<};
(4,-2)*{\bullet}+(5,0)*{ e_{j-1}};
(0,5)*{\bigb{\,\;\; c_i\,\;\;}};
\endxy\;\;-\;
\xy
 (-4,8);(-4,-9); **[black][|(2)]\dir{-} ?(0.5)*[black][|(1)]\dir{<};
 (4,8);(4,-9); **[black][|(2)]\dir{-} ?(1)*[black][|(1)]\dir{>};
(-6,7)*{\scriptstyle {k}};
 (2,7)*{\scriptstyle {k}};
%(-6,-7)*{\scriptstyle {k}};
%(2,-7)*{\scriptstyle {k}};
(-11,-9);(4,5)*{} **\crv{(-12,2) & (-3,2)}?(.1)*\dir{>};
(-11,8);(4,-3)*{} **\crv{(-12,-2) & (-3,-2)}?(0)*\dir{<};
(4,1)*{\bullet}+(5,0)*{ e_{j-1}};
(0,-6)*{\bigb{\,\;\; c_i\,\;\;}};
\endxy\;\;+
\;
\xy
 (-4,8);(-4,-9); **[black][|(2)]\dir{-} ?(0)*[black][|(1)]\dir{<};
 (4,8);(4,-9); **[black][|(2)]\dir{-} ?(.5)*[black][|(1)]\dir{>};
%(-6,7)*{\scriptstyle {k}};
% (2,7)*{\scriptstyle {k}};
(-6,-8)*{\scriptstyle {k}};
(2,-8)*{\scriptstyle {k}};
(-11,-9);(4,2)*{} **\crv{(-3,0) & (3,0)}?(.1)*\dir{>};
(-11,8);(4,-3)*{} **\crv{(-3,-2) & (3,-2)}?(0)*\dir{<};
(4,-6)*{\bullet}+(3,0)*{ e_{j}};
(-3,5)*{\bigb{\quad c_{i-1}\quad\,}};
\endxy\;\;-\;
\xy
 (-4,8);(-4,-9); **[black][|(2)]\dir{-} ?(0)*[black][|(1)]\dir{<};
 (4,8);(4,-9); **[black][|(2)]\dir{-} ?(.4)*[black][|(1)]\dir{>};
%(-6,7)*{\scriptstyle {k}};
% (2,7)*{\scriptstyle {k}};
(-6,-8)*{\scriptstyle {k}};
(2,-8)*{\scriptstyle {k}};
(-11,-9);(4,0)*{} **\crv{(-3,-2) & (3,-2)}?(.1)*\dir{>};
(-11,8);(4,-5)*{} **\crv{(-3,-4) & (3,-4)}?(0)*\dir{<};
(4,1)*{\bullet}+(3,0)*{ e_{j}};
(-3,5)*{\bigb{\quad c_{i-1}\,\quad}};
\endxy
\]
\end{lemma}

\begin{proof}
We first apply  \eqref{R3} to the first two summands on the right hand side,
then all diagrams without smoothings will look as shown below.
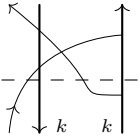
\[\xy
 (-7,8);(-7,-9); **[black][|(2)]\dir{-} ?(0)*[black][|(1)]\dir{<};
 (4,8);(4,-9); **[black][|(2)]\dir{-} ?(1)*[black][|(1)]\dir{>};
(-12,-2);(6,-2); **[black][|(1)]\dir{--};
%(-6,7)*{\scriptstyle {k}};
% (2,7)*{\scriptstyle {k}};
(-4,-8)*{\scriptstyle {k}};
(2,-8)*{\scriptstyle {k}};
(-11,-9);(4,4)*{} **\crv{(-10,4) & (5,4)}?(.1)*\dir{>};
(-11,8);(4,-4)*{} **\crv{(4,-4) & (-4,-4)}?(0)*\dir{<};
%(4,1)*{\bullet}+(3,0)*{ e_{j}};
%(-3,5)*{\bigb{\quad c_{i-1}\,\quad}};
\endxy\]
The strategy of the proof will be to move all
dots into the position shown by the dashed line.
Doing so for the last two summands we get
$$ c_{i-1} (e_1)_3 (e_{j-1})_4 -c_{i-1} (e_1)_1 (e_{j-1})_4
\in \Dot(\E\F^{(k)}\E\E^{(k-1)})$$
minus the second term on the left hand side of the identity
(obtained after sliding the dot though the down pointed $k$-line),
and in addition various terms with bubbles.
Let us first compare the terms without smoothings. The second
term on the right hand side will contribute
$$ -(c_{i})_{234} (e_{j-1})_4 \in \Dot(\E\F^{(k)}\E\E^{(k-1)})\, .$$
Finally, in the first term on the right hand side we replace
$(c_i)_{23}$ with $c_i -c_{i-1}(e_1)_1$ and move the dot down. 
The contribution of this term to the part without smoothings will be
$$c_i (e_{j-1})_4 - c_{i-1}(e_1)_3 (e_{j-1})_4
\in \Dot(\E\F^{(k)}\E\E^{(k-1)})$$
%(c_i)_{124} (e_{j-1})_4$$.
and in addition from moving the dot, we get the first term on the 
left hand side.  Collecting all non-smoothed terms together we get zero. 
It remains to show that all bubble terms vanish.  This easy check is
left to the reader.
\end{proof}

\end{document}